\documentclass{article}

\usepackage[greek, english]{babel}
\usepackage{amsmath}
\usepackage{amssymb}
\usepackage{eurosym}
\usepackage{epsf}
\usepackage{epsfig}
\usepackage{tikz}
\usetikzlibrary{arrows}

\usepackage{color}
\usepackage{hyperref}
\usepackage{tikz}
 \usetikzlibrary{arrows,shapes,positioning,shadows,trees}
\usetikzlibrary{decorations.text}
\usetikzlibrary{decorations.pathmorphing}
\usetikzlibrary{calc}
\usepackage{wrapfig}
\usepackage{multirow}
\usepackage{colortbl}

\newcommand{\bl}[1]{{\color{blue}#1}}
\newcommand{\ro}[1]{{\color{red}#1}}

\newcommand{\pa}[1]{{\color{purple}#1}}

\newfont{\gothic}{eufm10}

                   \def\R{{\RR}} 
\def\RR{{\mathbb{R}}}

        \newtheorem{theorem}{Theorem}[section]
\newtheorem{lemma}[theorem]{Lemma}
\newtheorem{proposition}[theorem]{Proposition}
\newtheorem{corollary}[theorem]{Corollary}
\newtheorem{definition1}[theorem]{Definition}
\newenvironment{definition}{\begin{definition1}\rm}{\hfill $\triangle$ \end{definition1}}
\newenvironment{proof}{\addvspace\baselineskip\noindent{\it
Proof:}\quad}{\hspace*{\fill}         $\Box$\par\addvspace\baselineskip}

\newtheorem{remark1}[theorem]{Remark}
\newenvironment{remark}{\begin{remark1}\rm}{\hfill $\triangle$ \end{remark1}}

\newtheorem{example1}[theorem]{Example}
\newenvironment{example}{\begin{example1}\rm}{\hfill $\triangle$ \end{example1}}

\def\barray{\begin{eqnarray*}}             \def\earray{\end{eqnarray*}}
\def\beq{\begin{equation}} \def\eeq{\end{equation}}

\makeatletter 
\title{Graph fibrations and symmetries of network dynamics}
\author{Eddie Nijholt\thanks{Department of Mathematics, VU University Amsterdam, The Netherlands, {\tt eddie.nijholt@gmail.com}.},  Bob Rink\thanks{Department of Mathematics, VU University Amsterdam, The Netherlands, {\tt b.w.rink@vu.nl}.} \ and Jan Sanders\thanks{Department of Mathematics, VU University Amsterdam, The Netherlands, {\tt jan.sanders.a@gmail.com}.}}
\begin{document}  \hyphenation{boun-da-ry mo-no-dro-my sin-gu-la-ri-ty ma-ni-fold ma-ni-folds re-fe-rence se-cond se-ve-ral dia-go-na-lised con-ti-nuous thres-hold re-sul-ting fi-nite-di-men-sio-nal ap-proxi-ma-tion pro-per-ties ri-go-rous mo-dels mo-no-to-ni-ci-ty pe-ri-o-di-ci-ties mi-ni-mi-zer mi-ni-mi-zers know-ledge ap-proxi-mate pro-per-ty poin-ting ge-ne-ra-li-za-tion ge-ne-ral re-pre-sen-ta-tions equi-variance equi-variant Equi-variance Choo-sing to-po-lo-gy brea-king corres-ponding}
 
\newcommand{\X}{\mathbb{X}}

\newcommand{\p}{\partial}
\maketitle
\noindent 
\abstract{\noindent Dynamical systems with a network structure can display collective behaviour such as synchronisation. Golubitsky and Stewart observed that all the robustly synchronous dynamics of a network is contained in the dynamics of its quotient networks. DeVille and Lerman have recently shown that the original network and its quotients are related by graph fibrations and hence their dynamics are conjugate. This paper demonstrates the importance of self-fibrations of network graphs. Self-fibrations give rise to symmetries in the dynamics of a network. We show that every homogeneous network admits a lift with self-fibrations and that every robust synchrony in this lift is determined by the symmetries of its dynamics. These symmetries moreover impact the global dynamics of network systems and can be used to explain and predict generic scenarios for synchrony breaking. We also discuss networks with interior symmetries and nonhomogeneous networks. }  
               
\section{Introduction}
There are remarkable similarities between dynamical systems with a network structure and dynamical systems with symmetry. It has for example often been noted \cite{field, curious, golstew3, golstew, torok, pikovsky, pivato} that network structure can force a dynamical system to support synchronous and partially synchronous solutions. This phenomenon is known as ``robust network synchrony''. 
Analogously, symmetry forces a dynamical system to admit symmetric solutions. 

It was also observed that network dynamical systems can display unusual bifurcations \cite{bifurcations, elmhirst, krupa, synbreak2, curious, claire2, claire, leite, RinkSanders2}. In fact, there are many examples now of generic one-parameter families of network  systems with anomalous steady state and Hopf bifurcations. These   ``synchrony breaking bifurcations'' are often governed by spectral degeneracies and are reminiscent of the symmetry breaking bifurcations that occur in equivariant dynamics. The latter can often be understood with the help of representation theory and equivariant singularity theory \cite{field3, field4, perspective, golschaef2}, but similar tools are currently not available for the study of network systems. The problem is, arguably, that ``network structure'' is not an intrinsic geometric property of a dynamical system: it is not preserved under coordinate changes.

It was found by Golubitsky and Stewart et al. \cite{golstew, torok, stewartnature, pivato} that the robustly synchronous dynamics of a network system can always be described by a so-called ``quotient network''.  This quotient network arises by identifying the cells of the original network that evolve synchronously. DeVille and Lerman \cite{deville} have recently pointed out that the corresponding quotient map (from the original network graph to its quotient) is an example of a so-called ``graph fibration''. This then implies that there is a conjugacy between the dynamics of the quotient and the dynamics of the original network. The result of DeVille and Lerman provides a geometric explanation for the existence of robust synchrony in networks. We note that very similar results can be found in the computer science literature \cite{boldivigna}.
       
Although robust synchrony is obviously important for the dynamics of networks, its presence does not explain the abundance of anomalous bifurcations in networks. The reason is that robust synchrony does not  affect the global phase space of a network, and is hence not very relevant for the non-synchronous dynamics of network systems. On the other hand, the results of DeVille and Lerman immediately imply that every self-fibration (i.e. graph fibration from a network graph to itself) yields a symmetry in the dynamics of a network. Self-fibrations are therefore dynamically important, and should be thought of as symmetries of network graphs. The self-fibrations of a network in general do not form a group but a semigroup, and self-fibrations may thus not correspond to classical symmetries. 

Network graphs need not admit any nontrivial self-fibrations. We will nevertheless prove in this paper that many networks are quotients of networks with self-fibrations. These network systems are therefore embedded in dynamical systems with symmetries. This can put heavy geometric restrictions on the dynamics of these networks and, in particular, have a nontrivial impact on the singularities that determine the emergence and breaking of synchrony. 
In fact, we prove the following theorem, that will later be formulated more precisely.

\begin{theorem}\label{stellingintro} Consider a network ${\rm\bf N}$ that does not have interchangeable inputs (i.e. all inputs that its cells receive are distinct). Such a network  is the quotient of a network with a semigroup(oid) $\Sigma_{\rm\bf N}$ of self-fibrations. 
The dynamics of this ``lift'' is thus $\Sigma_{\rm\bf N}$-equivariant.

 Moreover, every robust synchrony in the lift (and hence every robust synchrony in the original network) is an invariant subspace for any $\Sigma_{\rm\bf N}$-equivariant dynamical system.
\end{theorem}
Theorem \ref{stellingintro} shows that, for a large class of networks, robust synchrony is a direct consequence of symmetry. This symmetry may nevertheless be hidden in a lift of the network, and it may form a semigroup rather than a group. An even more important conclusion of Theorem \ref{stellingintro} is that network systems are examples of equivariant dynamical systems. This suggests to study networks with techniques that are common in equivariant dynamics (such as representation theory and equivariant singularity theory).

This paper is based on ideas that are present in rudimentary form in our earlier work \cite{RinkSanders3}, but it has been fully formulated in the language of graph fibrations. It moreover demonstrates that hidden symmetry has far-reaching consequences for dynamics. We show for example that ``interior symmetry'' can be viewed as hidden symmetry.

This paper is organised as follows. We start by discussing a few remarkable phenomena in network dynamical systems in Section \ref{examplessection}. We review some existing general theory on coupled cell networks in Sections \ref{networkssection} and \ref{graphfibrationssection}. So-called ``homogeneous'' networks and their symmetry properties are studied in Sections \ref{homnetsection}, \ref{fundamentalsection} and \ref{hiddensection}, and we prove Theorem \ref{stellingintro} for these networks in Section \ref{hiddensection}. In Section \ref{perspectivesection} we discuss the importance of hidden network symmetry, while in Section \ref{methodssection}  we demonstrate how it can be used in the study of local bifuctions. Finally, Section \ref{interiorsection} is concerned with interior symmetry, and in Section \ref{nonhomogeneoussection} we generalise our results for homogeneous networks to nonhomogeneous networks.
\vspace{-2mm}
\subsection*{Acknowledgement} The authors would like to thank Marty Golubitsky for his interest in this topic, and for asking the first questions that eventually led to this paper.
\vspace{-2mm}
\section{Three examples}\label{examplessection}
Figure \ref{pict1} depicts the homogeneous networks {\bf A}, {\bf B} and {\bf C} (see Section \ref{homnetsection} for a definition) that each contain three vertices receiving two different arrows. One could think of these networks as consisting of (groups of) identical neurons, that each receive for instance one excitatory signal (say through the solid blue arrow) and one inhibitory signal (the dashed red arrow). 
 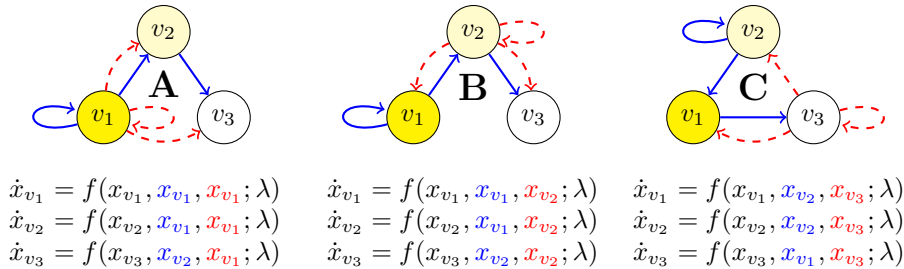
\begin{figure}[h]
 \begin{center}
\renewcommand{\figurename}{\rm \bf \footnotesize Figure}
\vspace{-4mm}
\begin{tabular}{p{4.1cm}p{4.1cm}p{4.1cm}}
\hspace{6mm}
\begin{tikzpicture}[->, scale=1.6]
	  \tikzstyle{vertextype1}=[circle, draw, minimum size=20pt,inner sep=1pt]
	 \tikzstyle{vertextype2} = [circle, draw, fill=yellow, minimum size=20pt,inner sep=1pt]
	 \tikzstyle{vertextype3} = [circle, draw, fill=yellow!25, minimum size=20pt,inner sep=1pt]
	 \tikzstyle{edgetype1} = [->, draw,line width=3pt,-,red!50]
	 \tikzstyle{edgetype2} = [draw,thick,-,black]
	  \node at (1,.95) {};
	 \node at (1,-.25) {};
	  \node at (1.5,.28) {\Large {\bf A}};
	 \node[vertextype2] (v3) at (1,0) {$v_1$};
	 \node[vertextype3] (v2) at (1.5, .71) {$v_2$};
	 \node[vertextype1] (v1) at (2,0) {$v_3$};
	\path[]
	(v2) edge [thick, blue] node {} (v1)
	(v3) edge [loop left, blue, thick] node {} (v3)
	(v3) edge [thick, blue] node {} (v2)
	(v3) edge [loop right, red, dashed, thick] node {} (v3)
	(v3) edge [thick, bend left, red, dashed] node {} (v2)
	(v3) edge [thick, dashed, bend right, red] node {} (v1);
 \end{tikzpicture}
&
\hspace{2mm}
\begin{tikzpicture}[->, scale=1.6]
	  \tikzstyle{vertextype1}=[circle, draw, minimum size=20pt,inner sep=1pt]
	 \tikzstyle{vertextype2} = [circle, draw, fill=yellow, minimum size=20pt,inner sep=1pt]
	 \tikzstyle{vertextype3} = [circle, draw, fill=yellow!25, minimum size=20pt,inner sep=1pt]
	 \tikzstyle{edgetype1} = [->, draw,line width=3pt,-,red!50]
	 \tikzstyle{edgetype2} = [draw,thick,-,black]
	  \node at (1,.95) {};
	 \node at (1,-.25) {};
	  \node at (1.5,.25) {\Large {\bf B}};
	 \node[vertextype2] (v2) at (1,0) {$v_1$};
	 \node[vertextype3] (v3) at (1.5, .71) {$v_2$};
	 \node[vertextype1] (v1) at (2,0) {$v_3$};
	\path[]
	(v2) edge [thick, blue] node {} (v3)
	(v2) edge [loop left, blue, thick] node {} (v2)
	(v3) edge [thick, blue] node {} (v1)
	(v3) edge [loop right, red, dashed, thick] node {} (v3)
	(v3) edge [thick, red, dashed, bend right] node {} (v2)
	(v3) edge [thick, dashed, red, bend left] node {} (v1);
 \end{tikzpicture} 
& \hspace{2mm}
\begin{tikzpicture}[->, scale=1.6]
	  \tikzstyle{vertextype1}=[circle, draw, minimum size=20pt,inner sep=1pt]
	 \tikzstyle{vertextype2} = [circle, draw, fill=yellow!100, minimum size=20pt,inner sep=1pt]
	 \tikzstyle{vertextype3} = [circle, draw, fill=yellow!25, minimum size=20pt,inner sep=1pt]
	 \tikzstyle{edgetype1} = [->, draw,line width=3pt,-,red!50]
	 \tikzstyle{edgetype2} = [draw,thick,-,black]
	 \node at (1,.95) {};
	 \node at (1,-.25) {};
	 \node at (1.5,.25) {\Large {\bf C}};
	 \node[vertextype1] (v3) at (2,0) {$v_3$};
	 \node[vertextype3] (v2) at (1.5, .71) {$v_2$};
	 \node[vertextype2] (v1) at (1,0) {$v_1$};
	\path[]
	(v1) edge [thick, blue] node {} (v3)
	(v2) edge [loop left, blue, thick] node {} (v2)
	(v2) edge [thick, blue] node {} (v1)
	(v3) edge [loop right, red, dashed, thick] node {} (v3)
	(v3) edge [thick, red, dashed] node {} (v2)
	(v3) edge [bend left, red, thick, dashed] node {} (v1);
 \end{tikzpicture}
 \end{tabular}
 \\ 
 \begin{tabular}{p{3.8cm}p{3.65cm}p{3.8cm}}
\vspace{-9mm} 
  \begin{equation}\label{NWA} \nonumber  
 \begin{array}{l}
 \dot x_{v_1} = f(x_{v_1}, \bl{x_{v_1}}, \ro{x_{v_1}}; \lambda) \\
 \dot x_{v_2} = f(x_{v_2}, \bl{x_{v_1}}, \ro{x_{v_1}}; \lambda)\\
\dot x_{v_3} = f(x_{v_3}, \bl{x_{v_2}}, \ro{x_{v_1}}; \lambda) 
 \end{array}
 \end{equation} & \vspace{-9mm}
  \begin{equation}\label{NWB}  \nonumber 
 \begin{array}{l}
 \dot x_{v_1} = f(x_{v_1}, \bl{x_{v_1}}, \ro{x_{v_2}}; \lambda)\\
 \dot x_{v_2} = f(x_{v_2}, \bl{x_{v_1}}, \ro{x_{v_2}}; \lambda) \\
  \dot x_{v_3} = f(x_{v_3}, \bl{x_{v_2}}, \ro{x_{v_2}}; \lambda) 
 \end{array}
 \end{equation}
& \vspace{-.9cm}
   \begin{equation}\label{NWC}  \nonumber 
 \begin{array}{l}   
  \dot x_{v_1} = f(x_{v_1}, \bl{ x_{v_2}}, \ro{x_{v_3}}; \lambda) \\
 \dot x_{v_2} = f(x_{v_2}, \bl{x_{v_2}}, \ro{x_{v_3}}; \lambda) \\
\dot x_{v_3} = f(x_{v_3}, \bl{x_{v_1}}, \ro{x_{v_3}}; \lambda) 
 \end{array}
 \end{equation} 
  \end{tabular}
\vspace{-6mm}
\caption{\footnotesize {\rm Homogeneous networks with $3$ identical cells and $2$ types of inputs.}}
   \label{pict1}
   \vspace{-4mm}
    \end{center}
       \end{figure}\\
\noindent The states of the cells of the networks are determined by variables $x_{v_1}, x_{v_2}, x_{v_3}\in \R$ (for example membrane potentials). These variables then obey the equations of motion displayed below the network graphs in Figure \ref{pict1}. The  response function $f:\R^3\times\R\to\R$ describes the precise dependence of the evolution of each cell on its own state and on its two incoming signals, and thus determines the actual dynamics of the network. We let this $f$ depend on a parameter $\lambda\in\R$, to express that it can sometimes be modified in experiments, or that it may not be entirely known. 
 
In spite of their different architectures, the dynamics of networks {\bf A}, {\bf B} and {\bf C} admit exactly the same (partial) synchronies. For example, setting $x_{v_1}=x_{v_2}$ in the equations of motion of either one of the networks, yields that $\dot x_{v_1} = \dot x_{v_2}$. The subspace $\{x_{v_1}=x_{v_2}\}$ is thus invariant under the dynamics of all three network systems, independently of the precise form of the function $f$. Similarly, $x_{v_1}=x_{v_2}=x_{v_3}$ gives that $\dot x_{v_1}=\dot x_{v_2}=\dot x_{v_3}$. Moreover, these are the only such equalities. We conclude that the subspaces 
\vspace{-1mm}
$$\{x_{v_1}=x_{v_2}\} \ \mbox{(``partial synchrony'') and}\ \{x_{v_1}=x_{v_2}=x_{v_3}\}\ \mbox{(``full synchrony'')}$$
are the two ``robust synchronies'' of networks {\bf A}, {\bf B} and {\bf C}. They can be thought of as dynamical invariants of the network graphs, see Section \ref{graphfibrationssection}. 

To understand how synchrony can emerge or disappear, assume now that $f(0,0,0;\lambda)=0$. This means that $x=(0,0,0)$ is a fully synchronous steady state of the network dynamics for all values of the parameter $\lambda$. One then says that a ``synchrony breaking steady state bifurcation'' occurs at $\lambda=0$, when less synchronous steady states emerge near this fully synchronous state as $\lambda$ varies near $0$. This can only happen if for $\lambda = 0$, the linearisation matrix of the differential equations around $x=(0,0,0)$ is degenerate. This linearisation matrix is easy to compute and reads (writing $a=D_1f(0,0,0;0)$ etc.)
\vspace{-.5cm}
\begin{center}
\begin{align}
 \nonumber
 \mbox{for network {\bf A}:}\  {  \left( \begin{array}{rrr} 
 a + \bl{b} +  \ro{c} & 0 & 0 \\
 \bl{b} + \ro{c} & a & 0    \\
 \ro{c} & \bl{b} & a 
 \end{array}\right)}  \ ; \\ 
 \nonumber
 \mbox{for network {\bf B}:}\  {  \left( \begin{array}{rrr} 
 a + \bl{b}   & \ro{c} & 0 \\
 \bl{b} & a+ \ro{c}  & 0    \\
 0 & \bl{b}+ \ro{c} & a 
 \end{array}\right)}\ ;\\ 
\nonumber
\hspace{-.5cm} \mbox{for network {\bf C}:}\ 
{   \left( \begin{array}{rrr} 
a & \bl{b} & \ro{c}\\
 0 & a + \bl{b} & \ro{c} \\
\bl{b} & 0& a+\ro{c} 
 \end{array}\right) }\ .
 \end{align}
\end{center}
\vspace{-1mm}
Interestingly, these three linearisation matrices all have an eigenvalue $a + \bl{b} + \ro{c}$ with multiplicity $1$ and a defective eigenvalue 
 $a$ with algebraic multiplicity $2$ and geometric multiplicity $1$.  The eigenvector for the eigenvalue $a + \bl{b} + \ro{c}$ is fully synchronous, so
 synchrony breaking can only occur when $a = 0$. The degeneracy of this eigenvalue suggests that the resulting steady state bifurcation may be quite unusual. Indeed, a singularity analysis (that we do not provide here) reveals that in a generic one-parameter synchrony breaking bifurcation in either one of the networks, two branches of steady states $x(\lambda)$ are born from the synchronous state: a partially synchronous and a non-synchronous branch. Table \ref{table} shows the asymptotics of these branches for the three networks.  
 \vspace{-.2cm}
\begin{table}[h]
\renewcommand{\tablename}{\rm \bf \footnotesize Table} 
\begin{center}
 \begin{tabular}{|c|c|}
 \multicolumn{2}{c}{\bf Network A}\\
\hline  \rowcolor[gray]{0.95}
{\bf Asymptotics} & {\bf Synchrony} \\ 
\hline \hline
$x_{v_1}=x_{v_2}=x_{v_3}=0$ & Full\\
\hline 
$x_{v_1}=x_{v_2}=0, x_{v_3} \sim \lambda$ & Partial \\
\hline 
$x_{v_1}=0, x_{v_2}\sim \lambda, x_{v_3} \sim \pm \sqrt{\lambda}$ & None\\
\hline   
 \multicolumn{2}{c}{ }
 \end{tabular}
 \begin{tabular}{|c|c|}
 \multicolumn{2}{c}{\bf Network B}\\
\hline  \rowcolor[gray]{0.95}
{\bf Asymptotics} & {\bf Synchrony} \\ 
\hline \hline
$x_{v_1}=x_{v_2}=x_{v_3}=0$ & Full\\
\hline 
$x_{v_1}=x_{v_2}=0, x_{v_3} \sim \lambda$ & Partial \\
\hline 
$x_{v_1}\sim \lambda, x_{v_2}\sim \lambda$ & 
\multicolumn{1}{ c| }{\multirow{2}{*}{None} }
\\ 
$x_{v_1}-x_{v_2}\sim \lambda, x_{v_3} \sim \pm \sqrt{\lambda}$ & \\
\hline 
 \end{tabular}
 \\ \vspace{1mm}
  \begin{tabular}{|c|c|}
  \multicolumn{2}{c}{\bf Network C}\\
\hline
\rowcolor[gray]{0.95} 
{\bf Asymptotics} & {\bf Synchrony} \\
\hline \hline 
$x_{v_1}=x_{v_2}=x_{v_3}=0$ & Full \\
\hline 
$x_{v_1}=x_{v_2}\sim \lambda, x_{v_3} \sim \lambda, x_{v_{1,2}}-x_{v_3} \sim \lambda$ & Partial \\ \hline 
$x_{v_1}\sim \lambda, x_{v_2} \sim \lambda, x_{v_3} \sim \lambda$ & None but \\ 
 $x_{v_1}-x_{v_2}\sim \lambda^2, x_{v_{1,2}}-x_{v_3} \sim \lambda$ & almost partial\\
\hline  
 \multicolumn{2}{c}{ }
 \end{tabular}
 \end{center}
 \vspace{-6mm}
 \caption{\footnotesize {\rm Asymptotics of steady state branches in generic synchrony breaking bifurcations. 
 }}
 \vspace{-2mm}
 \label{table}
 \end{table}
 \\
\noindent  Although networks {\bf A}, {\bf B} and {\bf C} display exactly the same robust synchronies and spectral properties, their synchrony breaking bifurcations are very different. For example, the generic synchrony breaking branches of the three networks clearly have different asymptotics. Another distinction between the networks concerns the dynamical stability of the bifurcating branches. In fact, in a generic synchrony breaking bifurcation, the fully synchronous state loses stability when $\lambda$ passes through $0$. In networks {\bf A} and {\bf B}, it is then only possible that the non-synchronous state becomes stable, but it turns out that in network {\bf C}, stability can also be transferred to the partially synchronous state. The different synchrony breaking behaviour of networks {\bf A}, {\bf B} and {\bf C} is fully determined by nonlinearities in the differential equations. This paper aims to give a geometric explanation of this nonlinear effect.    

\section{Networks}\label{networkssection}
Every dynamical system trivially has a network structure. Nevertheless, the observables of certain dynamical systems have a nontrivial interaction structure. Such a structure can be encoded in a network graph, that describes how the evolution of each observable depends on the values of others. In the literature \cite{deville, golstew, torok}, these network graphs are usually finite directed graphs, of which the vertices (also referred to as ``cells'') and arrows (``couplings'') are all of a certain type (``colour''). We have in mind that every cell has a state, that evolves in time under the influence of those cells from which it receives a coupling. One also requires compatibility between the coloured cells and coloured couplings, to express that cells of the same type respond in the same way to their inputs. The relevant definition is the following:

\begin{definition}\label{defnetwork}
A {\it network} is a finite directed graph ${\rm \bf N} = \{A\rightrightarrows^s_t V\}$ (where $A$ are the arrows, $V$ are the vertices, and $s$ and $t$ denote the source and target maps), in which all vertices and arrows are assigned a colour, such that
\begin{itemize}
\item[{\bf 1.}] if two arrows $a_1, a_2\in A$ have the same colour, then so do their sources $s(a_1)$ and $s(a_2)$, and so do their targets $t(a_1)$ and $t(a_2)$.
\item[{\bf 2.}] if two vertices $v_1, v_2\in V$ have the same colour, then there is a colour-preserving bijection $\beta_{v_2, v_1}:t^{-1}(v_1)\to t^{-1}(v_2)$ between the arrows that target $v_1$ and $v_2$.
\end{itemize}\vspace{-.5cm}
\end{definition}
The collection 
$$\mathbb{G}:=\{\, \beta_{v_2, v_1}: t^{-1}(v_1)\to t^{-1}(v_2)\ \mbox{colour preserving bijection} \, | \, v_1, v_2\in V \, \}$$ 
is   called the {\it symmetry groupoid} of the network {\bf N}. It is a groupoid, because its elements are invertible and the compositions $\beta_{v_3, v_2}\circ \beta_{v_2, v_1}$ define a partial associative product. 

The symmetry groupoid describes ``local symmetries'' between cells. Indeed, for fixed vertices $v_1, v_2\in V$, we can define $\mathbb{G}_{v_2,v_1}:= \{\beta_{v_2,v_1}\in \mathbb{G}\}$. This set is nonempty if and only if $v_1$ and $v_2$ have the same colour. The ``vertex groups'' $\mathbb{G}_{v_1,v_1}$ and  $\mathbb{G}_{v_2,v_2}$ are then isomorphic. 

Given a network {\bf N}, we now describe a natural class of maps compatible with ${\rm\bf N}$. These {\it network maps} will then give rise to network dynamical systems. First of all, we will assume that every vertex $v\in{\rm \bf N}$ has a ``state'' determined by a variable $x_v\in E_v$, taking values in a finite-dimensional vector space $E_v$ (or a manifold, but we will not pursue this straightforward generalisation). The total state of the network is thus given by an element 
$$x\in E_{\rm \bf N}:= \prod_{v\in V} E_v\, .$$         
We have in mind that the $v$-th component of a network map should only depend on the states of those vertices $w$ for which there is an arrow $a\in A$ with $s(a)=w$ and $t(a)=v$. Hence we define a ``projection'' from the total phase space onto the input variables of cell $v$, 
$$\pi_v:E_{\rm\bf N}\to  \prod_{t(a) =v} E_{s(a)} \, .$$
Note that $\pi_vx$ may contain some state variables repeatedly if different arrows that target $v$ have the same source. Finally, we choose for every vertex $v\in V$ a ``response function'' 
$$f^v: \prod_{t(a) =v} E_{s(a)} \to E_v\, . $$
The network and response functions together then yield a map with a ``network structure'':
$$\gamma_f^{\rm\bf N}: E_{\rm \bf N}\to E_{\rm \bf N} \ \mbox{defined by}\ (\gamma_f^{\rm\bf N})_v(x) := f^v(\pi_vx)\, .$$
As required, $(\gamma_f^{\rm\bf N})_v(x)$ only depends on the values $x_{s(a)}$ for those $a\in A$ with $t(a)=v$.

Finally, we will impose restrictions on the response functions to ensure compatibility of $\gamma_f^{\rm\bf N}$ with the colouring of the arrows and vertices of the network. First of all, it is natural to assume that
vertices with the same colour have the same sets of state variables: 
$$E_{v_1}=E_{v_2} \ \mbox{when}\ v_1 \ \mbox{and}\ v_2\ \mbox{have the same colour.} $$
It then follows from Definition \ref{defnetwork} that $E_{s(a)} = E_{s\left(\beta_{v_2,v_1}(a)\right)}$ for all $a\in A$ with $t(a)=v_1$. This last observation allows us to define, for all $\beta_{v_2, v_1} \in \mathbb{G}$, the input identification  
$$\beta_{v_2,v_1}^*:  \prod_{t(a) =v_2} E_{s(a)} \to \prod_{t(a) =v_1} E_{s(a)}\ \, \mbox{by}\ \, (\beta_{v_2,v_1}^*X)_{s(a)} := X_{s\left(\beta_{v_2, v_1}(a)\right)} \, .$$
We shall require that cells of the same colour respond in the same way to their incoming signals, and that signals of the same colour have the same impact on a cell, i.e.
\begin{itemize}
\item[{\bf 3.}] the response functions are groupoid-invariant:  
$$f^{v_1}\circ \beta_{v_2, v_1}^* = f^{v_2}\ \mbox{for all}\ \beta_{v_2, v_1} \in \mathbb{G}.$$ 
\end{itemize}
This final assumption expresses how the local symmetries of the network ${\rm\bf N}$ give rise to local symmetries in the components of the network maps $\gamma_f^{\rm \bf N}$. In particular, if a vertex group $\mathbb{G}_{v,v}$ is nontrivial, then $f^v$ must be invariant under certain permutations of inputs. 

We summarise our assumptions in the following definition:
\begin{definition}
A map $\gamma: E_{\rm\bf N}\to E_{\rm\bf N}$ is a {\it network map} for the network ${\rm\bf N}=\{A\rightrightarrows_t^s V\}$ if there is a set of smooth response functions $\{f^v\}_{v\in V}$ satisfying {\bf 3}, so that $\gamma=\gamma_f^{\rm \bf N}$. 
\end{definition}
 Network maps are also referred to as {\it admissible maps} in the literature. 
A ``network dynamical system'' on $E_{\rm \bf N}$  arises now from the flow of the ordinary differential equation
$$\dot x = \gamma_f^{\rm \bf N}(x)\, .$$ 
As was pointed out in \cite{deville}, one may think of this ODE as a set of coupled ``open control systems'' (namely the individual ODEs $\dot x_v = f^v(\pi_vx)$ for $v\in V$).
Rather than ODEs, one may also consider discrete-time network dynamical systems on $E_{\rm \bf N}$ of the form 
$$x_{n+1}= \gamma_f^{\rm\bf N}(x_n)\, .$$
We conclude by remarking that, as in Section \ref{examplessection}, we sometimes want to study parameter families of network dynamical systems. Then the response functions $f^v=f^v(\cdot;\lambda)$ themselves become smooth functions of a parameter $\lambda$ that takes values in some open set $\Lambda\subset \mathbb{R}^p$. For the moment, we shall suppress this parameter dependence in our notation though.

\section{Graph fibrations and robust synchrony}\label{graphfibrationssection}
Synchrony and partial synchrony are prominent forms of collective behaviour of network systems, in which certain cells undergo the same evolution. Mathematically, synchrony can be described as follows. Let $P=\{P_1, \ldots, P_r\}$ be a partition of the cells of a network ${\rm \bf N}=\{A\rightrightarrows_t^s V\}$, that is $P_1\cup \ldots\cup P_r=V$, and $P_i\cap P_j=\emptyset$ if $i\neq j$. For $v_1, v_2\in V$, we shall write $v_1\sim_P v_2$ if there is a $k$ such that $v_1, v_2\in P_k$. Then $\sim_P$ defines an equivalence relation. 
We   now define the {\it synchrony space} ${\rm Syn}_P\subset E_{\rm \bf N}$ associated to this partition as
$${\rm Syn}_P:=\{x\in E_{\rm \bf N}\, |\, x_{v_1}=x_{v_2} \ \mbox{when}\ v_1 \sim_P v_2\, \}\, .$$
For this definition to make sense, one must of course require that $E_{v_1}=E_{v_2}$ when $v_1 \sim_P v_2$. 

Of dynamical interest are those synchronies that are preserved in time. Such synchronies are determined by  synchrony spaces that are invariant under the network dynamics, i.e. for which $\gamma_f^{\rm\bf N}({\rm Syn}_P)\subset {\rm Syn}_P$. This latter inclusion clearly depends on the choice of the response functions $f^v$, but certain synchrony spaces are always dynamically invariant, irrespective of the choice of response functions. These special synchrony spaces depend only on the network ${\rm \bf N}$. The following well-known result characterises these synchrony spaces in terms of the network structure. For a more elaborate proof of Theorem \ref{balancedlemma}, we refer to \cite{pivato}.

\begin{theorem}\label{balancedlemma} Let $P$  be a partition of the cells of a network ${\rm \bf N}$. The following are equivalent:
\begin{itemize}
\item[{\it i)}] $\gamma_f^{\rm\bf N}({\rm Syn}_P)\subset {\rm Syn}_P$ for all choices of response functions $\{f^v\}_{v\in V}$ satisfying  {\bf 3}.
In this case, one says that ${\rm Syn}_P$ is a {\rm robust} synchrony space. 
\item[{\it ii)}] For all vertices $v_1 \sim_P v_2$, there is a $\beta_{v_2, v_1}\in \mathbb{G}_{v_2, v_1}$ such that for every arrow $a$  with $t(a)=v_1$, it holds that
$s(a) \sim_P s\left(\beta_{v_2,v_1}(a)\right)$. The partition is then called {\rm balanced}.
\end{itemize}
\end{theorem}
\begin{proof}{\bf (Sketch)}
{\it ``ii)}\ $\Rightarrow$ {\it i)''} It follows from {\it ii)} that $\pi_{v_1}x= \beta_{v_2, v_1}^*(\pi_{v_2}x)$ for all $x\in {\rm Syn}_P$ and all $v_1\sim_P v_2$. As a result, by assumption {\bf 3},
$$f^{v_1}(\pi_{v_1}x) = f^{v_1}( \beta_{v_2, v_1}^* ( \pi_{v_2}x) ) = f^{v_2}(\pi_{v_2}x) \, .$$
This proves that $(\gamma_f^{\rm\bf N})_{v_1}(x)=(\gamma_f^{\rm\bf N})_{v_2}(x)$ for any $v_1\sim_P v_2$ and $x\in {\rm Syn}_P$. In other words, $\gamma_f^{\rm\bf N}({\rm Syn}_P)\subset {\rm Syn}_P$. Since this is true for every choice of $\{f^v\}_{v\in V}$, the synchrony is robust.

{\it ``i)} $\Rightarrow$ {\it ii)''} 
Let $v_1\sim_P v_2$ be fixed and choose an arrow $a$ with $t(a)=v_1$. Let us write $a'\sim a$ if the arrows $a'$ and $a$ have same colour. Finally, let $\alpha: E_{s(a)}\to E_{v_1}$ be any nonzero linear map. Then we define the special response functions
$$f^{v}(X):=\!\! \sum_{\tiny \begin{array}{cc} t(a')=v\\ a'\sim a \end{array}}\!\! \alpha\left(X_{s(a')}\right)\, .$$
Because elements of $\mathbb{G}$ preserve the colour of arrows, these response functions satisfy assumption {\bf 3}. We shall evaluate $(\gamma_f^{\rm\bf N})_{v_1}$ and $(\gamma_f^{\rm\bf N})_{v_2}$ at a point $x^*\in {\rm Syn}_P$ with $\alpha(x^*_{s(a)})=e \neq 0$ (and hence $\alpha(x^*_{v}) = e$ for all $v \sim_P s(a)$) and $\alpha(x^*_{v})=0$ for all $v\not \sim_P s(a)$. For such $x^*$,
$$(\gamma_f^{\rm\bf N})_{v_i}(x^*)=f^{v_i}( \pi_{v_i}x^*) = \sharp \{a'\in t^{-1}(v_i)\, |\, a'\sim a\ \mbox{and}\ s(a')\sim_P s(a) \, \}   \cdot e \, .$$
By assumption {\it i)}, it holds that $(\gamma_f^{\rm\bf N})_{v_1}(x^*)= (\gamma_f^{\rm\bf N})_{v_2}(x^*)$. It follows for all arrows $a$ that
$$\sharp \{a'\in t^{-1}(v_1)\, |\, a'\sim a\ \mbox{and}\ s(a')\sim_P s(a) \, \} = \sharp \{a'\in t^{-1}(v_2)\, |\, a'\sim a\ \mbox{and}\ s(a')\sim_P s(a) \, \} \, . $$
This implies that there is a bijection $\beta_{v_2, v_1}\in \mathbb{G}_{v_2,v_1}$ with the desired properties.
\end{proof}
Theorem \ref{balancedlemma} was recently generalised by DeVille and Lerman \cite{deville}. They formulate their result in the language of category theory and graph fibrations. 
 
 \begin{definition}\label{deffibration}
  A map $\phi:{\rm \bf N}_1\to {\rm \bf N}_2$ of networks is a {\it graph fibration} if \\ \vspace{-.1cm}\\
 \begin{tabular}{lp{11.5cm}}
 {\it i)} & it sends cells to cells of the same colour, arrows to arrows of the same colour, and the head and tail of every arrow $a_1\in {\rm \bf N}_1$ to the head and tail of $\phi(a_1)\in {\rm \bf N}_2$;\\
  {\it ii)} & for every cell $v_1 \in {\rm \bf N}_1$ and every arrow $a_2\in {\rm \bf N}_2$ ending at $\phi(v_1)$, there is a unique arrow $a_1\in \phi^{-1}(a_2)$ that ends at $v_1$. 
  \end{tabular}
  \end{definition}
Property {\it i)} simply requires that $\phi$ is a morphism of coloured directed graphs. Property {\it ii)} is the fibration property: it says that $\phi$ restricts to a colour-preserving bijection 
$$\left. \phi \right|_{t^{-1}(v_1)}: t^{-1}(v_1) \to t^{-1}(\phi(v_1))\, ,$$
between the arrows targeting any vertex $v_1\in{\rm \bf N}_1$ and those targeting its image $\phi(v_1)\in {\rm \bf N}_2$. 

When $\phi:{\rm \bf N}_1\to {\rm \bf N}_2$ is a graph fibration, we call  ${\rm \bf N}_2$ a {\it quotient} of ${\rm \bf N}_1$ and ${\rm \bf N}_1$ a {\it lift} of ${\rm \bf N}_2$. Despite this terminology, we will not require that $\phi$ is surjective.
Figure \ref{pictquotient} depicts quotients of networks {\bf A}, {\bf B} and {\bf C}, and the action of the corresponding graph fibrations on vertices. 
\begin{figure}[h]\renewcommand{\figurename}{\rm \bf \footnotesize Figure} 
\vspace{-.1cm}
\begin{center}
\begin{tikzpicture}[->, scale=1.6]
  	 \tikzstyle{vertextype1}=[circle, draw, minimum size=20pt,inner sep=1pt]
	 \tikzstyle{vertextype2} = [circle, draw, fill=yellow, minimum size=20pt,inner sep=1pt]
	 \tikzstyle{vertextype3} = [circle, draw, fill=yellow!25, minimum size=20pt,inner sep=1pt]
	  \tikzstyle{vertextype4} = [circle, draw, fill=yellow!62, minimum size=20pt,inner sep=1pt]
	 \tikzstyle{edgetype1} = [->, draw,line width=3pt,-,red!50]
	 \tikzstyle{edgetype2} = [draw,thick,-,black]
	 \node at (1.3,-1.15) {\Large {\bf C}};
	 \node[vertextype1] (v3) at (1.8,-1.4) {$v_3$};
	 \node[vertextype2] (v2) at (1.3, -.7) {$v_2$};
	 \node[vertextype2] (v1) at (.8,-1.4) {$v_1$};
	 \node at (1.3,1.6) {\Large {\bf A}};
	 \node[vertextype2] (v3e) at (.8,1.3) {$v_1$};
	 \node[vertextype2] (v2e) at (1.3, 2) {$v_2$};
	 \node[vertextype1] (v1e) at (1.8,1.3) {$v_3$};
	 \node at (4.5,0) {\Large {\bf E}};
	 \node[vertextype2] (v1b) at (4,-.3) {$v_1$};
	 \node[vertextype1] (v2b) at (5,-.3) {$v_2$};
	  \node at (4.5,1.7) {\Large {\bf D}};
	 \node[vertextype2] (v1c) at (4,1.2) {$v_1$};
	 \node[vertextype1] (v2c) at (5,1.2) {$v_2$};
	  \node at (7.1,1) {\Large {\bf F}};
	  \node[vertextype3] (v) at (7.1,.5) {$v_1$};
	  
	   \node at (1.3,.2) {\Large {\bf B}};
	 \node[vertextype2] (v2f) at (0.8,0) {$v_1$};
	 \node[vertextype2] (v3f) at (1.3, .7) {$v_2$};
	 \node[vertextype1] (v1f) at (1.8,0) {$v_3$};

	 \draw [->,decorate,decoration=snake, thick] (2,1.8) -- (3.33,1.4) node [above, midway, sloped] {{\footnotesize $v_1,v_2\mapsto v_1$}} node [below, midway, sloped] {{\footnotesize $v_3\mapsto v_2$}}; 
	 \draw [->,decorate,decoration=snake, thick] (2,.4) -- (3.33,.8) node [above, midway, sloped] {{\footnotesize $v_1,v_2\mapsto v_1$}} node [below, midway, sloped] {{\footnotesize  $v_3\mapsto v_2$}};
	  \draw [->,decorate,decoration=snake, thick] (2,-1) -- (3.33,-.6) node [above, midway, sloped] {{\footnotesize $v_1,v_2\mapsto v_1$}} node [below, midway, sloped] {{\footnotesize  $v_3\mapsto v_2$}};
	 \draw [->,decorate,decoration=snake, thick] (5.7,-.2) -- (6.54,.1) node [above, midway, sloped] {{\footnotesize $v_1,v_2\mapsto v_1$}};
	  \draw [->,decorate,decoration=snake, thick] (5.7,1.1) -- (6.54,.8) node [above, midway, sloped] {{\footnotesize $v_1,v_2\mapsto v_1$}};
	\path[]
	(v1) edge [thick, blue] node {} (v3)
	(v2) edge [loop left, blue, thick] node {} (v2)
	(v2) edge [thick, blue] node {} (v1)
	(v3) edge [loop right, red, dashed, thick] node {} (v3)
	(v3) edge [thick, red, dashed] node {} (v2)
	(v3) edge [bend left, red, thick, dashed] node {} (v1)
	(v2e) edge [thick, blue] node {} (v1e)
	(v3e) edge [loop left, blue, thick] node {} (v3e)
	(v3e) edge [thick, blue] node {} (v2e)
	(v3e) edge [loop right, red, dashed, thick] node {} (v3e)
	(v3e) edge [thick, bend left, red, dashed] node {} (v2e)
	(v3e) edge [thick, dashed, bend right, red] node {} (v1e)
	(v1b) edge [thick, blue, loop left] node {} (v1b)
	(v1b) edge [blue, thick] node {} (v2b)
	(v2b) edge [loop right, red, dashed, thick] node {} (v1b)
	(v2b) edge [thick, bend left, red, dashed] node {} (v1b)
	(v1c) edge [thick, blue, loop left] node {} (v1c)
	(v1c) edge [blue, thick, bend left] node {} (v2c)
	(v1c) edge [loop right, red, dashed, thick] node {} (v1c)
	(v1c) edge [thick, bend right, red, dashed] node {} (v2c)
	(v) edge [blue, thick, loop left] node {} (v)
	(v) edge [loop right, red, dashed, thick] node {} (v)
	(v2f) edge [thick, blue] node {} (v3f)
	(v2f) edge [loop left, blue, thick] node {} (v2f)
	(v3f) edge [thick, blue] node {} (v1f)
	(v3f) edge [loop right, red, dashed, thick] node {} (v3f)
	(v3f) edge [thick, red, dashed, bend right] node {} (v2f)
	(v3f) edge [thick, dashed, red, bend left] node {} (v1f);
	 \end{tikzpicture}
 \caption{\footnotesize {\rm Graph fibrations that explain the robust synchronies of networks {\bf A}, {\bf B} and {\bf C}.}}
  \vspace{-.3cm}
\label{pictquotient}
\end{center}
 \end{figure}
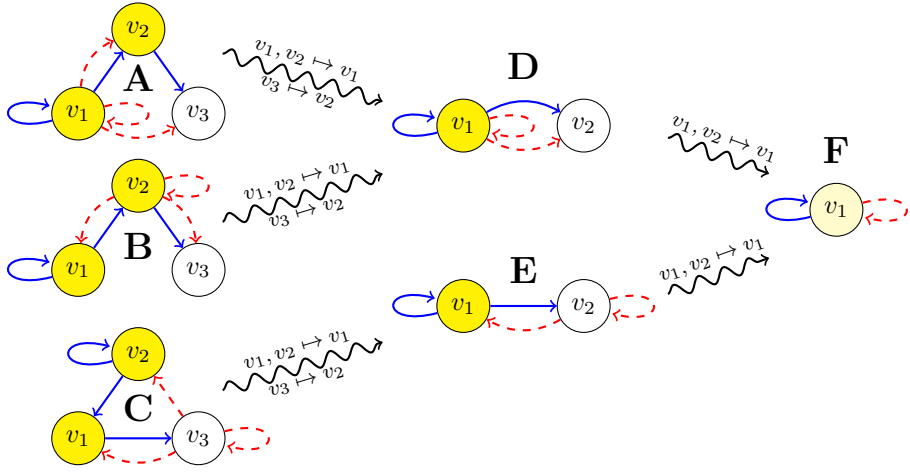
 \\
\noindent 
The dynamical relevance of graph fibrations is explained by the following theorem from \cite{deville}.
 
 \begin{theorem}[DeVille \& Lerman]\label{devilletheorem1}
Let $\phi: {\rm \bf N}_1\to {\rm \bf N}_2$ be a graph fibration. 
Define the map $\phi^*: E_{\rm \bf N_2} \to E_{\rm \bf N_1}$ between the phase spaces of networks ${\rm \bf N}_2$ and ${\rm \bf N}_1$ by
$$(\phi^*y)_v := y_{\phi(v)}\, .$$
Then $\phi^*$ sends every solution $y(t)$ of the dynamics of network ${\rm \bf N}_2$ to a solution $x(t):=\phi^*y(t)$ of the dynamics of network ${\rm \bf N}_1$, that is
$$\phi^* \circ \gamma_f^{\rm\bf N_2} = \gamma_f^{\rm\bf N_1} \circ \phi^*\, .$$
The solution $x(t)=\phi^*y(t)$ has the robust synchrony $x_{v_1}(t)=x_{v_2}(t)$ when $\phi(v_1)=\phi(v_2)$. Moreover, every robust synchrony of ${\rm \bf N}_1$ arises from a graph fibration in this way. 
\end{theorem}
\begin{proof}
Let ${\rm \bf N}_1=\{A_1\rightrightarrows_{t_1}^{s_1} V_1\}$ and ${\rm \bf N}_2=\{A_2\rightrightarrows_{t_2}^{s_2} V_2\}$ be networks and let $\phi:{\rm \bf N}_1\to{\rm \bf N}_2$ be a graph fibration. This implies that $s_2\circ \left.\phi\right|_{A_1} = \left.\phi\right|_{V_1}\circ s_1$ and $t_2\circ \left.\phi\right|_{A_1}= \left.\phi\right|_{V_1} \circ t_1$. Moreover, recall that for every $v\in V_1$ the restriction   
 $$\left. \phi \right|_{t_1^{-1}(v)}: t_1^{-1}(v) \to t_2^{-1}(\phi(v))\, $$
 is a colour preserving bijection, yielding an identification between the inputs of $v$ and $\phi(v)$
$$\left(  \left. \phi \right|_{t_1^{-1}(v)}  \right)^*:  \prod_{t_2(a) =\phi(v)} E_{s_2(a)} \to \prod_{t_1(a) =v} E_{s_1(a)}\, .$$
This map satisfies $\left( \left( \left. \phi \right|_{t_1^{-1}(v)} \right )^*Y\right)_{s_1(a)} := Y_{s_2(\phi(a))} = y_{s_2(\phi(a))} = y_{\phi(s_1(a))} =(\phi^*y)_{s_1(a)}$ for any choice of input variables $Y=\pi_{\phi(v)}(y)$ of cell $\phi(v)$.
In other words, 
$$\left(  \left. \phi \right|_{t_1^{-1}(v)}  \right)^*\circ \pi_{\phi(v)} = \pi_v \circ \phi^*\, .$$
Finally, because $v$ and $\phi(v)$ have the same colour, it holds that $f^{v} \circ  \left( \left. \phi \right|_{t_1^{-1}(v)}\right)^*  = f^{\phi(v)}$.

After these technical remarks, let   $y(t)\in E_{{\rm\bf N}_2}$ be a solution of the dynamics of ${\rm \bf N}_2$, that is $\dot y_w(t)=(\gamma_f^{{\rm\bf N}_2})_w(y)=f^w(\pi_w y(t))$ for all cells $w$ of ${\rm \bf N}_2$, and let $x(t):=\phi^*y(t)\in E_{{\rm\bf N}_1}$.
Then 
 \begin{align}\nonumber
  \dot x_v(t)  & = \dot y_{\phi(v)}(t) = f^{\phi(v)} \left(\pi_{\phi(v)}y(t) \right) =
  \left(  f^v \circ \left( \left. \phi \right|_{t_1^{-1}(v)}\right)^* \circ \pi_{\phi(v)}\right) ( y(t) )  \\ \nonumber & =  \left( f^{v}\circ \pi_{v}\circ \phi^*\right)(y(t))  = f^v(\pi_v(x(t))) = (\gamma_f^{{\rm\bf N}_1})_v(x(t))\, .
\end{align}
This proves that $\phi^*$ sends solutions to solutions and hence that $\phi^* \circ \gamma_f^{\rm\bf N_2} = \gamma_f^{\rm\bf N_1} \circ \phi^*$. 

To prove that every robust synchrony arises from a graph fibration, assume that $P=\{P_1, \ldots, P_r\}$ is a balanced partition of the cells of the network ${\rm \bf N}=\{A\rightrightarrows_t^s V\}$. We now define a new network ${\rm \bf  N}^P=\{A^P \rightrightarrows_t^s V^P\}$ with cells $V^P:=\{v^P_1, \ldots, v^P_r\}$. 

The arrows $A^P$ of ${\rm \bf N}^P$ are constructed by choosing, for every cell $v^P_j \in V^P$, one arbitrary but distinguished cell $v_j\in P_j$. We then construct, for every arrow $a_j\in t^{-1}(v_j)$, a corresponding arrow $a_j^P\in A^P$ with the same colour as $a_j$. We require that the source of $a_j^P$ is $v^P_i$ if $s(a_j)\in P_i$ and that the target of $a_j^P$ is $v^P_j$. 

From the fact that $P$ is balanced, it follows that ${\rm \bf N}^P$ is a quotient of ${\rm \bf N}$. A quotient map $\phi: {\rm \bf N}\to{\rm \bf N}^P$ can be constructed by letting 
$\phi(v):=v_j^P$ for 
 every cell $v\in P_j$, and by choosing for each $v\in P_j$ one of the $\beta_{v_j, v} \in \mathbb{G}_{v_j, v}$ of part {\it ii)} of Theorem \ref{balancedlemma}, to define $\phi(a) := (\beta_{v_j,v}(a))^P$ for every  $a\in t^{-1}(v)$. It is clear from property {\it ii)} of Theorem \ref{balancedlemma} that this $\phi$ is a graph fibration and that ${\rm im}\, \phi^*={\rm Syn}_{P}$.
\end{proof}
More so than the rather combinatorial Theorem \ref{balancedlemma}, the result of DeVille and Lerman provides a geometric explanation of the occurrence of synchrony: robust synchrony is a consequence of the existence of graph fibrations and of the resulting conjugacies of dynamical systems. Figure \ref{pictquotient} depicts the graph fibrations that are responsible for the robust synchronies of the networks {\bf A}, {\bf B} and {\bf C} that were discussed in Section \ref{examplessection}.
\begin{remark}\label{contravariant}
Let $\phi: {\rm \bf N}_1\to {\rm \bf N}_2$ and $\psi: {\rm \bf N}_2\to {\rm \bf N}_3$ be graph fibrations and 
$$\phi^*: E_{{\rm \bf N}_2} \to E_{{\rm \bf N}_1}\ \mbox{and}\ \psi^*:E_{{\rm \bf N}_3} \to E_{{\rm \bf N}_2}$$ the conjugacies resulting from Theorem \ref{devilletheorem1}. Then  $\psi\circ\phi: {\rm\bf N}_1\to{\rm\bf N}_3$ is also a graph fibration. 
Moreover, for $z \in E_{{\rm \bf N}_3}$ we have
$((\psi\circ \phi)^*z)_v = z_{(\psi\circ\phi)(v)} = z_{\psi(\phi(v))} = (\psi^*z)_{\phi(v)} = (\phi^*(\psi^*z))_v$.
This proves that
$$(\psi\circ \phi)^* = \phi^* \circ \psi^*\, .$$
Alternatively, one may express this  by saying that the map $*:\phi\mapsto \phi^*$ determines a contravariant functor from the category of networks to the category of dynamical systems. See \cite{deville} for more details on the categorical approach to network dynamics.
\end{remark}
We will use the following simple remark later:
\begin{proposition}\label{injsurjremark}
 When $\phi:{\rm \bf N}_1\to {\rm \bf N}_2$ is surjective, then  $\phi^*: E_{\rm \bf N_2} \to E_{\rm \bf N_1}$ is injective. When $\phi$ is injective, then $\phi^*$ is surjective. 
 \end{proposition}
 \begin{proof}
This all follows directly from the definition $(\phi^*x)_v :=x_{\phi(v)}$. 

Assume for instance that $\phi^*y=\phi^*Y$. Then $y_{\phi(v)}=Y_{\phi(v)}$ for all cells $v$ of ${\rm \bf N}_1$. When $\phi$ is surjective, this implies that $y_w=Y_w$ for all cells $w$ of ${\rm\bf N}_2$. Thus, $y=Y$ and $\phi^*$ is injective.

When $\phi$ is injective, let $x\in E_{{\rm\bf N}_1}$ be given and choose any $y\in E_{{\rm\bf N}_2}$ satisfying $y_{w}=x_{v}$ whenever $w=\phi(v)$. Injectivity makes this possible. Then $\phi^*y=x$ and $\phi^*$ is surjective. 
\end{proof}


\section{Homogeneous networks} \label{homnetsection}
 We shall restrict our attention to a rather simple class of networks for a while:
\begin{definition}\label{defhomogeneous}
A {\it homogeneous network} is a network with vertices of one single colour, in which the arrows that target one vertex all have a different colour.
\end{definition}
A network ${\rm\bf N}=\{A\rightrightarrows_t^s V\}$ is homogeneous  precisely if $\sharp \,\mathbb{G}_{v_2, v_1}=1$ for all pairs $v_1, v_2\in V$. In particular, every cell of such a network has the same phase space $E_v=E$ and responds in the same way to its incoming signals. Also, signals of a different colour may have a different effect on a cell. Homogeneous networks have the advantage that they allow for a rather simple algebraic treatment. 
For example, one calls the number of incoming arrows of a cell the ``valency'' of that cell. Note that homogeneous networks have a single valency. A homogeneous network of valency $m$ can thus conveniently be described by $m$ ``input maps'' 
$$\sigma_1, \ldots, \sigma_m: V \to V ,$$
in which $\sigma_j(v)$ is the source of the unique arrow of colour $j$ that targets vertex $v$. It is clear that $\pi_vx=(x_{\sigma_1(v)}, \ldots, x_{\sigma_m(v)})$ 
and that a choice of response function $f:E^m \to E$ leads to a homogeneous network map $\gamma_f^{\rm\bf N}: E_{\rm\bf N} \to E_{\rm\bf N}$ of the form
\begin{equation}\label{homogeneousequation}
(\gamma_f^{\rm\bf N})_v(x) = f\left(x_{\sigma_1(v)}, \ldots, x_{\sigma_m(v)}\right)\ \mbox{for all}\ v\in V\, .
\end{equation}
To guarantee that every cell notices its own state, we shall assume from now on that
$$\sigma_1 = {\rm Id}_{V}\, .$$
 Formula (\ref{homogeneousequation}) moreover shows that, without loss of generality, we can assume that all the $\sigma_j$'s are different: if $\sigma_i=\sigma_j$ for $i\neq j$, then the arrows of colours $i$ and $j$ can be identified, and $f$ can then be redefined to depend on less variables.
\begin{example}\label{inputmapsexample}
Networks {\bf A}, {\bf B} and {\bf C} of Figure \ref{pict1} are examples of homogeneous networks with $3$ cells of valency $3$. The maps $\sigma_1, \sigma_2, \sigma_3$ are given in this case by
 \begin{align}\nonumber
 \begin{array}{c|ccc} {\rm \bf A} & v_1 & v_2 & v_3 \\ \hline 
\sigma_1 & v_1 & v_2  & v_3   \\
\bl{\sigma_2} &  \bl{v_1} & \bl{v_1} & \bl{v_2}  \\
\ro{\sigma_3} & \ro{v_1} & \ro{v_1} & \ro{v_1}  
\end{array} \hspace{.5cm}
   \begin{array}{c|ccc} {\rm \bf B} & v_1 & v_2 & v_3 \\ \hline 
\sigma_1 & v_1 & v_2  & v_3   \\
\bl{\sigma_2} &  \bl{v_1} & \bl{v_1} & \bl{v_2}  \\
\ro{\sigma_3} & \ro{v_2} & \ro{v_2} & \ro{v_2}  
\end{array}
 \hspace{.5cm}
\begin{array}{c|ccc} {\rm \bf C} & v_1 & v_2 & v_3 \\ \hline 
\sigma_1  & v_1 & v_2 & v_3   \\
\bl{\sigma_2} &  \bl{v_2} & \bl{v_2} & \bl{v_1}  \\
\ro{\sigma_3}  & \ro{v_3} & \ro{v_3} & \ro{v_3}   
\end{array} 
 \ .
\end{align}
Thus, $\sigma_1=$``arrows from every cell to itself'', $\sigma_2 =$ ``all blue arrows" and $\sigma_3 =$ ``all red arrows". Figure \ref{pict1} does not depict the arrows corresponding to $\sigma_1$. The figure also displays the homogeneous network differential equations $\dot x=\gamma_f^{\rm \bf A}(x), \dot x=\gamma_f^{\rm \bf B}(x)$ and $\dot x=\gamma_f^{\rm \bf C}(x)$.
\end{example}
The following simple proposition characterises graph fibrations of homogeneous networks.
\begin{proposition}\label{obviousprop}
Let ${\rm \bf N}_1=\{A_1\rightrightarrows_{t_1}^{s_1} V_1\}$ and ${\rm \bf N}_2=\{A_2\rightrightarrows_{t_2}^{s_2} V_2\}$ be homogeneous networks of valency $m$, respectively with input maps 
$$\sigma_1^{(1)}, \ldots, \sigma_m^{(1)}: V_1 \to V_1\ \mbox{and} \ \sigma_1^{(2)}, \ldots, \sigma_m^{(2)}: V_2 \to V_2\, .$$ 
Then $\phi: {\rm \bf N}_1\to {\rm \bf N}_2$ is a graph fibration if and only if
$$\left.\phi\right|_{V_1} \circ \sigma_j^{(1)} = \sigma_j^{(2)} \circ \left. \phi\right|_{V_1}\ \mbox{for all colours} \ 1\leq j \leq m\, .$$
\end{proposition}
\begin{proof}
It is obvious from Definitions \ref{deffibration} and \ref{defhomogeneous} that $\phi$ must send $\sigma_j^{(1)}(v)$ to $\sigma_j^{(2)}(\phi(v))$. Moreover, it is a graph fibration if it does so.
\end{proof}
\begin{remark}\label{balancedremark}
It is not hard to prove (see for example Proposition 7.2 in \cite{CCN}) that a partition $V =P_1\cup \ldots \cup P_r$ of the cells of a homogeneous network ${\rm \bf N}$ with input maps $\sigma_1, \ldots, \sigma_m$ is balanced if and only if for all $1\leq j \leq m$ and $1\leq k\leq r$ there is an $1\leq l \leq r$ such that
$$\sigma_j(P_k)\subset P_l\, .$$ 
The input maps then descend to maps on the partition. In fact, we can construct a new homogeneous network ${\rm\bf N}^P$ with cells
$\{v_1^P, \ldots, v_r^P\}$ and input maps $\sigma_1^{P}, \ldots, \sigma_m^{P}$ that satisfy  $$\sigma_j^P(v_k^P)=v_l^P\ \mbox{if and only if} \ \sigma_j(P_k)\subset P_l\, .$$
By definition, the map of vertices
$$\phi: V \to \{v_1^P, \ldots, v_r^P\} \ \mbox{defined by}\ \phi(v) = v^P_j \ \mbox{if and only if}\ v\in P_j $$ 
then satisfies $\phi \circ \sigma_j = \sigma_j^{P} \circ \phi\ \mbox{for all} \ 1\leq j \leq m$ and thus extends to a graph fibration. This confirms that ${\rm\bf N}^P$ is a quotient of ${\rm\bf N}$. 
\end{remark}
``Nonhomogeneous networks'', which have different cell types but no interchangeable inputs, can be described in similar way \cite{CCN}. Although the notation is heavier, the results of this paper remain true for such networks, see Section \ref{nonhomogeneoussection}. 
Networks with a nontrivial symmetry groupoid, in which certain cells receive several arrows of the same colour, can not be described by a unique collection of input maps. Some results in this paper therefore do not have an obvious generalisation to such networks.

 \section{The fundamental network}\label{fundamentalsection}
In this section we define, for every homogeneous network, the lift with self-fibrations mentioned in the introduction. Recall that every homogeneous network ${\rm \bf N}=\{A\rightrightarrows_t^s V \}$ can be described by input maps $\sigma_1, \ldots, \sigma_m:V\to V$. In general, the composition $\sigma_j\circ\sigma_k$ need not be equal to any $\sigma_i$, but there does exist a smallest collection 
$$\Sigma_{\rm \bf N} = \{\sigma_1, \ldots, \sigma_m, \ldots, \sigma_n\}$$ 
that contains $\sigma_1, \ldots, \sigma_m$ and is closed under composition. This $\Sigma_{\rm \bf N}$ is unique up to renumbering its elements. By definition, it is a semigroup (composition of maps being the semigroup operation) with unit (i.e. a ``monoid''), where we recall our assumption that $\sigma_1={\rm Id}_{V}$. 
\begin{remark}\label{semigroupremark}
One aspect of the relevance of $\Sigma_{\rm \bf N}$ is easy to explain. Let  $1\leq j_1, \ldots j_q \leq m$ be a sequence of colours. Then there is a path in {\bf N} from cell $(\sigma_{j_1}\circ\ldots \circ \sigma_{j_q})(v)$ to cell $v$, consisting of a sequence of arrows of colours $j_1, \ldots, j_q$ respectively. Cell $(\sigma_{j_1}\circ\ldots \sigma_{j_q})(v)$ thus acts ``indirectly'' as an input of cell $v$. Because $\Sigma_{\rm \bf N}$ is closed under composition, the set 
$$V_{(v)}:= \{ \sigma_j(v)\, |\, \sigma_j\in\Sigma_{\rm\bf N} \}$$
is equal to the set of vertices in $\bf N$ from which there is a path to $v$. Moreover, $\Sigma_{\rm \bf N}$ determines all the sets $V_{(v)}$ (with $v\in V$) simultaneously. Nevertheless, it will become clear that much more information is contained in the product structure of $\Sigma_{\rm \bf N}$.
\end{remark}
We are now ready to define another homogeneous network as follows:
\begin{definition}
Let ${\rm \bf N}$ be a homogeneous network with input maps $\sigma_1, \ldots, \sigma_m$ and let $\Sigma_{\rm \bf N}$ be the above semigroup. 
The {\it fundamental network}
$\widetilde {\rm \bf N}$ of {\bf N} is the homogeneous network with vertex set $\Sigma_{\rm \bf N}$ and input maps $\widetilde \sigma_1,  \ldots, \widetilde \sigma_m$ defined by
$$\widetilde \sigma_k(\sigma_j):= \sigma_{k}\circ\sigma_j\ \mbox{for}\ 1\leq k\leq m \, .$$ 
In other words, $\widetilde {\rm \bf N}$ contains an arrow of colour $k$ from $\sigma_i$ to $\sigma_j$ if and only if $\sigma_i=\sigma_k\circ \sigma_j$. 
\end{definition} 
The map $\widetilde \sigma_k:\Sigma_{\rm \bf N}\to\Sigma_{\rm \bf N}$ encodes the left-multiplicative behaviour of $\sigma_k\in \Sigma_{\rm \bf N}$. Thus, the fundamental network is   a graphical representation of the semigroup $\Sigma_{\rm \bf N}$ together with its generators $\sigma_1, \ldots, \sigma_m$. Such a graphical representation is   called a {\it Cayley graph}.
Note that the fundamental network $\widetilde {\rm\bf N}$ of {\bf N} can easily be constructed from the product table of $\Sigma_{\rm \bf N}$. 
\begin{example}\label{fundamentalexample}
Recall the homogeneous networks {\bf A}, {\bf B} and {\bf C} of Figure \ref{pict1} and their input maps given in Example \ref{inputmapsexample}.  In network {\bf A}, $\sigma_2^2 = \sigma_3^2  = \sigma_3\circ \sigma_2=\sigma_2\circ\sigma_3=\sigma_3$. Hence
$$\Sigma_{\rm \bf A}=\{\sigma_1, \sigma_2, \sigma_3\}$$ 
is already a semigroup. In network {\bf B}, on the other hand, $\sigma_2^2\neq \sigma_{1,2,3}$, so the collection $\{\sigma_1, \sigma_2, \sigma_3\}$ needs to be extended
to obtain a collection
$$\Sigma_{\rm \bf B}=\{\sigma_1, \sigma_2, \sigma_3, \sigma_4\}\ \mbox{with}\ \sigma_4=\sigma_2^2\, $$
that is closed under composition. Similarly, the input maps of network ${\rm\bf C}$ require an extension (in fact by two elements)
 to  
$$\Sigma_{\rm \bf C}=\{\sigma_1, \sigma_2, \sigma_3, \sigma_4, \sigma_5\}\ \mbox{with} \ \sigma_4=\sigma_2^2 \ \mbox{and}\ \sigma_5=\sigma_2\circ\sigma_3\, .$$ 
The resulting input maps are the following.
 \begin{equation}
 \nonumber
 \begin{array}{c|ccc} {\rm \bf A} & v_1 & v_2 &v_3 \\ \hline 
\sigma_1 & v_1 & v_2  & v_3   \\
\bl{\sigma_2} &  \bl{v_1} & \bl{v_1} & \bl{v_2}  \\
\ro{\sigma_3} & \ro{v_1} & \ro{v_1} & \ro{v_1}  \\
\multicolumn{4}{c}{ }\\
\multicolumn{4}{c}{ }
\end{array}
 \hspace{.5cm}
   \begin{array}{c|ccc} {\rm \bf B} & v_1 & v_2 & v_3 \\ \hline 
\sigma_1 & v_1 & v_2  & v_3   \\
\bl{\sigma_2} &  \bl{v_1} & \bl{v_1} & \bl{v_2}  \\
\ro{\sigma_3} & \ro{v_2} & \ro{v_2} & \ro{v_2}  \\
\sigma_4 & v_1 & v_1 & v_1\\
\multicolumn{4}{c}{ }
\end{array}\hspace{.5cm}
\begin{array}{c|ccc} {\rm \bf C} & v_1 & v_2 & v_3 \\ \hline 
\sigma_1  & v_1 & v_2 & v_3   \\
\bl{\sigma_2} &  \bl{v_2} & \bl{v_2} & \bl{v_1}  \\
\ro{\sigma_3}  & \ro{v_3} & \ro{v_3} & \ro{v_3} \\
\sigma_4 & v_2&v_2&v_2\\
\sigma_5 & v_1 & v_1 & v_1 
\end{array}  \ .
\end{equation}

\noindent One checks that the composition/product tables of $\Sigma_{\rm \bf A}$, $\Sigma_{\rm \bf B}$ and $\Sigma_{\rm \bf C}$ read
 \begin{align}\nonumber 
\begin{array}{c|ccc} \Sigma_{\rm \bf A} & \sigma_1 & \sigma_2 & \sigma_3   \\ \hline 
\sigma_1 & \sigma_1 & \sigma_2  & \sigma_3  \\
\sigma_2 & \sigma_2 & \sigma_3 & \sigma_3  \\
\sigma_3 & \sigma_3 & \sigma_3 & \sigma_3  \\
\multicolumn{4}{c}{ }\\
\multicolumn{4}{c}{ }
\end{array}  \hspace{.5cm}
\begin{array}{c|cccc} \Sigma_{\rm \bf B} & \sigma_1 & \sigma_2 & \sigma_3 &\sigma_4  \\ \hline 
\sigma_1 & \sigma_1 & \sigma_2  & \sigma_3  & \sigma_4 \\
\sigma_2 & \sigma_2 & \sigma_4 & \sigma_4 & \sigma_4  \\
\sigma_3 & \sigma_3 & \sigma_3 & \sigma_3 & \sigma_3\\
\sigma_4 & \sigma_4 & \sigma_4 & \sigma_4 & \sigma_4 \\
\multicolumn{5}{c}{ }
\end{array}   \hspace{.5cm}
\begin{array}{c|ccccc} \Sigma_{\rm \bf C} & \sigma_1 & \sigma_2 & \sigma_3 & \sigma_4 & \sigma_5 \\ \hline 
\sigma_1 & \sigma_1 & \sigma_2  & \sigma_3 & \sigma_4 & \sigma_5 \\
\sigma_2 & \sigma_2 & \sigma_4 & \sigma_5 & \sigma_4& \sigma_4 \\
\sigma_3 & \sigma_3 & \sigma_3 & \sigma_3 & \sigma_3  & \sigma_3 \\
\sigma_4 & \sigma_4 & \sigma_4 & \sigma_4 & \sigma_4  & \sigma_4 \\
\sigma_5 & \sigma_5 & \sigma_5 & \sigma_5 & \sigma_5 &  \sigma_5
\end{array} 
\hspace{.1cm} .
\end{align}
One reads off the input maps $\widetilde \sigma_1, \widetilde \sigma_2$ and $\widetilde \sigma_3$ of the lifts $\widetilde {\rm \bf A}, \widetilde {\rm \bf B}$ and $\widetilde{\rm \bf C}$. They are given by
\begin{align}\nonumber
 \begin{array}{c|ccc} \widetilde{\rm \bf A} & \sigma_1 & \sigma_2& \sigma_3 \\ \hline 
\widetilde \sigma_1  & \sigma_1 & \sigma_2 & \sigma_3     \\
\bl{\widetilde \sigma_2} & \bl{\sigma_2}  & \bl{\sigma_3} & \bl{\sigma_3}  \\
 \ro{\widetilde \sigma_3}   & \ro{\sigma_3}   & \ro{\sigma_3} & \ro{\sigma_3}
\end{array}
\hspace{.5cm}
\begin{array}{c|cccc} \widetilde{\rm \bf B} & \sigma_1 & \sigma_2& \sigma_3& \sigma_4\\ \hline 
\widetilde \sigma_1  & \sigma_1 & \sigma_2 & \sigma_3  & \sigma_4     \\
\bl{\widetilde \sigma_2} & \bl{\sigma_2} &  \bl{\sigma_4} & \bl{\sigma_4} & \bl{\sigma_4}   \\
 \ro{\widetilde \sigma_3}   & \ro{\sigma_3} & \ro{\sigma_3} & \ro{\sigma_3}& \ro{\sigma_3}   
\end{array}
\hspace{.5cm}
\begin{array}{c|ccccc} \widetilde{\rm \bf C} & \sigma_1 & \sigma_2& \sigma_3& \sigma_4& \sigma_5\\ \hline 
\widetilde \sigma_1  & \sigma_1 & \sigma_2 & \sigma_3  & \sigma_4 & \sigma_5    \\
\bl{\widetilde \sigma_2} & \bl{\sigma_2} &  \bl{\sigma_4} & \bl{\sigma_5} & \bl{\sigma_4} & \bl{\sigma_4}  \\
 \ro{\widetilde \sigma_3}   & \ro{\sigma_3} & \ro{\sigma_3} & \ro{\sigma_3}& \ro{\sigma_3} & \ro{\sigma_3} 
\end{array}
\hspace{.1cm} .
\end{align}
The graphs of the fundamental networks $\widetilde {\rm \bf A}, \widetilde {\rm \bf B}$ and $\widetilde {\rm \bf C}$ are depicted in Figure \ref{pictfundamental}. The figure also displays the differential equations $\dot X = \gamma_f^{\widetilde {\rm\bf A}}(X)$,  $\dot X = \gamma_f^{\widetilde {\rm\bf B}}(X)$ and  $\dot X = \gamma_f^{\widetilde {\rm\bf C}}(X)$.
\vspace{-.4cm}
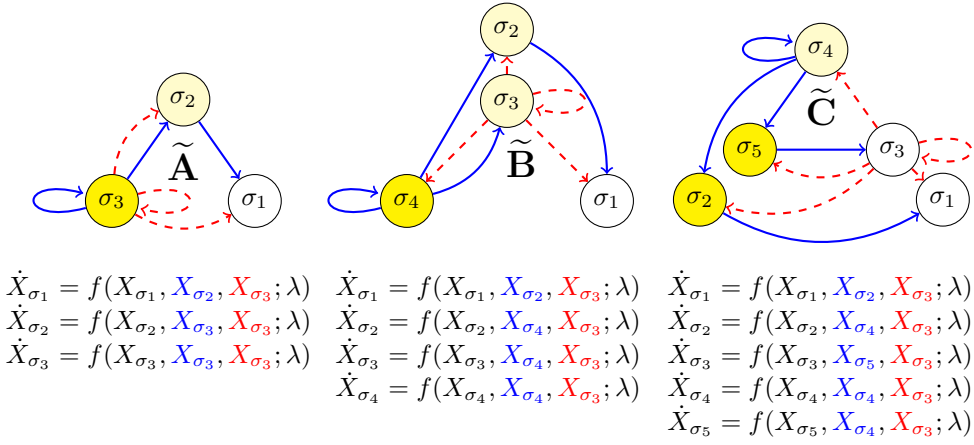
\begin{figure}[h]\renewcommand{\figurename}{\rm \bf \footnotesize Figure} 
\begin{center}
\hspace{.4cm}
\begin{tabular}{p{4cm}p{4.5cm}p{4.4cm}} \hspace{.3cm}
 \begin{tikzpicture}[->, scale=1.9]
	  \tikzstyle{vertextype1}=[circle, draw, minimum size=20pt,inner sep=1pt]
	 \tikzstyle{vertextype2} = [circle, draw, fill=yellow, minimum size=20pt,inner sep=1pt]
	 \tikzstyle{vertextype3} = [circle, draw, fill=yellow!25, minimum size=20pt,inner sep=1pt]
	 \tikzstyle{edgetype1} = [->, draw,line width=3pt,-,red!50]
	 \tikzstyle{edgetype2} = [draw,thick,-,black]
	  \node at (1,1.4) {};
	 \node at (1,-.3) {};
	  \node at (1.5,.28) {\Large {$\widetilde {\rm \bf A}$}};
	 \node[vertextype2] (v3) at (1,0) {$\sigma_3$};
	 \node[vertextype3] (v2) at (1.5, .71) {$\sigma_2$};
	 \node[vertextype1] (v1) at (2,0) {$\sigma_1$};
	\path[]
	(v2) edge [thick, blue] node {} (v1)
	(v3) edge [loop left, blue, thick] node {} (v3)
	(v3) edge [thick, blue] node {} (v2)
	(v3) edge [loop right, red, dashed, thick] node {} (v3)
	(v3) edge [thick, bend left, red, dashed] node {} (v2)
	(v3) edge [thick, dashed, bend right, red] node {} (v1);
 \end{tikzpicture}
& \hspace{-.2cm}
\begin{tikzpicture}[->, scale=1.9]
	 \tikzstyle{vertextype1}=[circle, draw, minimum size=20pt,inner sep=1pt]
	 \tikzstyle{vertextype2} = [circle, draw, fill=yellow, minimum size=20pt,inner sep=1pt]
	 \tikzstyle{vertextype3} = [circle, draw, fill=yellow!25, minimum size=20pt,inner sep=1pt]
	 \tikzstyle{edgetype1} = [->, draw,line width=3pt,-,red!50]
	 \tikzstyle{edgetype2} = [draw,thick,-,black]
	  \node at (.6,.3) {\Large {$\widetilde {\rm \bf B}$}};
	 \node at (1,1.4) {};
	 \node at (1,-.3) {};
	 \node[vertextype3] (v2) at (.5,1.2) {$\sigma_2$};
	 \node[vertextype1] (v1) at (1.2,0) {$\sigma_1$};
	  \node[vertextype3] (v3) at (.5,.7) {$\sigma_3$};
	 \node[vertextype2] (v4) at (-.2,0) {$\sigma_4$};
	\path[]
	(v2) edge [bend left, blue, thick] node {} (v1)
	(v4) edge [thick, blue] node {} (v2)
	(v4) edge [bend right, thick, blue] node {} (v3)
	(v4) edge [loop left, thick, blue] node {} (v4)
	(v3) edge [ red, thick, dashed] node {} (v1)
	(v3) edge [red, thick, dashed] node {} (v2)
	(v3) edge [loop right, red, thick, dashed] node {} (v3)
	(v3) edge [red, thick, dashed] node {} (v4);
 \end{tikzpicture}%
  & \hspace{-.3cm}
 \begin{tikzpicture}[->, scale=1.9]
	 \tikzstyle{vertextype1}=[circle, draw, minimum size=20pt,inner sep=1pt]
	 \tikzstyle{vertextype2} = [circle, draw, fill=yellow, minimum size=20pt,inner sep=1pt]
	 \tikzstyle{vertextype3} = [circle, draw, fill=yellow!25, minimum size=20pt,inner sep=1pt]
	 \tikzstyle{edgetype1} = [->, draw,line width=3pt,-,red!50]
	 \tikzstyle{edgetype2} = [draw,thick,-,black]
	 \node at (1,1.4) {};
	 \node at (1,-.3) {};
	  \node at (1,.65) {\Large {$\widetilde {\rm \bf C}$}};
	 \node[vertextype1] (v1) at (1.85,0) {$\sigma_1$};
	 \node[vertextype2] (v2) at (.15,0) {$\sigma_2$};
	 \node[vertextype1] (v3) at (1.5,.35) {$\sigma_3$};
	  \node[vertextype3] (v4) at (1,1.05) {$\sigma_4$};
	 \node[vertextype2] (v5) at (.5,.35) {$\sigma_5$};
	\path[]
	(v2) edge [bend right, blue, thick] node {} (v1)
	(v4) edge [thick, bend right, blue] node {} (v2)
	(v5) edge [thick, blue] node {} (v3)
	(v4) edge [loop left, thick, blue] node {} (v4)
	(v4) edge [thick, blue] node {} (v5)
	(v3) edge [red, thick, dashed] node {} (v1)
	(v3) edge [red, thick, bend left, dashed] node {} (v2)
	(v3) edge [loop right, red, thick, dashed] node {} (v3)
	(v3) edge [red, thick, dashed] node {} (v4)
	(v3) edge [bend left, red, thick, dashed] node {} (v5);
 \end{tikzpicture} 
 \end{tabular}\\
 \hspace{-.75cm}
 \begin{tabular}{p{4.3cm}p{3.7cm}p{4.2cm}}
 \vspace{-9mm} \hspace{0mm}
 \begin{equation}\hspace{0mm}  \nonumber  
 \begin{array}{l}
 \dot X_{\sigma_1} = f(X_{\sigma_1}, \bl{X_{\sigma_2}}, \ro{X_{\sigma_3}}; \lambda) \\
 \dot X_{\sigma_2} = f(X_{\sigma_2}, \bl{X_{\sigma_3}}, \ro{X_{\sigma_3}}; \lambda)\\
\dot X_{\sigma_3} = f(X_{\sigma_3}, \bl{X_{\sigma_3}}, \ro{X_{\sigma_3}}; \lambda) 
 \end{array}
 \hspace{-.8cm}
 \end{equation} & \vspace{-9mm}
 \hspace{-2mm}
  \begin{equation}\hspace{0mm}  \nonumber 
 \begin{array}{l}
 \dot X_{\sigma_1} = f(X_{\sigma_1}, \bl{X_{\sigma_2}}, \ro{X_{\sigma_3}}; \lambda)\\
 \dot X_{\sigma_2} = f(X_{\sigma_2}, \bl{X_{\sigma_4}}, \ro{X_{\sigma_3}}; \lambda) \\
  \dot X_{\sigma_3} = f(X_{\sigma_3}, \bl{X_{\sigma_4}}, \ro{X_{\sigma_3}}; \lambda) \\
   \dot X_{\sigma_4} = f(X_{\sigma_4}, \bl{X_{\sigma_4}}, \ro{X_{\sigma_3}}; \lambda) 
 \end{array}
 \hspace{-.5cm}
 \end{equation} 
 & \vspace{-.9cm}
 \hspace{-2mm}
   \begin{equation}\hspace{0mm}    \nonumber 
 \begin{array}{l}   
  \dot X_{\sigma_1} = f(X_{\sigma_1}, \bl{ X_{\sigma_2}}, \ro{X_{\sigma_3}}; \lambda) \\
 \dot X_{\sigma_2} = f(X_{\sigma_2}, \bl{X_{\sigma_4}}, \ro{X_{\sigma_3}}; \lambda) \\
\dot X_{\sigma_3} = f(X_{\sigma_3}, \bl{X_{\sigma_5}}, \ro{X_{\sigma_3}}; \lambda)\\ 
\dot X_{\sigma_4} = f(X_{\sigma_4}, \bl{X_{\sigma_4}}, \ro{X_{\sigma_3}}; \lambda)\\
\dot X_{\sigma_5} = f(X_{\sigma_5}, \bl{X_{\sigma_4}}, \ro{X_{\sigma_3}}; \lambda)
 \end{array}
 \hspace{-.8cm}
 \end{equation} 
 \end{tabular}\vspace{-.6cm}
 \caption{\footnotesize {\rm The fundamental networks of {\bf A}, {\bf B} and {\bf C} and their equations of motion.}}
 \vspace{-.6cm}
\label{pictfundamental}
\end{center}
 \end{figure}\\
\mbox{}
 \\ \indent
We note that network $\widetilde {\rm \bf A}$ is isomorphic to network {\bf A}. One may also observe that network {\bf B} is isomorphic to a quotient of network $\widetilde {\rm\bf B}$ and that network {\bf C} is isomorphic to a quotient of network $\widetilde {\rm\bf C}$. We show below that this is not a coincidence.
\end{example}

\noindent The following result clarifies the relation between a homogeneous network and its fundamental network.
 
 \begin{theorem}\label{fundamentaltheorem}
 Every homogeneous network ${\rm \bf N}=\{A\rightrightarrows_t^s V\}$ is a quotient of its fundamental network $\widetilde {\rm \bf N}$. 
 More precisely, for every vertex $v\in V$ of ${\rm \bf N}$, the map of vertices
$$\phi_v: \Sigma_{\rm\bf N} \to V \ \mbox{defined by}\ \phi_v(\sigma_j) := \sigma_j(v)$$
extends to a graph fibration from $\widetilde {\rm \bf N}$ to ${\rm \bf N}$. 
In particular, the map $\phi_v^*: E_{\rm\bf N} \to  E_{\widetilde {\rm\bf N}}$ defined by $(\phi_v^*x)_{\sigma_j}(t):= x_{\phi_v(\sigma_j)}=x_{\sigma_j(v)}$ conjugates the network maps $\gamma_f^{\rm\bf N}$ and $\gamma_f^{\widetilde {\rm\bf N}}$, that is
$$\phi_v^* \circ \gamma_f^{\rm\bf N}=\gamma_f^{\widetilde {\rm\bf N}}\circ \phi_v^*\, .$$
 \end{theorem}
 \begin{proof}
 It follows from the definition of $\widetilde \sigma_k$ that
 $$\sigma_k(\phi_v(\sigma_j)) = \sigma_k(\sigma_j(v)) = (\sigma_k \circ \sigma_j)(v) = (\widetilde \sigma_k(\sigma_j))(v) = \phi_v(\widetilde \sigma_k(\sigma_j))\, .$$
 This shows that $\sigma_k\circ \phi_v=\phi_v\circ \widetilde \sigma_k$ and thus by Proposition \ref{obviousprop} that $\phi_v$ extends to a graph fibration. The remaining statements  follow from Theorem \ref{devilletheorem1}.
 \end{proof} 
The image of the map $\phi_v$ of Theorem \ref{fundamentaltheorem} is equal to the subset $\{\sigma_j(v)\, |\, \sigma_j\in \Sigma_{\rm \bf N}\}$ of the vertices of ${\rm \bf N}$. Recall from Remark \ref{semigroupremark} that this set consists of all the direct and indirect inputs of cell $v$. In general, we   define the {\it input network} ${\rm \bf N}_{(v)}=\{A_{(v)}\rightrightarrows_t^s V_{(v)}\}$ of a cell $v$ in an arbitrary (i.e. not necessarily homogeneous) network ${\rm \bf N}=\{A\rightrightarrows_t^s V\}$ by
 $$V_{(v)} := \{w\in  V\, |\, \exists \ \mbox{path}\ \mbox{in} \ {\rm \bf N}\ \mbox{from}\ w \ \mbox{to}\ v\, \}\ \mbox{and}\ A_{(v)}:= \{a\in A\,|\, t(a)\in V_{(v)}\, \}\, .$$
This input network consists of those cells that can be ``felt'' by cell $v$, either directly or indirectly. In fact, it automatically holds that $s(a)\in V_{(v)}$ for all arrows $a\in A_{(v)}$. Hence, ${\rm \bf N}_{(v)}$ is a subnetwork of ${\rm \bf N}$ and the embedding 
$$e_{{\rm \bf N}_{(v)}}: {\rm \bf N}_{(v)} \to {\rm \bf N}$$  
is an injective graph fibration. Theorem \ref{fundamentaltheorem} can now be rephrased as follows:
\begin{corollary}\label{embeddedcor}
The dynamics of the input network ${\rm\bf N}_{(v)}$ of every cell $v$ of ${\rm\bf N}$ is embedded as the robust synchrony space 
$${\rm Syn}_{P_{(v)}}:=\{X\in E_{\widetilde {\rm\bf N}}\, |\, X_{\sigma_j}=X_{\sigma_k}\, \mbox{when}\ \sigma_j(v)=\sigma_k(v)\, \}$$
inside the dynamics of the fundamental network $\widetilde {\rm\bf N}$. \end{corollary}
\begin{proof}
Theorem \ref{fundamentaltheorem} implies that $\phi_{v}:\widetilde {\rm \bf N}\to{\rm\bf N}_{(v)}$ is a surjective graph fibration for every cell $v$ of ${\rm\bf N}$. By Proposition \ref{injsurjremark}, the linear map
$$\phi_{v}^*: E_{{\rm\bf N}_{(v)}} \to E_{\widetilde {\rm\bf N}}\ \mbox{defined by}\ (\phi_v^*x)_{\sigma_j} = x_{\sigma_j(v)}$$
is therefore injective. By Theorem \ref{fundamentaltheorem}, it thus embeds the dynamics of $\gamma_f^{{\rm \bf N}_{(v)}}$ inside the dynamics of $\gamma_f^{\widetilde {\rm\bf N}}$. It is clear that ${\rm im}\, \phi_v^*={\rm Syn}_{P_{(v)}}$.

One readily checks that the partition $P_{(v)}$ of $\Sigma_{\rm\bf N}$ for which 
$$\sigma_j \sim_{P_{(v)}} \sigma_k\ \mbox{if and only if}\ \sigma_j(v)=\sigma_k(v)$$ 
is balanced. Indeed, if $\sigma_j \sim_{P_{(v)}} \sigma_k$, then we have for every input map $\widetilde \sigma_l$ of $\widetilde {\rm\bf N}$ that
$(\widetilde \sigma_l(\sigma_j))(v) = (\sigma_l\circ\sigma_j)(v)=\sigma_l(\sigma_j(v))=\sigma_l(\sigma_k(v)) = (\sigma_l\circ\sigma_k)(v)=(\widetilde \sigma_l(\sigma_k))(v)$, and hence that $\widetilde \sigma_l(\sigma_j)\sim_{P^{(v)}} \widetilde \sigma_l(\sigma_k)$. Thus $P_{(v)}$ is balanced.

Alternatively, one may recall from Theorem \ref{fundamentaltheorem} that $\gamma_f^{\widetilde{\rm\bf N}}\circ\phi_v^*=\phi_v^*\circ\gamma_f^{\rm\bf N}$ for any response function $f$. This implies in particular that
$$\gamma_f^{\widetilde{\rm\bf N}}({\rm im}\, \phi_v^*) \subset {\rm im}\, \phi_v^*\, $$
and hence that ${\rm im}\, \phi_v^*$ is invariant under the dynamics of $\gamma_f^{\widetilde {\rm\bf N}}$. 
\end{proof}
\begin{remark}\label{restrictionremark}
Identifying $E_{{\rm \bf N}_{(v)}}$ with the synchrony space ${\rm Syn}_{P_{(v)}} \subset E_{\widetilde {\rm\bf N}}$ by means of the embedding $\phi_v^*$, we can also write the identity 
$\phi_v^* \circ \gamma_f^{\rm\bf N}=\gamma_f^{\widetilde {\rm\bf N}}\circ \phi_v^*$ as 
$$\gamma_f^{{\rm \bf N}_{(v)}} =  \gamma_f^{\widetilde {\rm\bf N}}|_{E_{{\rm\bf N}_{(v)}}}\, .$$
In other words, we may think of the dynamics of ${\rm \bf N}_{(v)}$ as the restriction to a synchrony subspace of the dynamics of $\widetilde {\rm \bf N}$.
\end{remark}
\begin{remark}
The dynamics of ${\rm \bf N}$ is itself embedded in the dynamics of $\widetilde {\rm \bf N}$ if there is a cell $v$ in {\bf N} so that ${\rm \bf N}_{(v)} = {\rm \bf N}$. It is natural to assume that such a cell exists: otherwise, the network may be considered pathological, or at least quite  irrelevant for our understanding of network dynamics.
\end{remark}
 \begin{remark}\label{allfibrations}
 Theorem \ref{fundamentaltheorem} shows that for every cell $v$ in the homogeneous network ${\rm \bf N}$, there is a graph fibration $\phi_v:\widetilde {\rm \bf N} \to {\rm \bf N}$ that sends cell $\sigma_1$ of $\widetilde {\rm \bf N}$ (representing the unit of $\Sigma_{\rm \bf N}$) to cell $v$ of network {\bf N}. On the other hand, there is only one graph fibration $\phi:\widetilde {\rm \bf N}\to {\rm \bf N}$ that maps cell $\sigma_1$ of $\widetilde {\rm \bf N}$ to cell $v$ of ${\rm \bf N}$, because if $\phi(\sigma_1)=v$, then $\phi(\sigma_k) = \phi(\widetilde \sigma_k(\sigma_1)) = \sigma_k(\phi(\sigma_1)) = \sigma_k(v)$. So Theorem \ref{fundamentaltheorem} in fact describes all possible graph fibrations from $\widetilde {\rm \bf N}$ to ${\rm \bf N}$.
 \end{remark}
\begin{example}\label{realisation}
Our networks ${\rm \bf A}, {\rm \bf B}$ and ${\rm \bf C}$ are themselves input networks of one or more of their cells. For example, ${\rm\bf A}={\rm\bf A}_{(v_3)}$, ${\rm\bf B}={\rm\bf B}_{(v_3)}$ and ${\rm\bf C}={\rm\bf C}_{(v_3)}$, so the networks are quotients of their fundamental networks. The corresponding graph fibrations are shown in Figure \ref{pictfundamental1}. For example, the graph fibration $\phi_{v_3}: \widetilde {\rm\bf C}\to  {\rm\bf C}$ sends 
$$ (\sigma_1, \sigma_2, \sigma_3, \sigma_4, \sigma_5)\mapsto (v_3, v_1, v_3, v_2, v_1)\, .$$
This means that when $(x_{v_1}(t), x_{v_2}(t), x_{v_3}(t))$ solves the equations of network ${\rm\bf C}$, then 
$$(X_{\sigma_1}(t), X_{\sigma_2}(t), X_{\sigma_3}(t), X_{\sigma_4}(t), X_{\sigma_5}(t)) = (x_{v_3}(t), x_{v_1}(t), x_{v_3}(t), x_{v_2}(t), x_{v_1}(t))$$
solves those of network $\widetilde {\rm \bf C}$. Network ${\rm \bf C}$ is therefore embedded inside network $\widetilde {\rm \bf C}$ as the robust synchrony space $\{X_{\sigma_1}=X_{\sigma_3}, X_{\sigma_2}=X_{\sigma_5}\}$. Similarly, network ${\rm \bf B}$ is realised inside $\widetilde {\rm \bf B}$ as the robust synchrony space $\{X_{\sigma_2}=X_{\sigma_3}\}$. Finally, the dynamics of networks ${\rm \bf A}$ and $\widetilde {\rm\bf A}$ are bi-conjugate because $\phi_{v_3}: \widetilde{\rm \bf A} \to  {\rm\bf A}$ is an isomorphism. 
\end{example}
\vspace{-.4cm}
\begin{figure}[h]\renewcommand{\figurename}{\rm \bf \footnotesize Figure} 
\begin{center}
\hspace{0cm}
\begin{tabular}{p{4cm}p{4.2cm}p{4.4cm}}
 \begin{tikzpicture}[->, scale=1.9]
	  \tikzstyle{vertextype1}=[circle, draw, minimum size=20pt,inner sep=1pt]
	 \tikzstyle{vertextype2} = [circle, draw, fill=yellow, minimum size=20pt,inner sep=1pt]
	 \tikzstyle{vertextype3} = [circle, draw, fill=yellow!25, minimum size=20pt,inner sep=1pt]
	 \tikzstyle{edgetype1} = [->, draw,line width=3pt,-,red!50]
	 \tikzstyle{edgetype2} = [draw,thick,-,black]
	    \node at (1,1.3) {};
	 \node at (1,-2.8) {};
	  \node at (1.5,.28) {\Large {$\widetilde {\rm \bf A}$}};
	 \node[vertextype2] (v3) at (1,0) {$\sigma_3$};
	 \node[vertextype3] (v2) at (1.5, .71) {$\sigma_2$};
	 \node[vertextype1] (v1) at (2,0) {$\sigma_1$};
	    \node at (1.5,-2.35) {\Large {\bf A}};
	 \node[vertextype2] (v1a) at (1,-2.6) {$v_1$};
	 \node[vertextype3] (v2a) at (1.5, -1.89) {$v_2$};
	 \node[vertextype1] (v3a) at (2,-2.6) {$v_3$};
	  \draw [->,decorate,decoration=snake, thick] (1.5,-.4) -- (1.5,-1.6) node [above, midway, sloped]  {{\tiny $\sigma_1\mapsto v_3,\sigma_2\mapsto v_2 $}} node [below, midway, sloped] {{\tiny $\sigma_3\mapsto v_1$}};
	\path[]
	(v2) edge [thick, blue] node {} (v1)
	(v3) edge [loop left, blue, thick] node {} (v3)
	(v3) edge [thick, blue] node {} (v2)
	(v3) edge [loop right, red, dashed, thick] node {} (v3)
	(v3) edge [thick, bend left, red, dashed] node {} (v2)
	(v3) edge [thick, dashed, bend right, red] node {} (v1)
	(v2a) edge [thick, blue] node {} (v3a)
	(v1a) edge [loop left, blue, thick] node {} (v1a)
	(v1a) edge [thick, blue] node {} (v2a)
	(v1a) edge [loop right, red, dashed, thick] node {} (v1a)
	(v1a) edge [thick, bend left, red, dashed] node {} (v2a)
	(v1a) edge [thick, dashed, bend right, red] node {} (v3a);
 \end{tikzpicture} 
& \hspace{-.6cm}
\begin{tikzpicture}[->, scale=1.9]
	 \tikzstyle{vertextype1}=[circle, draw, minimum size=20pt,inner sep=1pt]
	 \tikzstyle{vertextype2} = [circle, draw, fill=yellow, minimum size=20pt,inner sep=1pt]
	 \tikzstyle{vertextype3} = [circle, draw, fill=yellow!25, minimum size=20pt,inner sep=1pt]
	 \tikzstyle{edgetype1} = [->, draw,line width=3pt,-,red!50]
	 \tikzstyle{edgetype2} = [draw,thick,-,black]
	    \node at (1,1.3) {};
	 \node at (1,-2.8) {};
	  \node at (.6,.3) {\Large {$\widetilde {\rm \bf B}$}};
	 \node[vertextype3] (v2) at (.5,1.2) {$\sigma_2$};
	 \node[vertextype1] (v1) at (1.2,0) {$\sigma_1$};
	  \node[vertextype3] (v3) at (.5,.7) {$\sigma_3$};
	 \node[vertextype2] (v4) at (-.2,0) {$\sigma_4$};
	   \node at (.4,-2.35) {\Large {\bf B}};
	 \node[vertextype1] (v3a) at (.9,-2.6) {$v_3$};
	 \node[vertextype3] (v2a) at (.4, -1.89) {$v_2$};
	 \node[vertextype2] (v1a) at (-.1,-2.6) {$v_1$};
	  \draw [->,decorate,decoration=snake, thick] (0.4,-.4) -- (0.4,-1.6) node [above, midway, sloped]  {{\tiny $\sigma_4\mapsto v_1,\sigma_1\mapsto v_3 $}} node [below, midway, sloped] {{\tiny $\sigma_2,\sigma_3\mapsto v_2$}};
	\path[]
	(v2) edge [bend left, blue, thick] node {} (v1)
	(v4) edge [thick, blue] node {} (v2)
	(v4) edge [bend right, thick, blue] node {} (v3)
	(v4) edge [loop left, thick, blue] node {} (v4)
	(v3) edge [ red, thick, dashed] node {} (v1)
	(v3) edge [red, thick, dashed] node {} (v2)
	(v3) edge [loop right, red, thick, dashed] node {} (v3)
	(v3) edge [red, thick, dashed] node {} (v4)
	(v1a) edge [loop left, thick, blue] node {} (v1a)
	(v1a) edge [blue, thick] node {} (v2a)
	(v2a) edge [thick, blue] node {} (v3a)
	(v2a) edge [bend left, red, dashed, thick] node {} (v3a)
	(v2a) edge [thick, red, loop right, dashed] node {} (v2a)
	(v2a) edge [bend right, red, thick, dashed] node {} (v1a);
 \end{tikzpicture}
 & 
 \hspace{-.4cm} 
 \begin{tikzpicture}[->, scale=1.9]
	 \tikzstyle{vertextype1}=[circle, draw, minimum size=20pt,inner sep=1pt]
	 \tikzstyle{vertextype2} = [circle, draw, fill=yellow, minimum size=20pt,inner sep=1pt]
	 \tikzstyle{vertextype3} = [circle, draw, fill=yellow!25, minimum size=20pt,inner sep=1pt]
	 \tikzstyle{edgetype1} = [->, draw,line width=3pt,-,red!50]
	 \tikzstyle{edgetype2} = [draw,thick,-,black]
	  \node at (1,1.3) {};
	 \node at (1,-2.8) {};
	  \node at (.5,.65) {\Large {$\widetilde {\rm \bf C}$}};
	 \node[vertextype1] (v1) at (1.35,0) {$\sigma_1$};
	 \node[vertextype2] (v2) at (-.35,0) {$\sigma_2$};
	 \node[vertextype1] (v3) at (1,.35) {$\sigma_3$};
	  \node[vertextype3] (v4) at (.5,1.05) {$\sigma_4$};
	 \node[vertextype2] (v5) at (0,.35) {$\sigma_5$};
	  \node at (.4,-2.35) {\Large {\bf C}};
	 \node[vertextype1] (v3a) at (.9,-2.6) {$v_3$};
	 \node[vertextype3] (v2a) at (.4, -1.89) {$v_2$};
	 \node[vertextype2] (v1a) at (-.1,-2.6) {$v_1$};
	  \draw [->,decorate,decoration=snake, thick] (0.4,-.4) -- (0.4,-1.6) node [above, midway, sloped] {{\tiny $\sigma_2,\sigma_5\mapsto v_1, \sigma_4\mapsto v_2$}} node [below, midway, sloped] {{\tiny $\sigma_1,\sigma_3\mapsto v_3$}};
	\path[]
	(v2) edge [bend right, blue, thick] node {} (v1)
	(v4) edge [thick, bend right, blue] node {} (v2)
	(v5) edge [thick, blue] node {} (v3)
	(v4) edge [loop left, thick, blue] node {} (v4)
	(v4) edge [thick, blue] node {} (v5)
	(v3) edge [red, thick, dashed] node {} (v1)
	(v3) edge [red, thick, bend left, dashed] node {} (v2)
	(v3) edge [loop right, red, thick, dashed] node {} (v3)
	(v3) edge [red, thick, dashed] node {} (v4)
	(v3) edge [bend left, red, thick, dashed] node {} (v5)
	(v1a) edge [thick, blue] node {} (v3a)
	(v2a) edge [loop left, blue, thick] node {} (v2a)
	(v2a) edge [thick, blue] node {} (v1a)
	(v3a) edge [loop right, red, dashed, thick] node {} (v3a)
	(v3a) edge [thick, red, dashed] node {} (v2a)
	(v3a) edge [bend left, red, thick, dashed] node {} (v1a);   
 \end{tikzpicture} 
 \end{tabular}
 \vspace{-.3cm}
 \caption{\footnotesize {\rm The graph fibrations $\phi_{v_3}: \widetilde {\rm \bf A} \to {\rm \bf A}$ and $\phi_{v_3}: \widetilde {\rm \bf B} \to {\rm \bf B}$ and $\phi_{v_3}: \widetilde {\rm \bf C} \to {\rm \bf C}$.}}
 \vspace{-.6cm}
\label{pictfundamental1}
\end{center}
 \end{figure}
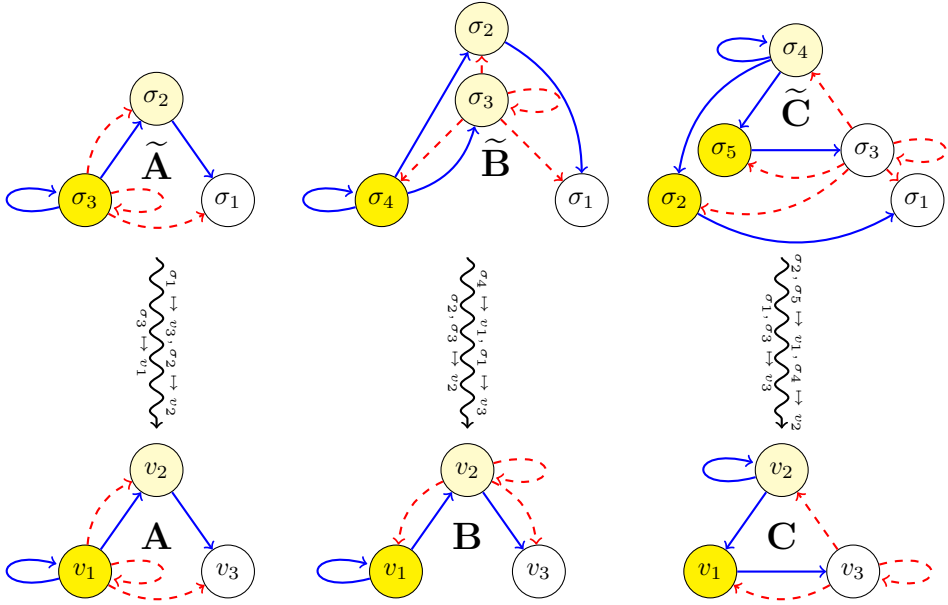

\section{Hidden symmetry}\label{hiddensection}
In this section, we prove the main results of this paper. We start with an observation.
\begin{lemma}\label{homomorphismlemma}
 The fundamental network ${\widetilde {\widetilde {\rm\bf N}}}$ of a fundamental network $\widetilde {\rm \bf N}$ is isomorphic to $\widetilde {\rm \bf N}$. 
\end{lemma}
 \begin{proof} 
Recall that the vertex set of $\widetilde {\rm\bf N}$ is the semigroup $\Sigma_{\rm\bf N}=\{\sigma_1, \ldots, \sigma_n\}$, and that $\widetilde {\rm\bf N}$ has input maps  $\widetilde \sigma_1, \ldots, \widetilde \sigma_m$ defined by $\widetilde \sigma_j(\sigma_k)=\sigma_j\circ\sigma_k$. Consequently, the vertex set of ${\widetilde {\widetilde {\rm\bf N}}}$ is the semigroup $\Sigma_{\widetilde {\rm \bf N}}$ generated by $\widetilde \sigma_1,\ldots, \widetilde \sigma_m$, while ${\widetilde {\widetilde {\rm\bf N}}}$ has input maps ${\widetilde {\widetilde \sigma}}_1, \ldots, {\widetilde {\widetilde \sigma}}_m$ defined by ${\widetilde {\widetilde \sigma}}_j(\widetilde \sigma_k) = \widetilde \sigma_j \circ \widetilde \sigma_k$. We claim that $\Sigma_ {\rm \bf N}$ and $\Sigma_{\widetilde {\rm \bf N}}$ are isomorphic semigroups,   which implies the lemma. To prove our claim, simply note that $(\widetilde \sigma_k\circ\widetilde \sigma_j)(\sigma_i) = \sigma_k\circ\sigma_j\circ\sigma_i = \widetilde {(\sigma_k\circ \sigma_j)}(\sigma_i)$, i.e. 
 $$\widetilde \sigma_k\circ\widetilde \sigma_j= \widetilde {\sigma_k\circ \sigma_j}\, .$$
 Because $\Sigma_{\rm\bf N}$ is the smallest semigroup containing $\sigma_1, \ldots, \sigma_m$, this observation implies that 
 $$\phi: \sigma_j\mapsto \widetilde \sigma_j$$ defines a surjective homomorphism from $\Sigma_{\rm \bf N}$ to $\Sigma_{\widetilde {\rm \bf N}}$. Moreover, $\phi$ is injective, because $\Sigma_{\rm\bf N}$ contains a unit $\sigma_1$, so that if $\widetilde \sigma_j = \widetilde \sigma_k$, then $\sigma_j= \sigma_j \circ \sigma_1 =\widetilde \sigma_j(\sigma_1)=\widetilde \sigma_k(\sigma_1)= \sigma_k\circ \sigma_1= \sigma_k$. We conclude that $\phi$ is an isomorphism and, in particular, a bijection between the vertices of $\widetilde {\rm\bf N}$ and those of $\widetilde {\widetilde {\rm \bf N}}$. 
It is also clear that $\phi$ intertwines the input maps of $\widetilde {\rm\bf N}$ and ${\widetilde {\widetilde {\rm\bf N}}}$, since
$${\widetilde {\widetilde \sigma}}_k(\phi( \sigma_j)) = {\widetilde {\widetilde \sigma}}_k(\widetilde \sigma_j) = \widetilde \sigma_k\circ \widetilde \sigma_j  = \widetilde {\sigma_k\circ \sigma_j}  = \phi(\widetilde \sigma_k(\sigma_j))\, .$$
This proves that $\phi$ extends to an isomorphism between $\widetilde {\rm\bf N}$ and ${\widetilde {\widetilde { \rm\bf N}}}$.
\end{proof}
Combining Theorem \ref{fundamentaltheorem} and Lemma \ref{homomorphismlemma}, we obtain:
\begin{theorem}\label{symmetrytheorem}
Let $\widetilde {\rm \bf N}$ be a homogeneous fundamental network. For all $1\leq i\leq n$, the map
$$\phi_{\sigma_i}: \Sigma_{\rm\bf N} \to \Sigma_{\rm\bf N}\ \mbox{defined by}\ \phi_{\sigma_i}(\sigma_j):=\sigma_j\circ \sigma_i$$
extends to a graph fibration from $\widetilde {\rm \bf N}$ to itself. Every network map $\gamma_f^{\widetilde {\rm\bf N}}$ is thus $\Sigma_{\rm\bf N}$-equivariant:
$$\phi_{\sigma_i}^* \circ  \gamma_f^{\widetilde {\rm\bf N}}= \gamma_f^{\widetilde {\rm\bf N}} \circ \phi_{\sigma_i}^*\ \mbox{for all}\ \sigma_i\in \Sigma_{\rm \bf N}\, ,$$
where we recall that $\phi_{\sigma_i}^*:E_{\widetilde {\rm \bf N}}\to E_{\widetilde {\rm \bf N}}$ is defined by $(\phi_{\sigma_i}^*X)_{\sigma_j}:=X_{\phi_{\sigma_i}(\sigma_j)}=X_{\sigma_j\circ\sigma_i}$. 
\end{theorem}
\begin{proof}
The statement of this theorem is a special case of the statement of Theorem \ref{fundamentaltheorem}, with ${\rm \bf N}$ replaced by $\widetilde {\rm \bf N}$ and  $\widetilde {\rm \bf N}$ replaced by $\widetilde {\widetilde {\rm \bf N}}$, noting that the latter is isomorphic to $ \widetilde {\rm \bf N}$. 
 
 Alternatively, from the fact that left-multiplication and right-multiplication in $\Sigma_{\rm \bf N}$  commute, it also follows directly that every $\phi_{\sigma_i}$ commutes with every input map $\widetilde\sigma_k$ of $\widetilde {\rm \bf N}$: 
$$\phi_{\sigma_i}(\widetilde\sigma_k(\sigma_j)) = \sigma_k\circ\sigma_j\circ\sigma_i =\widetilde \sigma_k (\phi_{\sigma_i}(\sigma_j))\, . $$
By Proposition \ref{obviousprop} it thus follows that $\phi_{\sigma_i}$ extends to a graph fibration. The remaining statements of the theorem now follow from Theorem \ref{fundamentaltheorem}.
\end{proof}
\begin{remark} The maps $\phi_{\sigma_i}: \widetilde {\rm \bf N} \to \widetilde {\rm \bf N}$ of Theorem \ref{symmetrytheorem} are examples of graph fibrations from a network graph to itself. We shall refer to such graph fibrations as {\it self-fibrations}. The self-fibrations of a network need not form a group. For example, because the right-multiplication by $\sigma_i$ in $\Sigma_{\rm\bf N}$ need not be an invertible operation, the self-fibrations $\phi_{\sigma_i}$ need not be invertible. Nevertheless, because the composition of two graph fibrations is obviously a graph fibration, the self-fibrations of a network form a semigroup with unit. 
\end{remark}
\begin{remark} Recall from Remark \ref{allfibrations} that Theorem \ref{symmetrytheorem} describes all the self-fibrations of a fundamental network. 
It clearly holds that $(\phi_{\sigma_k}\circ\phi_{\sigma_j})(\sigma_i) = \sigma_i\circ\sigma_j\circ\sigma_k = \phi_{\sigma_j\circ\sigma_k}(\sigma_i)$, i.e.
$$\phi_{\sigma_k}\circ\phi_{\sigma_j} =  \phi_{\sigma_j\circ\sigma_k}\, .$$
This contravariant transformation formula shows that the self-fibrations of a fundamental network $\widetilde {\rm \bf N}$ form a semigroup that is isomorphic to $\Sigma_{\rm \bf N}^*$, the so-called {\it opposite semigroup} of $\Sigma_{\rm \bf N}$, with product $\sigma_j*\sigma_k:=\sigma_k\circ\sigma_j$.
\end{remark}
\begin{remark}
On the other hand, Remark \ref{contravariant} implies the covariant transformation formula
$$\phi_{\sigma_j}^* \circ \phi_{\sigma_k}^* = (\phi_{\sigma_k}\circ \phi_{\sigma_j})^* = \phi_{\sigma_j\circ\sigma_k}^*$$
 In particular, the assignment
$$\sigma_j\mapsto \phi_{\sigma_j}^* \ \mbox{from}\ \Sigma_{\rm \bf N} \ \mbox{to}\ \mathfrak{gl}(E_{\widetilde {\rm \bf N}})$$   
defines a representation of the semigroup $\Sigma_{\rm \bf N}$ in the phase space of the fundamental network. This justifies that in Theorem \ref{symmetrytheorem}
the fundamental network maps $\gamma_f^{\widetilde{\rm\bf N}}$ are called ``$\Sigma_{\rm\bf N}$-equivariant''. One could also say that $\Sigma_{\rm \bf N}$ is a ``symmetry-semigroup'' of the fundamental network maps.
\end{remark}
\begin{remark}
By Corollary \ref{embeddedcor}, the dynamics of every input network ${\rm \bf N}_{(v)}$ is embedded as the robust synchrony space ${\rm Syn}_{P_{(v)}}$ inside the phase space of the fundamental network $\widetilde {\rm\bf N}$. Nevertheless, this synchrony space may not be invariant under the action of $\Sigma_{\rm \bf N}$ on the phase space of the fundamental network, i.e. it may not hold that $\phi_{\sigma_j}^*({\rm Syn}_{P_{(v)}} )\subset {\rm Syn}_{P_{(v)}}$ for all $\sigma_j\in \Sigma_{\rm \bf N}$.  Alternatively, if it so happens that  $\phi_{\sigma_j}^*({\rm Syn}_{P_{(v)}} )\subset {\rm Syn}_{P_{(v)}}$, then it is possible that $\phi_{\sigma_j}^*$ acts trivially on ${\rm Syn}_{P_{(v)}}$ (i.e. fixes it pointwise).

All this means that $\Sigma_{\rm\bf N}$ may not act (or not act faithfully) on the phase space of the   network ${\rm \bf N}$, but only on the extended phase space of its fundamental lift $\widetilde {\rm\bf N}$, in which that of ${\rm\bf N}$ is embedded. We think of the elements of $\Sigma_{\rm\bf N}$ as {\it hidden symmetries} of ${\rm \bf N}$. Perhaps counterintuitively, these hidden symmetries may have a major impact on the dynamics of ${\rm\bf N}$, see for example Remark \ref{symmetrysynchronyremark}.
\end{remark}
\begin{example}\label{symmetriesexample}
Recall the fundamental networks $\widetilde {\rm \bf A}$, $\widetilde {\rm \bf B}$ and $\widetilde {\rm \bf C}$ of Example \ref{fundamentalexample} and Figure \ref{pictfundamental}. Their self-fibrations can be read off from the product tables of $\Sigma_{\rm \bf A}$, $\Sigma_{\rm \bf B}$ and $\Sigma_{\rm \bf C}$ given in Example \ref{fundamentalexample}. The action of these self-fibrations on vertices is as follows:
\begin{align}\nonumber
 \begin{array}{c|ccc} \widetilde{\rm \bf A} & \sigma_1 & \sigma_2& \sigma_3 \\ \hline 
\phi_{\sigma_1}  & \sigma_1 & \sigma_2 & \sigma_3     \\
 \phi_{\sigma_2} &  {\sigma_2}  &  {\sigma_3} &  {\sigma_3}  \\
 \phi_{\sigma_3}   &  {\sigma_3}   &  {\sigma_3} &  {\sigma_3}\\
 \multicolumn{3}{c}{ }\\
 \multicolumn{3}{c}{ }
\end{array}
\hspace{.5cm}
\begin{array}{c|cccc} \widetilde{\rm \bf B} & \sigma_1 & \sigma_2& \sigma_3& \sigma_4\\ \hline 
\phi_{\sigma_1}  & \sigma_1 & \sigma_2 & \sigma_3  & \sigma_4     \\
 \phi_{\sigma_2} &  {\sigma_2} &   {\sigma_4} &  {\sigma_3} &  {\sigma_4}   \\
  \phi_{\sigma_3}   &  {\sigma_3} &  {\sigma_4} &  {\sigma_3}&  {\sigma_4}   \\
 \phi_{\sigma_4} & \sigma_4& \sigma_4&\sigma_3 &\sigma_4 \\
 \multicolumn{4}{c}{ }
\end{array}
\hspace{.5cm}
\begin{array}{c|ccccc} \widetilde{\rm \bf C} & \sigma_1 & \sigma_2& \sigma_3& \sigma_4& \sigma_5\\ \hline 
\phi_{\sigma_1}  & \sigma_1 & \sigma_2 & \sigma_3  & \sigma_4 & \sigma_5    \\
\phi_{\sigma_2} & {\sigma_2} &  {\sigma_4} & {\sigma_3} & {\sigma_4} & {\sigma_5}  \\
  {\phi_{\sigma_3}}   &  {\sigma_3} &  {\sigma_5} &  {\sigma_3}&  {\sigma_4} &  {\sigma_5} \\
  \phi_{\sigma_4} &  \sigma_4& \sigma_4&\sigma_3 &\sigma_4 &\sigma_5  \\
  \phi_{\sigma_5} & \sigma_5&\sigma_4 & \sigma_3&\sigma_4 &\sigma_5 
\end{array}
\hspace{.1cm} .
\end{align}
We note that, other than the identity $\phi_{\sigma_1}$, none of these self-fibrations is invertible. 

The  symmetries of the equations of motion of the fundamental networks can in turn be read off from these tables. They are given by: 
\begin{center}
 \begin{tabular}{p{3.6cm}}
 \multicolumn{1}{c}{\bf Network $\widetilde {\rm \bf A}$}\\
\hline
\vspace{-7mm} %
  \begin{equation} \hspace{-2mm}\nonumber
\begin{array}{c}
 \phi_{\sigma_1}^*(X)=(X_{\sigma_1}, X_{\sigma_2}, X_{\sigma_3})\\
\phi_{\sigma_2}^*(X)=(X_{\sigma_2}, X_{\sigma_3}, X_{\sigma_3})\\
\phi_{\sigma_3}^*(X)=(X_{\sigma_3}, X_{\sigma_3}, X_{\sigma_3})\\
 \multicolumn{1}{c}{}
 \end{array}
 \end{equation}
 \end{tabular}
 \hspace{.5cm}
\begin{tabular}{p{4.3cm}}
  \multicolumn{1}{c}{\bf Network $\widetilde {\rm \bf B}$}\\
\hline
\vspace{-7mm} %
  \begin{equation} \hspace{-2mm}\nonumber
 \begin{array}{c} 
\phi_{\sigma_1}^*(X)=(X_{\sigma_1}, X_{\sigma_2}, X_{\sigma_3}, X_{\sigma_4}) \\ 
\phi_{\sigma_2}^*(X)=(X_{\sigma_2}, X_{\sigma_4}, X_{\sigma_3}, X_{\sigma_4}) \\ 
\phi_{\sigma_3}^*(X)=(X_{\sigma_3}, X_{\sigma_4}, X_{\sigma_3}, X_{\sigma_4})\\
\phi_{\sigma_4}^*(X)=(X_{\sigma_4}, X_{\sigma_4}, X_{\sigma_3}, X_{\sigma_4})  
\end{array}
\end{equation}
 \end{tabular} 
 \begin{tabular}{p{5cm}} 
  \multicolumn{1}{c}{\bf Network $\widetilde {\rm \bf C}$}\\
\hline
\vspace{-7mm} %
  \begin{equation} \hspace{-2mm} \nonumber  
 \begin{array}{l}
 \phi_{\sigma_1}^*(X)=(X_{\sigma_1}, X_{\sigma_2}, X_{\sigma_3}, X_{\sigma_4}, X_{\sigma_5}) \\
\phi_{\sigma_2}^*(X)=(X_{\sigma_2}, X_{\sigma_4}, X_{\sigma_3}, X_{\sigma_4}, X_{\sigma_5}) \\ 
\phi_{\sigma_3}^*(X)=(X_{\sigma_3}, X_{\sigma_5}, X_{\sigma_3}, X_{\sigma_4}, X_{\sigma_5})\\
\phi_{\sigma_4}^*(X)=(X_{\sigma_4}, X_{\sigma_4}, X_{\sigma_3}, X_{\sigma_4}, X_{\sigma_5}) \\
\phi_{\sigma_5}^*(X)=(X_{\sigma_5}, X_{\sigma_4}, X_{\sigma_3}, X_{\sigma_4}, X_{\sigma_5}) 
 \end{array}\, .
 \end{equation}
\end{tabular}
 \end{center}\vspace{-3mm}
One may also check directly from the equations of motion that these maps send solutions to solutions. We remark that in network $\widetilde {\rm \bf C}$ the synchrony space $\{X_{\sigma_1}=X_{\sigma_3}, X_{\sigma_2}=X_{\sigma_5}\}$ (which is isomorphic to network ${\rm \bf C}$) is only invariant under the symmetries $\phi_{\sigma_1}^*$ and $\phi_{\sigma_3}^*$, which both act trivially on it. This confirms that network ${\rm \bf C}$ does not admit any nontrivial symmetries, while its fundamental lift does. Similarly, in network ${\rm\bf B}$, the synchrony space $\{X_{\sigma_2}=X_{\sigma_3}\}$ (which is isomorphic to network ${\rm\bf B}$) is only invariant under the trivial symmetry $\phi_{\sigma_1}^*$. Network ${\rm\bf A}$, on the other hand, is isomorphic to network $\widetilde {\rm \bf A}$, and is hence symmetric itself: it admits the full symmetry semigroup $\Sigma_{\rm\bf A}$.
\end{example}

\noindent The following result emphasises the geometric importance of the hidden symmetries of the fundamental network. It states that they determine its robust synchronies.\begin{theorem}\label{symmetrysynchrony}
Let $P=\{P_1, \ldots, P_r\}$ be a balanced partition of the cells $\Sigma_{\rm\bf N}$ of a homogeneous fundamental network $\widetilde {\rm\bf N}$ and let $\gamma: E_{\widetilde {\rm \bf N}}\to E_{\widetilde {\rm \bf N}}$ be any $\Sigma_{\rm \bf N}$-equivariant map. Then $$\gamma({\rm Syn}_P)\subset {\rm Syn}_P\, .$$
\end{theorem}
\begin{proof}
Assume that $\gamma: E_{\widetilde {\rm \bf N}}\to E_{\widetilde {\rm \bf N}}$ is $\Sigma_{\rm \bf N}$-equivariant, i.e. that $\gamma\circ \phi^*_{\sigma_i} =  \phi^*_{\sigma_i}\circ \gamma$ for all $\sigma_i\in \Sigma_{\rm \bf N}$. This implies that 
\begin{align}\nonumber 
& \gamma_{\sigma_i}(X) = \gamma_{\sigma_1\circ\sigma_i}(X) = \left(\phi_{\sigma_i}^*\gamma\right)_{\sigma_1}\!\!(X) = \left(\gamma\circ\phi_{\sigma_i}^*\right)_{\sigma_1}\!\!(X) =\\ \nonumber & \gamma_{\sigma_1}(X_{\sigma_1\circ\sigma_i}, \ldots, X_{\sigma_n\circ\sigma_i})=\gamma_{\sigma_1}(X_{\widetilde \sigma_1(\sigma_i)}, \ldots, X_{\widetilde \sigma_n(\sigma_i)} ) \, .
\end{align}
In other words, $\gamma$ is a homogeneous network vector field with response function $\gamma_{\sigma_1}$ on the network with vertex set $\Sigma_{\rm \bf N}$ and with input maps $\widetilde \sigma_1, \ldots, \widetilde \sigma_n$. Note that this does not imply that $\gamma$ is a network vector field for the fundamental network $\widetilde {\rm \bf N}$, for which the input maps are $\widetilde \sigma_1, \ldots, \widetilde \sigma_m$ (recall that $m$ may be strictly less than $n$ in general).

Now recall from Remark \ref{balancedremark} that $P$ is a balanced partition if and only if for all $1\leq j\leq m$ and $1\leq k\leq r$ there is an $1\leq l\leq r$ such that $\widetilde \sigma_j(P_k)\subset P_l$ (that is if $\widetilde \sigma_1, \ldots, \widetilde\sigma_m$ preserve the partition). But the $\widetilde \sigma_j$ with $m+1\leq j\leq n$ are all of the form $\widetilde \sigma_j = \widetilde \sigma_{j_1}\circ \ldots \circ \widetilde \sigma_{j_q}$ for $1\leq j_1, \ldots, j_q\leq m$. Hence all the $\widetilde \sigma_1, \ldots, \widetilde \sigma_n$ preserve the partition and the partition is automatically balanced for the extended network with  input maps $\widetilde \sigma_1, \ldots, \widetilde \sigma_n$. In particular, ${\rm Syn}_P$ is invariant under $\gamma$.
\end{proof}

\begin{remark}
Let $\phi: {\rm \bf N}\to {\rm \bf N}$ be a self-fibration of a network and let $\gamma: E_{\rm\bf N}\to E_{\rm \bf N}$ be an equivariant map, i.e. $\phi^*\circ\gamma=\gamma\circ\phi^*$.  Then ${\rm Fix}\, \phi^* := \{x\in  E_{\rm\bf N}\, |\, \phi^*x=x\}$ is an example of an invariant subspace for $\gamma$, because $\phi^*(\gamma(x))=\gamma(\phi^*(x)) = \gamma(x)$ if $\phi^*(x)=x$. This is how invertible network symmetries (those that form the symmetry group of the network) yield invariant subspaces in a network dynamical system. 

But when $\phi$ is not invertible, then one can imagine many more invariant subspaces induced by symmetry. For example, the image ${\rm im}\, \phi^*$ of $\phi^*$ and the inverse image $(\phi^*)^{-1}(W)$ of a $\gamma$-invariant subspace $W$ are invariant under the dynamics of $\gamma$.
\end{remark}

\begin{remark}\label{symmetrysynchronyremark}
Recall from Remark \ref{restrictionremark} that we may think of the phase space $E_{{\rm \bf N}_{(v)}}$ of the input network ${\rm\bf N}_{(v)}$ as a robust synchrony space in the phase space $E_{\widetilde {\rm\bf N}}$ of the fundamental network $\widetilde {\rm \bf N}$. It holds that $\gamma_f^{{\rm \bf N}_{(v)}} =  \gamma_f^{\widetilde {\rm\bf N}}|_{E_{{\rm\bf N}_{(v)}}}$ and hence every robust synchrony space ${\rm Syn}_P \subset E_{{\rm \bf N}_{(v)}}$ for the dynamics of ${\rm \bf N}_{(v)}$ is also a robust synchrony space for the dynamics of $\widetilde {\rm\bf N}$. Theorem \ref{symmetrysynchrony} states that not only the class of network maps $\gamma_f^{\widetilde {\rm\bf N}}: E_{\widetilde {\rm\bf N}}\to E_{\widetilde {\rm\bf N}}$ leaves $E_{{\rm\bf N}_{(v)}}$ and ${\rm Syn}_P\subset E_{{\rm\bf N}_{(v)}}$ invariant, but the possibly much larger class of $\Sigma_{\rm\bf N}$-equivariant maps $\gamma: E_{\widetilde {\rm\bf N}}\to E_{\widetilde {\rm\bf N}}$ does so as well. 
Thus, one could argue for homogeneous networks that robust synchrony is not caused by network structure, but by the more general phenomenon that the network is embedded inside a (possibly larger) network with symmetries. We will see in Sections \ref{interiorsection} and \ref{nonhomogeneoussection} that the same is true for networks with ``interior symmetries'' and for nonhomogeneous networks. One may conjecture that robust synchrony is always a consequence of hidden symmetry, even in networks with nontrivial symmetry groupoids.
\end{remark}

\section{The hidden symmetry perspective}\label{perspectivesection}
Every homogeneous network is embedded in a network with semigroup symmetry, and this explains some of the most important structural features of homogeneous network dynamical systems: their robust synchronies. We will see in Sections \ref{interiorsection} and \ref{nonhomogeneoussection} that the same is true for networks with interior symmetries and for nonhomogeneous networks with a trivial symmetry groupoid. But symmetry and hidden symmetry may cause many more of the intriguing phenomena that have been observed in networks, and that can not be explained from the existence of robust synchrony alone. These phenomena range from the existence  of multirythms \cite{romano} to the emergence of synchronous chaos \cite{pikovsky}, and also include ``anomalous'' synchrony breaking bifurcations. For instance, semigroup symmetry forces spectral degeneracies at local bifurcations, see \cite{RinkSanders3}. As a result, bifurcations of networks that may be conceived as anomalous at first sight, may turn out generic in certain classes of semigroup equivariant dynamical systems. This is true, for example, for the synchrony breaking  bifurcations in networks ${\rm\bf A}, {\rm \bf B}$ and ${\rm\bf C}$ that were discussed in Section \ref{examplessection}. In addition, (hidden) symmetry is easier to incorporate in the analysis of network systems than ``network structure'' - if only because hidden symmetry is not lost under coordinate changes, and is therefore an intrinsic property of a dynamical system. Thus, the analysis of networks may become simpler when their (hidden) symmetries are taken into account.  

All this suggests adopting a ``hidden symmetry perspective'' towards network dynamics: many network systems are special examples of dynamical systems with (hidden) symmetries, and one may organise the analysis of these networks around their hidden symmetries.  We remark that a similar perspective has been very fruitful for our understanding of dynamical systems with ``classical'' symmetries \cite{field4, perspective, golschaef2}. In particular,  many generic phenomena in dynamical systems with compact symmetry groups have been classified, and there exists a well-developed theory of local bifurcations for dynamical systems with compact symmetry groups. This theory relies on representation theory, equivariant singularity theory, and (group-)equivariant counterparts of the most important methods from local bifurcation theory, such as {\it normal form reduction, Lyapunov-Schmidt reduction} and {\it centre manifold reduction}. Neither of these theories and methods admits a natural generalisation to systems with a network structure. On the other hand, it turns out that (hidden) semigroup symmetry can be preserved in all three aforementioned reduction methods. For normal form reduction this was essentially proved in \cite{CCN}, and for Lyapunov-Schmidt reduction in \cite{RinkSanders3}. For centre manifold reduction the situation is more technical. How semigroup symmetry affects a centre manifold, is the topic of a paper that we are currently finishing.

\section{Hidden symmetry in local bifurcations}\label{methodssection}
It is not our goal to develop the local bifurcation theory of dynamical systems with semigroup symmetry any further in this paper. Instead, we shall briefly sketch now how hidden symmetry can impact local bifurcations, at the hand of our example networks {\bf A}, {\bf B} and {\bf C}. We claim that the synchrony breaking bifurcations in these networks that were discussed in Section \ref{examplessection}, are determined by hidden symmetry, and we will sketch how this can be proved. We stress that this section is only meant as an illustration of the importance of hidden symmetry for the synchrony breaking behaviour of networks. Several claims that are made in this section have been or will be proved elsewhere. 

We start with recalling some general theory from \cite{RinkSanders3}. First of all, when $\Sigma$ is a semigroup and $W$ a finite dimensional real vector space, then we call a map
$$A: \Sigma \to \mathfrak{gl}(W)\ \mbox{for which}\ A_{\sigma_i}\circ A_{\sigma_j} = A_{\sigma_i\circ \sigma_j}\ \mbox{for all} \ \sigma_i, \sigma_j\in\Sigma$$ 
a {\it representation} of the semigroup $\Sigma$ in $W$. A subspace $W_1\subset W$ is called a {\it subrepresentation} of $W$ if $A_{\sigma_i}(W_1)\subset W_1$ for all $\sigma_i\in \Sigma$. The smallest subrepresentations that build up a given representation, have a special name:
\begin{definition}
A subrepresentation $W_1\subset W$ of $\Sigma$ is called {\it indecomposable} if $W_1$ is not a direct sum $W_1=W_2\oplus W_3$ with $W_2$ and $W_3$ both nonzero subrepresentations of $W_1$. 
\end{definition}
Unlike so-called {\it irreducible} subrepresentations, indecomposable subrepresentations may contain nontrivial subrepresentations, but these  can then not be complemented by another nontrivial subrepresentation. By definition, every representation is a direct sum of indecomposable subrepresentations. Moreover, by the Krull-Schmidt theorem  \cite{RinkSanders3}, the decomposition of a representation into indecomposable subrepresentations is unique up to isomorphism.

When $A:\Sigma \to \mathfrak{gl}(W)$ is a representation and $L: W\to W$ is a linear map so that 
$$L\circ A_{\sigma_j} = A_{\sigma_j}\circ L \ \mbox{for all}\ \sigma_j\in \Sigma \, ,$$
then we call $L$ an {\it endomorphism} of $W$ and write $L\in {\rm End}(W)$. 
\begin{remark}
Recall from Theorem \ref{symmetrytheorem} that each fundamental network map $\gamma_f^{\widetilde {\rm\bf N}}: E_{\widetilde {\rm\bf N}}\to  E_{\widetilde {\rm\bf N}}$ is $\Sigma_{\rm\bf N}$-equivariant, i.e. $\gamma_f^{\widetilde {\rm\bf N}}\circ\phi^*_{\sigma_i} = \phi^*_{\sigma_i}\circ \gamma_f^{\widetilde {\rm\bf N}}$ for all $\sigma_i\in \Sigma_{\rm\bf N}$. Differentiation of this identity at a fully synchronous (and hence fixed by $\Sigma_{\rm\bf N}$) point (say $X=0$) yields that
$$L\circ \phi^*_{\sigma_i} = \phi^*_{\sigma_i}\circ L\ \mbox{for}\ L:= D_X\gamma_f^{\widetilde {\rm\bf N}}(0)\, .$$
In other words, the linearisation of a fundamental network map at a fully synchronous point is an example of an endomorphism of the representation of $\Sigma_{\rm\bf N}$ in $E_{\widetilde {\rm\bf N}}$.
\end{remark} 
When $\lambda\in \R$ is an eigenvalue of an endomorphism $L\in{\rm End}(W)$, the generalised eigenspace
$$\mathbb{E}_{\lambda}:={\rm ker}\, (L-\lambda {\rm Id}_W)^{{\rm dim}\, W}$$ 
is a subrepresentation of $W$, and the same is true for the real generalised eigenspaces 
of the complex eigenvalues  of $L$. It follows that 
the (unique) splitting of $W$ into indecomposable subrepresentations determines to a large extent the spectral properties of its endomorphisms, 
and this explains how symmetry and hidden symmetry can force the linearisation matrix of a network map to have a degenerate spectrum. See \cite{RinkSanders3} for more precise statements on the relation between indecomposable subrepresentations and the spectrum of endomorphisms. 

\begin{example}\label{examplelinearstuff}
Recall the fundamental network maps $\gamma_f^{\widetilde {\rm\bf A}}, \gamma_f^{\widetilde {\rm\bf B}}$ and $\gamma_f^{\widetilde {\rm\bf C}}$ given in Figure \ref{pictfundamental}. Assume now that the cells in the networks are $1$-dimensional (that is $X_{\sigma_i}\in \R$ for all $\sigma_i$). Then the linearisation  $L_{\widetilde {\rm\bf A}} := D_X\gamma_f^{\widetilde {\rm\bf A}}(0;0)$ has the form (writing $a:=D_1f(0;0)\in \R$ etc.)
\begin{align}\nonumber
L_{\widetilde {\rm\bf A}} = 
{   \left( 
\begin{array}{rrr} 
a  & \bl{b} & \ro{c}\\
 0 & a & \bl{b}+\ro{c} \\
0&0& a+\bl{b}+\ro{c} 
\end{array}\right) }
 \, .
  \end{align} 
When $\bl{b}+\ro{c}\neq 0$,  then $L_{\widetilde {\rm\bf A}}$ has an eigenvalue $a+\bl{b}+\ro{c}$ with algebraic and geometric multiplicity $1$ and an eigenvalue $a$ with algebraic multiplicity $2$ and geometric multiplicity $1$. The generalised eigenspaces of $L_{\widetilde {\rm\bf A}}$ are 
\begin{align}\nonumber
 & \mathbb{E}_{a+\bl{b}+\ro{c}} = \{X_{\sigma_1}=X_{\sigma_2}=X_{\sigma_3}\} \ \mbox{and} \\ \nonumber
 & \mathbb{E}_{a} = \{X_{\sigma_3}=0\}\, .
 \end{align}
Recall that $L_{\widetilde {\rm\bf A}}$ is an endomorphism of the representation of  $\Sigma_{\rm\bf A}$ (this representation was given in Example \ref{symmetriesexample}). It turns out that $ \mathbb{E}_{a+\bl{b}+\ro{c}}$ and $ \mathbb{E}_{a}$ both are indecomposable subrepresentations of $\Sigma_{\rm\bf A}$. Because the splitting of a representation into indecomposable summands is unique up to isomorphism, it follows that every endomorphism of $\Sigma_{\rm\bf A}$ can have at most $2$ generalised eigenspaces. Thus, the spectral degeneracy of $L_{\widetilde {\rm\bf A}}$ is a consequence of symmetry. Because networks ${\rm\bf A}$ and $\widetilde {\rm\bf A}$ are isomorphic, the double degeneracy of the eigenvalue $a$ in network ${\rm \bf A}$ is also a result of $\Sigma_{\rm\bf A}$-equivariance.

Similar considerations apply to $\widetilde {\rm\bf B}$ and $\widetilde {\rm\bf C}$. 
The linearisation $L_{\widetilde {\rm\bf B}}:=D_X\gamma_f^{\widetilde {\rm\bf B}}(0;0)$ reads
\begin{align}\nonumber
L_{\widetilde {\rm\bf B}} = 
{   \left( \begin{array}{rrrr} 
a  & \bl{b} & \ro{c}& 0\\
 0 & a & \ro{c}&  \bl{b} \\
0&0& a+\ro{c}&\bl{b}\\
0 & 0 & \ro{c} & a+\bl{b}  
 \end{array}\right) }\, .
  \end{align}
When $\bl{b}+\ro{c}\neq 0$, its generalised eigenspaces are
 \begin{align}\nonumber
 & \mathbb{E}_{a+\bl{b}+\ro{c}} = \{X_{\sigma_1}=X_{\sigma_2}=X_{\sigma_3}=X_{\sigma_4}\} \ \mbox{and} \\ \nonumber
 & \mathbb{E}_{a} = \{\ro{c}X_{\sigma_3} + \bl{b}X_{\sigma_4}  =0\}\, .
 \end{align} 
  Both are indecomposable subrepresentations of $\Sigma_{\rm\bf B}$. 
  In addition, equivariance implies that $L_{\widetilde {\rm\bf B}}$ leaves the synchrony space $\{X_{\sigma_2}=X_{\sigma_3}\}$ (that is, network ${\rm\bf B}$) invariant. This synchrony space intersects $\mathbb{E}_{a}$ in a $2$-dimensional subspace, and this explains the double degeneracy of the eigenvalue $a$ in network {\bf B}.
 Finally, 
 the linearisation matrix $L_{\widetilde {\rm\bf C}}:=D_X\gamma_f^{\widetilde {\rm\bf C}}(0;0)$  is
\begin{align}\nonumber
\label{matrixtildeC}\hspace{-.5cm}  L_{\widetilde {\rm\bf C}} = 
{   \left( \begin{array}{rrrrr} 
a  & \bl{b} & \ro{c}& 0& 0\\
 0 & a & \ro{c}&  \bl{b} &0 \\
0&0& a+\ro{c}&0& \bl{b}\\
0 & 0 & \ro{c} & a+\bl{b} &0  \\
0 & 0 & \ro{c} & \bl{b} & a
 \end{array}\right) }\, .
  \end{align} 
  When $\bl{b}+\ro{c}\neq 0$, it has generalised eigenspaces
 \begin{align}\nonumber
 & \mathbb{E}_{a+\bl{b}+\ro{c}} = \{X_{\sigma_1}=X_{\sigma_2}=X_{\sigma_3}=X_{\sigma_4}=X_{\sigma_5}\} \ \mbox{and} \\ \nonumber
 & \mathbb{E}_{a} = \{\ro{c}(\bl{b}+\ro{c})X_{\sigma_3} + \bl{b}^2X_{\sigma_4} + \bl{b}\ro{c}X_{\sigma_5}=0\}\, .
 \end{align}
  The degenerate eigenvalue $a$ now has algebraic multiplicity $4$ and geometric multiplicity $1$. Both generalised eigenspaces are indecomposable subrepresentations of $\Sigma_{\rm\bf C}$. Moreover, $\mathbb{E}_a$ intersects the robust synchrony space $\{X_{\sigma_1}=X_{\sigma_3}, X_{\sigma_2}=X_{\sigma_5}\}$ (that is, network ${\rm\bf C}$) in a two-dimensional subspace. This explains the double degeneracy of the eigenvalue $a$ in network ${\rm\bf C}$.
  \end{example}
 Hidden symmetries do not only affect the linear, but also the nonlinear terms of network maps. One can therefore expect different nonlinear dynamics and bifurcations in networks with non-isomorphic (hidden) symmetry semigroups. Indeed, this is what explains the different character of the synchrony breaking bifurcations in networks {\bf A}, {\bf B} and {\bf C}. 
 
 These bifurcations can be investigated with various classical methods, including normal form reduction, centre manifold reduction and Lyapunov-Schmidt reduction. We will use the remainder of this section to sketch how Lyapunov-Schmidt reduction (which is perhaps the simplest of these methods) can predict the local asymptotics of the synchrony breaking steady state branches of a fundamental network $\widetilde {\rm \bf N}$. In principle, information about the stability of solution branches can not be obtained with this method.
 
 So let us study the steady states of a parameter dependent fundamental network map 
$$\gamma_f^{\widetilde {\rm\bf N}}: E_{\widetilde {\rm\bf N}}\times \Lambda \to E_{\widetilde {\rm\bf N}}\ \mbox{with}\ \Lambda\subset \R^p \ \mbox{an open set of parameters},$$
near a given synchronous steady state (say $X=0$) and given parameter value (say $\lambda=0$). Thus, we assume that $\gamma_f^{\widetilde {\rm\bf N}}(0;0)=0$. Synchrony breaking can occur when the linearisation
 $L:=D_X\gamma_f^{\widetilde {\rm\bf N}}(0;0)$ is nonsynchronously degenerate, i.e. when
 $$\mathbb{E}_0:= {\rm gen\ ker}\, L \not\subset \{X_{\sigma_1}=\ldots=X_{\sigma_n}\}\, .$$
Lyapunov-Schmidt reduction is a method to reduce the steady state equation $\gamma_f^{\widetilde {\rm\bf N}}(X;\lambda)=0$, locally near $(X;\lambda)=(0;0)$, to an equivalent equation of the form 
 $$F(X;\lambda)=0\ \mbox{for}\ F:\mathbb{E}_0\times \Lambda \to\mathbb{E}_0\ \mbox{defined near}\ (0;0).$$ 
It was proved in \cite{RinkSanders3} that it can be arranged that this $F$ inherits $\Sigma_{\rm\bf N}$-equivariance from $\gamma_f^{\widetilde {\rm\bf N}}$ (recall that $\Sigma_{\rm\bf N}$ restricts to a representation on $\mathbb{E}_0$). 
Equivariance now imposes restrictions on $F$ that impact the solutions of the reduced bifurcation equation $F(X;\lambda)=0$.  

Moreover, if ${\rm Syn}_P\subset E_{\widetilde {\rm\bf N}}$ is any robust synchrony space, then equivariance implies that
 $$F(\mathbb{E}_0\cap{\rm Syn}_P; \lambda) \subset \mathbb{E}_0\cap{\rm Syn}_P\, ,$$
even when $\mathbb{E}_0\cap{\rm Syn}_P$ is not a subrepresentation of $\Sigma_{\rm\bf N}$. In this way, Lyapunov-Schmidt reduction replaces the problem of finding synchronous steady states of $\gamma_f^{\widetilde {\rm\bf N}}$ by the problem of  finding zeroes of 
$$F: \mathbb{E}_0\cap {\rm Syn}_P\times \Lambda \to \mathbb{E}_0\cap {\rm Syn}_P\, .$$ 
This   may entail a considerable dimension reduction of the bifurcation problem.

We shall now illustrate how these observations can be used to predict the asymptotics of generic synchrony breaking steady state branches in networks $\widetilde {\rm\bf A}$, $\widetilde {\rm\bf B}$ and $\widetilde {\rm\bf C}$. 

\begin{example}
Example \ref{examplelinearstuff} shows that network $\widetilde {\rm\bf A}$ can only break synchrony when $a=D_1f(0;0)=0$. Assuming that $a=0$ and $\bl{b}+\ro{c} \neq 0$, it holds that
$$\mathbb{E}_0= \{  X_{\sigma_3} = 0 \}\, .$$
Let us coordinatise $\mathbb{E}_0$ with the variables $(X_{\sigma_1}, X_{\sigma_2})$, and accordingly write $F=(F_{\sigma_1}, F_{\sigma_2})$. In these coordinates, the action of $\Sigma_{\rm\bf A}$ (see Example \ref{symmetriesexample}) on $\mathbb{E}_0$ is given by
\begin{align}\nonumber
& \phi^*_{\sigma_1}(X_{\sigma_1}, X_{\sigma_2}) = (X_{\sigma_1}, X_{\sigma_2})\, ,\\ \nonumber
& \phi^*_{\sigma_2}(X_{\sigma_1}, X_{\sigma_2}) = (X_{\sigma_2}, 0)\, ,\\ \nonumber
& \phi^*_{\sigma_3}(X_{\sigma_1}, X_{\sigma_2}) = (0, 0)\, .
\end{align}
The equivariance of $F$ under $\phi^*_{\sigma_2}$ now gives the identities
\begin{align}\nonumber
F_{\sigma_2}(X_{\sigma_1}, X_{\sigma_2}; \lambda) &= (\phi^*_{\sigma_2}F)_{\sigma_1}(X_{\sigma_1}, X_{\sigma_2}; \lambda) = F_{\sigma_1}(\phi_{\sigma_2}^*(X_{\sigma_1}, X_{\sigma_2}); \lambda) = F_{\sigma_1}(X_{\sigma_2},0; \lambda) \, ,\\ \nonumber
 \mbox{and} \ 0 &= (\phi^*_{\sigma_2}F)_{\sigma_2}(X_{\sigma_1}, X_{\sigma_2}; \lambda) = F_{\sigma_2}(\phi_{\sigma_2}^*(X_{\sigma_1}, X_{\sigma_2}); \lambda) = F_{\sigma_2}(X_{\sigma_2},0; \lambda)\, .
\end{align}
In other words, the map $F$ is of the  form 
$$F(X_{\sigma_1}, X_{\sigma_2};\lambda) = (F_{\sigma_1}(X_{\sigma_1}, X_{\sigma_2};\lambda), F_{\sigma_1}(X_{\sigma_2}, 0;\lambda))\ \mbox{with}\ F_{\sigma_1}(0,0;\lambda)=0\, .$$ 
Also, every $F$ that is of this form (for some smooth function $F_{\sigma_1}$) is $\Sigma_{\rm\bf A}$-equivariant. Because the bifurcation equation $F=0$ has a special form, 
 one may expect its local solutions to have a special structure as well. Indeed, when $\lambda\in \Lambda := \R$ and $F_{\sigma_1}$ admits the generic expansion
\begin{align}\nonumber
F_{\sigma_1}(X_{\sigma_1}, X_{\sigma_2}; \lambda) & = \alpha \lambda X_{\sigma_1} \!+\! \bl{b} X_{\sigma_2}  \!+\! AX_1^2  \! + \! \mathcal{O}( |X_{\sigma_1}|\! \cdot\! |\lambda|^2 \!+\! |X_{\sigma_1}|^3 \!+\! |X_{\sigma_2}|\!\cdot\! ||X|| \!+\!  |X_{\sigma_2}|\!\cdot\! |\lambda| )\, ,
\end{align}
then it follows that 
 $$F_{\sigma_2}(X_{\sigma_1}, X_{\sigma_2};\lambda)=  \alpha\lambda X_{\sigma_2} + AX_{\sigma_2}^2 +\mathcal{O}(|X_{\sigma_2}| \cdot |\lambda|^2 +   |X_{\sigma_2}|^3) \, . $$
Under the nondegeneracy conditions that $\alpha, \bl{b}, A\neq 0$, the equation $F_{\sigma_2}=0$   yields that $X_{\sigma_2}=0$ or $X_{\sigma_2} = -\frac{\alpha}{A}\lambda+\mathcal{O}(\lambda^2)$. In the first case, the equation $F_{\sigma_1}=0$   gives that either $X_{\sigma_1}=0$ or $X_{\sigma_1} = -\frac{\alpha}{A}\lambda+\mathcal{O}(\lambda^2)$. In the second case, we find that $X_{\sigma_1}= \pm \sqrt{\frac{\bl{b}\alpha}{A^2}\lambda} + \mathcal{O}(\lambda)$. As a result, one can expect network $\widetilde {\rm\bf A}$ to 
generically support three solution branches near $(X; \lambda)=(0; 0)$. They have the asymptotics
\begin{align}\nonumber
& X_{\sigma_1}=X_{\sigma_2}=X_{\sigma_3}=0\, ,\\ \nonumber & X_{\sigma_1}\sim \lambda, X_{\sigma_2} = X_{\sigma_3} = 0\ \mbox{and} \\ \nonumber & X_{\sigma_1}\sim\pm\sqrt{\lambda}, X_{\sigma_2}\sim\lambda, X_{\sigma_3}=0\, .
\end{align}
These branches lie on one $\Sigma_{\rm\bf A}$-orbit. Moreover, recalling from Example \ref{pictfundamental1} that $x_{v_1}=X_{\sigma_3}, x_{v_2}=X_{\sigma_2}, x_{v_3}=X_{\sigma_1}$, this shows that the bifurcation in network ${\rm\bf A}$ displayed in Table \ref{table}, is a generic equivariant bifurcation. A proof of this can also be found in \cite{RinkSanders2, RinkSanders3}. 
\end{example} 

 \begin{example}
 Also network $\widetilde {\rm\bf B}$ can only break
   synchrony when $a=0$. Assuming this and $\bl{b}+\ro{c} \neq 0$, we recall from Example \ref{examplelinearstuff} that
$$\mathbb{E}_0= \{\ro{c} X_{\sigma_3} + \bl{b} X_{\sigma_4}  = 0 \}\, .$$
When $\bl{b}\neq 0$,   we may coordinatise $\mathbb{E}_0$ by $(X_{\sigma_1}, X_{\sigma_2}, X_{\sigma_3})$, letting $X_{\sigma_4}=  -\frac{\ro{c}}{\bl{b}}X_{\sigma_3}$.
Similarly we coordinatise $F: \mathbb{E}_0\times \R \to\mathbb{E}_0$ as $F =(F_{\sigma_1} , F_{\sigma_2}, F_{\sigma_3})$.
The action of $\Sigma_{\rm\bf B}$ on $\mathbb{E}_0$ is then
\begin{align}\nonumber
& \phi^*_{\sigma_1}(X_{\sigma_1}, X_{\sigma_2}, X_{\sigma_3}) = (X_{\sigma_1}, X_{\sigma_2}, X_{\sigma_3})\, ,\\ \nonumber
& \phi^*_{\sigma_2}(X_{\sigma_1}, X_{\sigma_2}, X_{\sigma_3}) = (X_{\sigma_2}, -\frac{\ro{c}}{\bl{b}}X_{\sigma_3}, X_{\sigma_3})\, ,\\ \nonumber
& \phi^*_{\sigma_3}(X_{\sigma_1}, X_{\sigma_2}, X_{\sigma_3}) = (X_{\sigma_3}, -\frac{\ro{c}}{\bl{b}}X_{\sigma_3}, X_{\sigma_3})\, ,\\ \nonumber
& \phi^*_{\sigma_4}(X_{\sigma_1}, X_{\sigma_2}, X_{\sigma_3}) = (-\frac{\ro{c}}{\bl{b}}X_{\sigma_3}, -\frac{\ro{c}}{\bl{b}}X_{\sigma_3}, X_{\sigma_3})\, .
\end{align}
The equivariance of $F$ implies among others that 
\begin{align}\nonumber
 F_{\sigma_2}(X_{\sigma_1}, X_{\sigma_2}, X_{\sigma_3}; \lambda) = & (\phi_{\sigma_2}^*F)_{\sigma_1}(X_{\sigma_1}, X_{\sigma_2}, X_{\sigma_3}; \lambda)  =  \\  & F_{\sigma_1}(\phi_{\sigma_2}^*(X_{\sigma_1}, X_{\sigma_2}, X_{\sigma_3}); \lambda)=F_{\sigma_1}(X_{\sigma_2}, -\frac{\ro{c}}{\bl{b}}X_{\sigma_3}, X_{\sigma_3}; \lambda) \ \mbox{and} \nonumber \\ \nonumber 
  F_{\sigma_3}(X_{\sigma_1}, X_{\sigma_2}, X_{\sigma_3}; \lambda) = & (\phi_{\sigma_3}^*F)_{\sigma_1}(X_{\sigma_1}, X_{\sigma_2}, X_{\sigma_3}; \lambda)  =  \\  & F_{\sigma_1}(\phi_{\sigma_3}^*(X_{\sigma_1}, X_{\sigma_2}, X_{\sigma_3}); \lambda)=F_{\sigma_1}(X_{\sigma_3}, -\frac{\ro{c}}{\bl{b}}X_{\sigma_3}, X_{\sigma_3}; \lambda) \, . \nonumber
\end{align}
In particular, it holds that $F_{\sigma_2}=F_{\sigma_3}$ if $X_{\sigma_2}=X_{\sigma_3}$ and we see that $F$ leaves the robust synchrony space $\mathbb{E}_0\cap \{X_{\sigma_2}=X_{\sigma_3}\}$ (that is, network ${\rm\bf B}$) invariant. Moreover, the remaining restrictions on $F_{\sigma_1}, F_{\sigma_2}, F_{\sigma_3}$ imposed by equivariance  can be formulated as restrictions on $F_{\sigma_1}$. It turns out that they all reduce to a single additional restriction:
$$-\frac{\ro{c}}{\bl{b}}F_{\sigma_1}(X_{\sigma_3}, -\frac{\ro{c}}{\bl{b}}X_{\sigma_3}, X_{\sigma_3};\lambda) = F_{\sigma_1}(-\frac{\ro{c}}{\bl{b}}X_{\sigma_3}, -\frac{\ro{c}}{\bl{b}}X_{\sigma_3}, X_{\sigma_3};\lambda) \, .$$
Zeroes of $F$ inside $\{X_{\sigma_2}=X_{\sigma_3}\}$ thus correspond to zeroes of  
\begin{align}
 \nonumber
 G(X_{\sigma_1}, X_{\sigma_2}; \lambda) &:= \left( \begin{array}{ll} F_{\sigma_1}(X_{\sigma_1}, X_{\sigma_2}, X_{\sigma_2};\lambda ) \\ F_{\sigma_1}(X_{\sigma_2}, -\frac{\ro{c}}{\bl{b}}X_{\sigma_2}, X_{\sigma_2}; \lambda) \end{array}\right) \, ,
 \end{align}
with the above restriction on the otherwise arbitrary function $F_{\sigma_1}$. If we assume for instance that $\lambda \in \Lambda :=\R$ and that $F_{\sigma_1}$ admits the generic expansion
\begin{align}
 F_{\sigma_1}&(X_{\sigma_1}, X_{\sigma_2}, X_{\sigma_3};\lambda)  = \alpha \lambda X_{\sigma_1} + (\bl{b}+\beta\lambda) X_{\sigma_2} + (\ro{c}+\gamma\lambda) X_{\sigma_3} + AX_1^2 + BX_1X_2 +CX_2^2 + \nonumber \\ \nonumber  
& D X_1X_3 +EX_2X_3 +FX_3^2 +  \mathcal{O}( ||X||^3+  |\lambda| \cdot ||X||^2+  |\lambda|^2 \cdot ||X|| )\, ,
\end{align}
then it follows from the condition on $F_{\sigma_1}$ that $\beta \ro{c}-\gamma \bl{b}=A\bl{b}\ro{c}+C\ro{c}^2-E\bl{b}\ro{c}+F\bl{b}^2=0$. Also, 
 $$G(X;\lambda)\!= \! \left( \!\! \begin{array}{ll} \alpha \lambda X_{\sigma_1} + (\bl{b}+\ro{c})X_{\sigma_2} +AX_{\sigma_1}^2 \\
 \ \ \ \ +   \mathcal{O}(|\lambda|\cdot |X_{\sigma_2}| + |X_1|\cdot|X_2|  + |X_{\sigma_2}|^2 +  ||X||^3+  |\lambda| \cdot ||X||^2+  |\lambda|^2 \cdot ||X|| ) \\  \alpha \lambda X_{\sigma_2} + H X_{\sigma_2}^2  +\mathcal{O}(  |X_{\sigma_2}|^3 + |\lambda|\cdot|X_{\sigma_2}|^2+   |\lambda|^2\cdot |X_{\sigma_2}| ) \end{array}\!\! \right) $$
in which $H:=A-\frac{\ro{c}}{\bl{b}}B+\frac{\ro{c}^2}{\bl{b}^2}C + D -\frac{\ro{c}}{\bl{b}}E+F$. Under the nondegeneracy conditions that $\alpha, A, H\neq 0$, the equation $F_{\sigma_2}=0$ now gives that $X_{\sigma_2}=0$ or $X_{\sigma_2}=-\frac{\alpha}{H}\lambda+\mathcal{O}(\lambda^2)$. In the first case, the equation $F_{\sigma_1}=0$ gives that $X_{\sigma_1}=0$ or $X_{\sigma_1}=-\frac{\alpha}{A}\lambda +\mathcal{O}(\lambda^2)$. In the second case we find that $X_{\sigma_1}=\pm\sqrt{\frac{\alpha(\bl{b}+\ro{c})}{AH}\lambda}+\mathcal{O}(\lambda)$.
Using our assumption that $X_{\sigma_2}=X_{\sigma_3}$ and the relation $X_{\sigma_4}=-\frac{\bl{b}}{\ro{c}}X_{\sigma_3}$, this yields three generic local steady state branches:
\begin{align}\nonumber
& X_{\sigma_1}=X_{\sigma_2}=X_{\sigma_3} = X_{\sigma_4}=0\, ,\\ & \nonumber X_{\sigma_1} \sim \lambda, X_{\sigma_2}=X_{\sigma_3}= X_{\sigma_4}= 0\ \mbox{and}\\ & \nonumber X_{\sigma_1}\sim \pm\sqrt{\lambda}, X_{\sigma_2}=X_{\sigma_3}\sim \lambda, X_{\sigma_4}=-\frac{\ro{c}}{\bl{b}}X_{\sigma_3}\sim \lambda\, .
\end{align}
These branches are not related by symmetry. On the other hand, because $x_{v_1}=X_{\sigma_4}, x_{v_2}=X_{\sigma_2}=X_{\sigma_3}, x_{v_3}=X_{\sigma_1}$, we have proved that the steady state asymptotics of network ${\rm\bf B}$ in Table \ref{table} is generic in systems with hidden $\Sigma_{\rm\bf B}$-symmetry.
\end{example}

\begin{example}
As for the previous examples, network $\widetilde {\rm\bf C}$ can only break synchrony when $a=0$. If we assume this and demand that  $\bl{b}+\ro{c} \neq 0$, then it follows from Example \ref{examplelinearstuff} that 
 $$\mathbb{E}_{0} = \{\ro{c}(\bl{b}+\ro{c})X_{\sigma_3} + \bl{b}^2X_{\sigma_4} + \bl{b}\ro{c}X_{\sigma_5}=0\}\, .$$
In the generic situation that $\bl{b}, \ro{c} \neq 0$, let us coordinatise this subspace by $(X_{\sigma_1}, X_{\sigma_2}, X_{\sigma_3}, X_{\sigma_4})$. In particular, we then have that $X_{\sigma_5}=  -\frac{\bl{b} + \ro{c}}{\bl{b}}X_{\sigma_3} - \frac{\bl{b}}{\ro{c}}X_{\sigma_4}$. 
Moreover, the action of $\Sigma_{\rm\bf C}$ on $\mathbb{E}_0$ is given in these coordinates by
\begin{align}\nonumber
& \phi^*_{\sigma_1}(X_{\sigma_1}, X_{\sigma_2}, X_{\sigma_3}, X_{\sigma_4}) = (X_{\sigma_1}, X_{\sigma_2}, X_{\sigma_3}, X_{\sigma_4})\, ,\\ \nonumber
& \phi^*_{\sigma_2}(X_{\sigma_1}, X_{\sigma_2}, X_{\sigma_3}, X_{\sigma_4}) = (X_{\sigma_2}, X_{\sigma_4}, X_{\sigma_3}, X_{\sigma_4})\, ,\\ \nonumber
& \phi^*_{\sigma_3}(X_{\sigma_1}, X_{\sigma_2}, X_{\sigma_3}, X_{\sigma_4}) = (X_{\sigma_3}, -\frac{\bl{b} + \ro{c}}{\bl{b}}X_{\sigma_3} - \frac{\bl{b}}{\ro{c}}X_{\sigma_4}, X_{\sigma_3}, X_{\sigma_4})\, ,\\ \nonumber
& \phi^*_{\sigma_4}(X_{\sigma_1}, X_{\sigma_2}, X_{\sigma_3}, X_{\sigma_4}) = (X_{\sigma_4}, X_{\sigma_4}, X_{\sigma_3}, X_{\sigma_4})\, ,\\ \nonumber 
& \phi^*_{\sigma_5}(X_{\sigma_1}, X_{\sigma_2}, X_{\sigma_3}, X_{\sigma_4}) = (-\frac{\bl{b} + \ro{c}}{\bl{b}}X_{\sigma_3} - \frac{\bl{b}}{\ro{c}}X_{\sigma_4}, X_{\sigma_4}, X_{\sigma_3}, X_{\sigma_4})\, . \nonumber
\end{align}
Coordinatising  $F: \mathbb{E}_0\times \Lambda \to\mathbb{E}_0$ as $F =(F_{\sigma_1} , F_{\sigma_2}, F_{\sigma_3}, F_{\sigma_4})$, we see that $\Sigma_{\rm\bf C}$-equivariance implies among others that the $F_{\sigma_i}$ can be expressed in terms of $F_{\sigma_1}$. For example,
\begin{align}\nonumber
  F_{\sigma_2}(X_{\sigma_1}, X_{\sigma_2}, X_{\sigma_3},X_{\sigma_4}; \lambda)  & = (\phi_{\sigma_2}^*F)_{\sigma_1}(X_{\sigma_1}, X_{\sigma_2}, X_{\sigma_3},X_{\sigma_4}; \lambda)  =  \nonumber \\  
 &F_{\sigma_1}(\phi_{\sigma_2}^*(X_{\sigma_1}, X_{\sigma_2}, X_{\sigma_3}, X_{\sigma_4}); \lambda) =  \nonumber 
 F_{\sigma_1}(X_{\sigma_2}, X_{\sigma_4}, X_{\sigma_3},  X_{\sigma_4}; \lambda) \, ,
 \end{align}
 and similarly,
 \begin{align}
 F_{\sigma_3}(X_{\sigma_1}, X_{\sigma_2}, X_{\sigma_3},X_{\sigma_4}; \lambda) =  & 
 F_{\sigma_1}(X_{\sigma_3}, -\frac{\bl{b} + \ro{c}}{\bl{b}}X_{\sigma_3} - \frac{\bl{b}}{\ro{c}}X_{\sigma_4}, X_{\sigma_3}, X_{\sigma_4} ; \lambda) \ \mbox{and} \nonumber  \\ 
 F_{\sigma_4}(X_{\sigma_1}, X_{\sigma_2}, X_{\sigma_3},X_{\sigma_4}; \lambda) =  & 
 F_{\sigma_1}(X_{\sigma_4}, X_{\sigma_4}, X_{\sigma_3}, X_{\sigma_4} ; \lambda) \, . \nonumber   
\end{align}
It turns out that equivariance is met precisely when $F_{\sigma_1}$ satisfies the additional condition
\begin{align} \nonumber
&F_{\sigma_1}( -\frac{\bl{b} + \ro{c}}{\bl{b}}X_{\sigma_3} - \frac{\bl{b}}{\ro{c}}X_{\sigma_4}, X_{\sigma_4}, X_{\sigma_3}, X_{\sigma_4} ; \lambda) = \nonumber \\
 -\frac{\bl{b} + \ro{c}}{\bl{b}} &F_{\sigma_1}(X_{\sigma_3}, -\frac{\bl{b} + \ro{c}}{\bl{b}}X_{\sigma_3} - \frac{\bl{b}}{\ro{c}}X_{\sigma_4}, X_{\sigma_3}, X_{\sigma_4} ; \lambda) - \frac{\bl{b}}{\ro{c}}F_{\sigma_1}(X_{\sigma_4}, X_{\sigma_4}, X_{\sigma_3}, X_{\sigma_4} ; \lambda) \, .  \nonumber 
\end{align}
In particular,  one may verify that $F_{\sigma_1}=F_{\sigma_3}$ and $F_{\sigma_2} = -\frac{\bl{b} + \ro{c}}{\bl{b}}F_{\sigma_3} - \frac{\bl{b}}{\ro{c}}F_{\sigma_4}$ whenever $X_{\sigma_1}=X_{\sigma_3}$ and $X_{\sigma_2} = -\frac{\bl{b} + \ro{c}}{\bl{b}}X_{\sigma_3} - \frac{\bl{b}}{\ro{c}}X_{\sigma_4}$. This confirms that $F$ leaves the robust synchrony space $\mathbb{E}_0\cap \{X_{\sigma_1}=X_{\sigma_3}, X_{\sigma_2}=X_{\sigma_5}\}$ invariant - recall that this corresponds to network ${\rm\bf C}$.  

We will choose $X_{\sigma_2}$ and $X_{\sigma_4}$ as the free variables in this restricted system, and write
$$X_{\sigma_1} = X_{\sigma_3} =-\frac{\bl{b}}{\bl{b} + \ro{c}} X_{\sigma_2}- \frac{\bl{b}^2}{\ro{c} (\bl{b} + \ro{c})} X_{\sigma_4} \  .$$
It thus follows that we are searching for zeroes of the map
\begin{align}
 \nonumber
 G(X_{\sigma_2}, X_{\sigma_4}; \lambda) &:= \left( \begin{array}{ll} F_{\sigma_1}(X_{\sigma_2}, X_{\sigma_4}, -\frac{\bl{b}}{\bl{b} + \ro{c}} X_{\sigma_2}- \frac{\bl{b}^2}{\ro{c} (\bl{b} + \ro{c})} X_{\sigma_4} ,X_{\sigma_4};\lambda ) \\ F_{\sigma_1}(X_{\sigma_4}, X_{\sigma_4}, -\frac{\bl{b}}{\bl{b} + \ro{c}} X_{\sigma_2}- \frac{\bl{b}^2}{\ro{c} (\bl{b} + \ro{c})} X_{\sigma_4} ,X_{\sigma_4};\lambda ) \end{array}\right) \, ,
 \end{align}
with $F_{\sigma_1}$ satisfying the above restriction. Assuming from this point on that $\lambda \in \Lambda :=\R$, one can quite easily translate the restriction on $F_{\sigma_1}$ into a set of equations for its Taylor series coefficients (up to any desired order), that we do not present here. 
The analysis of the equation $G(X_{\sigma_2}, X_{\sigma_4}; \lambda)=0$ now proceeds as in the previous examples. 

For instance, it is clear that $G_1=G_2$ when we put $X_{\sigma_2} = X_{\sigma_4}$ (corresponding to partial synchrony). Setting  $X_{\sigma_2} = X_{\sigma_4}$ in the equation $G_1=0$ then gives that $X_{\sigma_2} = X_{\sigma_4} = 0$ or $X_{\sigma_2} = X_{\sigma_4} \sim \lambda$, under generic conditions on the Taylor coefficients of $F_{\sigma_1}$.

To find non-synchronous solutions, one may observe that the equation $\frac{G_1-G_2}{X_{\sigma_2}-X_{\sigma_4}}=0$ generically leads to a relation of the form $X_{\sigma_2} = X_{\sigma_2}(X_{\sigma_4}, \lambda)$. Substituting this relation in the equation $G_2=0$ then yields a solution branch in which $X_{\sigma_2} \sim \lambda$ and $X_{\sigma_4} \sim \lambda$. Furthermore, doing the calculation explicitly one finds that $X_{\sigma_2}-X_{\sigma_4}\sim \lambda^2$ generically. In particular, this branch is not partially synchronous. Summarizing, we find the following local branches of steady state solutions:
 \begin{align}\nonumber
& X_{\sigma_1}=X_{\sigma_2}=X_{\sigma_3} = X_{\sigma_4}=X_{\sigma_5}=0\, \nonumber ,\\ 
& X_{\sigma_2}=X_{\sigma_4}= X_{\sigma_5} \sim \lambda, X_{\sigma_1} = X_{\sigma_3} = -\frac{\bl{b}}{\ro{c}}X_{\sigma_2} \sim \lambda \ \mbox{and} \nonumber \\ 
& \nonumber X_{\sigma_2} = X_{\sigma_5} \sim \lambda, X_{\sigma_4}  \sim \lambda, X_{\sigma_1} = X_{\sigma_3} =   -\frac{\bl{b}}{\bl{b} + \ro{c}} X_{\sigma_2}- \frac{\bl{b}^2}{\ro{c} (\bl{b} + \ro{c})} X_{\sigma_4} \sim \lambda \, ,
\end{align}
where $X_{\sigma_2} - X_{\sigma_4}\sim \lambda^2$ for the last branch. The identification $X_{\sigma_1} =  X_{\sigma_3} = x_{v_3}$, $X_{\sigma_2} =  X_{\sigma_5} = x_{v_1}$ and $X_{\sigma_4} =   x_{v_2}$ then yields the results on network ${\rm\bf C}$ reported in Table \ref{table}.
\end{example}

\noindent Under generic conditions on the response function  $f=f(X_{\sigma_1}, X_{\sigma_2}, X_{\sigma_3};\lambda)$ of networks ${\rm\bf A}, {\rm\bf B}$ and ${\rm\bf C}$, Lyapunov-Schmidt reduction at a synchrony breaking bifurcation leads to a reduced bifurcation equation $F(X;\lambda)=0$ that satisfies all the nondegeneracy conditions required of a generic equivariant bifurcation. This fact can be checked by performing the Lyapunov-Schmidt reduction explicitly, and such an analysis proves that the asymptotics displayed in Table \ref{table} is correct. Not surprisingly, the explicit Lyapunov-Schmidt reduction requires a long analysis as well.  For now, it is enough to remark that ``generic hidden symmetry considerations'' correctly predict the content of Table \ref{table}.

 Information on the stability of the bifurcating branches  can not be obtained from Lyapunov-Schmidt  reduction, but can be revealed with  techniques like centre manifold reduction. We are currently developing this technique for dynamical systems with semigroup symmetry, so we shall not prove any of the statements on stability that were made in Section \ref{examplessection}.

\section{Interior symmetry}\label{interiorsection}
In this section, we show that ``interior network symmetry'' can be interpreted as hidden network symmetry. The concept of interior symmetry was introduced in \cite{pivato2} and further studied in for example \cite{anto4}. We foresee that thinking of interior symmetry as hidden symmetry may be particularly useful for understanding bifurcations in networks with interior symmetry.
 
Interior symmetries are symmetries of certain subsets of a network that may not extend to symmetries of the full network. To make the concept of interior symmetry precise, let ${\rm \bf N}=\{A\rightrightarrows_t^s V\}$ be a network. For a subset of vertices $S\subset V$, let us define 
$$A^{S}:=\{a\in A\, |\, t(a)\in S\, \}\ \mbox{and}\ V^S:=S\cup \{s(a)\, |\, a\in A^S\} \, .$$ 
Then ${\rm\bf N}^S:=\{A^S\rightrightarrows_t^s V^S\}$ defines a subgraph of ${\rm\bf N}$ and we shall denote by 
$$e^S: {\rm\bf N}^S\to {\rm\bf N}$$
the inclusion of ${\rm\bf N}^S$ in ${\rm\bf N}$. In general, $e^S$ is not a graph fibration and hence ${\rm\bf N}^S$ is not a subnetwork of ${\rm\bf N}$, if only because there are no arrows $a^S \in A^S$ with $$t(a^S)\in \p S:=V^S\backslash S\, ,$$
while of course there may well be arrows $a\in A$ with $t(a)\in \p S$. Using the terminology of graph fibrations, interior symmetry can be defined as follows:
\begin{definition}
An {\it interior symmetry} of $S$ is a self-fibration $\phi^S: {\rm \bf N}^S\to {\rm \bf N}^S$, such that $\phi^S(v^S)=v^S$ for all $v^S\in \p S$ and $\phi^S(v^S)\in S$ for all $v^S\in S$.
\end{definition}
Note that an interior symmetry need not extend to a graph fibration of ${\rm \bf N}$. Nevertheless, the following lemma was proved in \cite{pivato2}. 
\begin{lemma}\label{interiorlemma}
Let ${\rm \bf N}=\{A\rightrightarrows_t^s V\}$ be a network, $S\subset V$ and $\phi^S$ an interior symmetry of $S$. The finest partition $P$ of $V$ such that $v\sim_P \phi^S(v)$ for all $v\in S$, is balanced. Hence, $${\rm Syn}_P=\{ x\in E_{\rm\bf N}\,|\, x_v=x_{\phi^S(v)}\ \mbox{for all}\ v\in S\}$$ is a robust synchrony space.
\end{lemma}
\begin{proof}
Note that $\{v\}\in P$ for every $v\notin S$ because $\phi^S(S)\subset S$. Thus, by Theorem \ref{balancedlemma} we only need to check that for every $v\in S$ there is a colour preserving bijection $\beta: t^{-1}(v)\to t^{-1}(\phi^S(v))$ such that $s(a) \sim_P s(\beta(a))$ for all $a\in t^{-1}(v)$. We claim that $\beta=\phi^S|_{t^{-1}(v)}$ satisfies this requirement because $\phi^S$ is a graph fibration. Indeed, if $a\in t^{-1}(v)$ and $s(a)\in S$, then $\phi^S(s(a))\in S$ and hence $s(a)  \sim_P  \phi^S(s(a)) = s(\phi^S(a))$. Otherwise, when $s(a)\notin S$, then $s(a)\in \p S$ and $s(a) = \phi^S(s(a)) = s(\phi^S(a))$.
\end{proof}
We now provide an alternative explanation of Lemma \ref{interiorlemma} based on hidden symmetry. We first describe how to ``attach'' or ``glue'' a copy of ${\rm\bf N}^S$ onto ${\rm\bf N}$ along its ``boundary'' $\p S$:
\begin{definition}
Given the networks ${\rm\bf N}=\{A\rightrightarrows_t^s V\}$ and ${\rm\bf N}^S=\{A^S\rightrightarrows V^S\}$ as above (for some $S\subset V$), we define the {\it connected sum}
$${\rm\bf N} \dot\cup_{\p S} {\rm\bf N}^S$$ 
as the network that arises from first constructing the disjoint union ${\rm\bf N} \dot\cup {\rm\bf N}^S$ (i.e. the network with vertex set $V\dot\cup V^S$ and arrow set $A\dot\cup A^S$ and source and target maps induced by inclusion) and subsequently identifying the vertices $v^S \in \p S$ of ${\rm\bf N}^S$ with their images $e^S(v^S)$ in ${\rm \bf N}$. 
\end{definition}
 
\begin{proposition}
The connected sum ${\rm\bf N} \dot\cup_{\p S} {\rm\bf N}^S$ is a network. The inclusion $$i: {\rm\bf N} \to {\rm\bf N} \dot\cup_{\p S} {\rm\bf N}^S$$ is an injective graph fibration and the  fold map 
$$j: {\rm\bf N} \dot\cup_{\p S}   {\rm\bf N}^S\to {\rm\bf N}\ \mbox{defined by}\ j|_{\rm\bf N} := {\rm Id}_{\rm\bf N} \ \mbox{and}\ j|_{{\rm\bf N}^S} :=e^S$$ 
is a surjective graph fibration.
\end{proposition}
\begin{proof}
Recall properties {\bf 1} and {\bf 2} required of a network in Definition \ref{defnetwork}. Because $e^S: {\rm\bf N}^S\to {\rm\bf N}$ is colour preserving, property {\bf 1} holds for ${\rm \bf N}\dot\cup{\rm\bf N}^S$ and hence also for ${\rm \bf N}\dot\cup_{\p S} {\rm\bf N}^S$. Property {\bf 2}  follows because there are no arrows in ${\rm \bf N}\dot\cup_{\p S} {\rm\bf N}^S$ from   $S\subset {\rm\bf N}^S$ to ${\rm\bf N}$. This implies that for any $v^S\in S$ there is a colour preserving bijection between $t^{-1}(v^S)\subset {\rm \bf N}^S$ and $t^{-1}(e^S(v^S))\subset {\rm \bf N}$. 
This is sufficient for the proof that there is a colour preserving bijection between the arrows targeting any two vertices of the same colour in ${\rm \bf N}\dot\cup_{\p S} {\rm\bf N}^S$.

Because there are no arrows from vertices outside ${\rm\bf N}\subset {\rm \bf N}\dot\cup_{\p S} {\rm\bf N}^S$  into ${\rm\bf N}$, clearly ${\rm\bf N}$ is a subnetwork of ${\rm\bf N} \dot\cup_{\p S} {\rm\bf N}^S$ and hence $i: {\rm\bf N}\to {\rm\bf N} \dot\cup_{\p S} {\rm\bf N}^S$ is an injective graph fibration.

To prove that $j$ is a surjective graph fibration, we remark that the maps ${\rm Id}_{\rm\bf N}: {\rm\bf N}\to{\rm\bf N}$ and $e^S: {\rm\bf N}^S\to{\rm\bf N}$ both satisfy the properties of a graph fibration and they coincide on $\p S\subset {\rm\bf N}^S$ and $e^S(\p S)\subset {\rm\bf N}$ that are identified in ${\rm\bf N} \dot\cup_{\p S} {\rm\bf N}^S$.
\end{proof}

\begin{remark}
The maps $i^*: E_{{\rm\bf N} \dot\cup_{\p S} {\rm\bf N}^S} \to E_{\rm\bf N}$ and $j^*:E_{{\rm\bf N}} \to E_{{\rm\bf N} \dot\cup_{\p S} {\rm\bf N}^S}$ are  given by 
$$(i^*x)_v= x_v\ \mbox{for} \ v\in {\rm \bf N}\, , \ (j^*x)_v=x_v\ \mbox{for}\ v\in {\rm \bf N}\ \mbox{and}\ (j^*x)_{v^S}=x_{e^S(v^S)}\ \mbox{for}\ v^S\in{\rm\bf N}^S\, .$$ 
Because $j$ is surjective, $j^*$ embeds the dynamics of ${\rm \bf N}$ into the phase space of ${\rm\bf N} \dot\cup_{\p S} {\rm\bf N}^S$ as the robust synchrony space 
$${\rm im}\, j^*={\rm im}\, (i\circ j)^*=\{x_v=x_{v^S}\ \mbox{when}\ v=e^S(v^S)\}\, .$$ 
\end{remark}
In the connected sum an interior symmetry manifests itself as a true symmetry:
\begin{proposition}
When $\phi^S: {\rm\bf N}^S\to{\rm\bf N}^S$ is an interior symmetry, then 
$$\phi: {\rm\bf N} \dot\cup_{\p S} {\rm\bf N}^S \to {\rm\bf N} \dot\cup_{\p S} {\rm\bf N}^S \ \mbox{defined by}\ \phi|_{\rm\bf N}:={\rm Id}_{\rm\bf N} \  \mbox{and} \ \phi|_{{\rm\bf N}^S}:=\phi^S\, $$
is a self-fibration.
\end{proposition}
\begin{proof}
The maps ${\rm Id}_{\rm\bf N}: {\rm\bf N}\to{\rm\bf N}$ and $\phi^S: {\rm\bf N}^S\to{\rm\bf N}^S$ both satisfy the properties of a graph fibration and they coincide on $\p S\subset {\rm\bf N}^S$ and $e(\p S)\subset {\rm\bf N}$ that are identified in ${\rm\bf N} \dot\cup_{\p S} {\rm\bf N}^S$
\end{proof}
\begin{remark}
The self-fibration $\phi$ may not commute with the self-fibration $i\circ j$ and hence $\phi^*$ may not commute with $(i\circ j)^*$. Consequently, the symmetries 
$\phi^*$ and $(i\circ j)^*$ of the dynamics of the connected sum may not commute either. In particular, $\phi^*$ may not leave the embedded phase space ${\rm im}\, j^* = {\rm im}\, (i\circ j)^*$ of network ${\rm\bf N}$ invariant. Thus, we may think of $\phi^*$ as a hidden symmetry for the dynamics of network ${\rm \bf N}$. It holds that
\begin{align}\nonumber
 & i^*\left( {\rm Fix}\, (i\circ j)^* \cap {\rm Fix}\, \phi^* \right)  = \\ \nonumber 
& i^*\!  \left\{ x\in E_{{\rm\bf N} \dot\cup_{\p S} {\rm\bf N}^S}\, |\, x_{v^S} = x_{e^S(v^S)}\ \mbox{and}\ x_{v^S}=x_{\phi^S(v^S)}\ \mbox{for all}\ v^S\in S \right\} = 
\\ \nonumber
& \{x\in E_{\rm\bf N}\, |\, x_v=x_{\phi^S(v)}\ \mbox{for all}\ v\in S\} ={\rm Syn}_P\, ,
\end{align}
where $P$ is the partition of Lemma \ref{interiorlemma}. This proves that ${\rm Syn}_P$ is an invariant subspace in the dynamics of ${\rm \bf N}$ because there are certain symmetries in the dynamics of a lift of ${\rm \bf N}$. In other words, Lemma \ref{interiorlemma} is a consequence of hidden symmetry.
\end{remark}
\begin{example} 
Figure \ref{pictinterior} shows a network ${\rm \bf N}$ (drawn left of the black vertical dashed line) with cells $V=\{v_1, v_2, v_3, v_4\}$  of which the subset $S=\{v_1, v_2, v_3\}$ has an  interior symmetry group $S_3$ that acts by permutation of these vertices. In this example, $\p S=\{v_4\}$ and ${\rm \bf N}^S$ is isomorphic to ${\rm \bf N}$ with its (purple dashed-dotted) arrow from $v_1$ to $v_4$ removed. The connected sum ${\rm \bf N}\dot\cup_{\p S} {\rm\bf N}^S$ has vertices $\{v_1, v_2, v_3, v_4=v_4^S, v_1^S, v_2^S, v_3^S\}$. It is the total network draw in Figure \ref{pictinterior}. It is clear from the figure that ${\rm \bf N}\dot\cup_{\p S} {\rm\bf N}^S$ has a true symmetry group $S_3$ that acts by permutation of the vertices $\{v_1^S, v_2^S, v_3^S\}$, keeping the remaining vertices fixed. In addition, the connected sum admits the noninvertible self-fibration $i\circ j$ that folds ${\rm \bf N}^S$ over ${\rm \bf N}$, i.e. it maps $v_i^S$ to $v_i$ (for $i=1,2,3$) and keeps the remaining vertices fixed.\\
 \begin{center}
 \begin{figure}[h]
\renewcommand{\figurename}{\rm \bf \footnotesize Figure}
\hspace{25mm}
\begin{tabular}{p{10cm}}
\begin{tikzpicture}[->, scale=1.8]
	  \tikzstyle{vertextype1}=[circle, draw, minimum size=20pt,inner sep=1pt]
	 \tikzstyle{vertextype2} = [draw, fill=yellow!100, minimum size=20pt,inner sep=1pt]
	 \tikzstyle{vertextype3} = [circle, draw, fill=yellow!25, minimum size=20pt,inner sep=1pt]
	 \tikzstyle{edgetype1} = [->, draw,line width=3pt,-,red!50]
	 \tikzstyle{edgetype2} = [draw, densely dashed, thick,-,black]
	  \draw [edgetype2] (2.5,-.25) -- (2.5,1.25); 
	 \node[vertextype1] (v3) at (0,0) {$v_3$};
	 \node[vertextype1] (v2) at (0, 1) {$v_2$};
	 \node[vertextype1] (v1) at (.71,.5) {$v_1$};
	  \node[vertextype1] (v3s) at (4,0) {$v_3^S$};
	 \node[vertextype1] (v2s) at (4, 1) {$v_2^S$};
	 \node[vertextype1] (v1s) at (3.29,.5) {$v_1^S$};
	\node[vertextype2] (v4) at (2,.5) {$v_4$};
	\path[]
	(v1) edge [thick, blue, bend right] node {} (v2)
	(v2) edge [blue, thick] node {} (v1)
	(v1) edge [thick, blue] node {} (v3)
	(v3) edge [bend right, blue, thick] node {} (v1)
	(v2) edge [thick, blue, bend right] node {} (v3)
	(v3) edge [blue, thick] node {} (v2)
	(v1s) edge [thick, blue, bend left] node {} (v2s)
	(v2s) edge [blue, thick] node {} (v1s)
	(v1s) edge [thick, blue] node {} (v3s)
	(v3s) edge [bend left, blue, thick] node {} (v1s)
	(v2s) edge [thick, blue, bend left] node {} (v3s)
	(v3s) edge [blue, thick] node {} (v2s)
	(v4) edge [red, dashed, thick] node {} (v1)
	(v4) edge [red, dashed, thick] node {} (v1s)
	(v4) edge [red, dashed, thick, bend right] node {} (v2)
	(v4) edge [red, dashed, thick, bend left] node {} (v3)
	(v4) edge [red, dashed, thick, bend left] node {} (v2s)
	(v4) edge [red, dashed, thick, bend right] node {} (v3s)
	(v1) edge [thick, purple, dashdotted, bend left] node {} (v4);
 \end{tikzpicture}
 \end{tabular}\vspace{-.5cm}
 \\ 
 \begin{tabular}{p{4.2cm}p{2.9cm}p{2cm}}
   \begin{equation}  \nonumber 
 \hspace{8mm} \begin{array}{rl}   
  \dot x_{v_1} &= f(x_{v_1}, \bl{ x_{v_2}}, \bl{x_{v_3}}, \ro{x_{v_4}}) \\
 \dot x_{v_2} &= f(x_{v_2}, \bl{x_{v_3}}, \bl{x_{v_1}}, \ro{x_{v_4}}) \\
\dot x_{v_3} &= f(x_{v_3}, \bl{x_{v_1}}, \bl{x_{v_2}}, \ro{x_{v_4}}) 
 \end{array}
 \end{equation} & 
 \begin{equation}  \nonumber  
\hspace{8mm} \begin{array}{rl} \mbox{} \\
 \dot x_{v_4} &= g(x_{v_4}, \pa{x_{v_1}}) \\ \mbox{}   
 \end{array}
 \end{equation} & 
  \begin{equation}  \nonumber 
\hspace{8mm} \begin{array}{rl}   
  \dot x_{v_1^S} &= f(x_{v_1^S}, \bl{ x_{v_2^S}}, \bl{x_{v_3^S}}, \ro{x_{v_4}}) \\
 \dot x_{v_2^S} &= f(x_{v_2^S}, \bl{x_{v_3^S}}, \bl{x_{v_1^S}}, \ro{x_{v_4}}) \\
\dot x_{v_3^S} &= f(x_{v_3^S}, \bl{x_{v_1^S}}, \bl{x_{v_2^S}}, \ro{x_{v_4}}) 
 \end{array}
 \end{equation}
 \end{tabular}
\vspace{-6mm}
\caption{\footnotesize {\rm A network with interior symmetry group $S_3$ embedded inside a connected sum.}}
   \label{pictinterior}
   \vspace{-6mm}
       \end{figure}
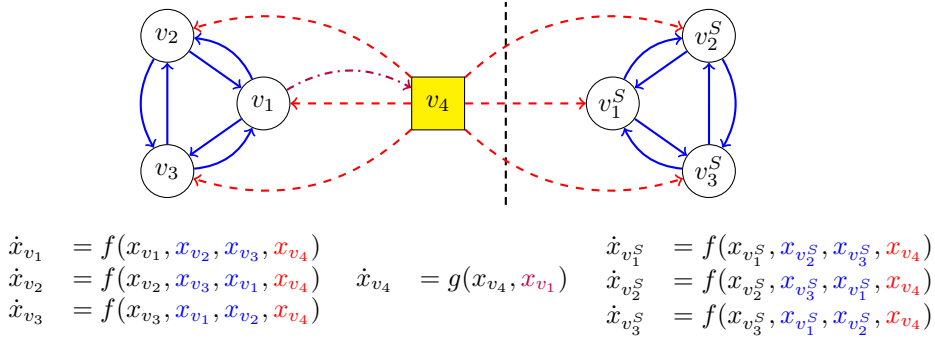  
          \end{center} 
The equations of motion of the networks ${\rm \bf N}$ and ${\rm \bf N}\dot\cup_{\p S} {\rm\bf N}^S$ are also provided in Figure \ref{pictinterior}, where we remark that $f(X_1, X_2, X_3, X_4)=f(X_1, X_3, X_2, X_4)$ because ${\rm \bf N}$ has a nontrivial symmetry groupoid (cells $v_1, v_2$, $v_3$, $v_1^S, v_2^S$ and $v_3^S$ receive two blue arrows).

It is easy to check directly from the equations of motion that
$$
(i\circ j)^*: (x_{v_1}, x_{v_2}, x_{v_3}, x_{v_4}, x_{v_1^S}, x_{v_2^S}, x_{v_3^S})\mapsto  (x_{v_1}, x_{v_2}, x_{v_3}, x_{v_4}, x_{v_1}, x_{v_2}, x_{v_3})
$$
sends solutions to solutions. Moreover, for any permutation $\phi^S:\{v_1^S, v_2^S, v_3^S\} \to \{v_1^S, v_2^S, v_3^S\}$, so does the map
$$
\phi^*: (x_{v_1}, x_{v_2}, x_{v_3}, x_{v_4}, x_{v_1^S}, x_{v_2^S}, x_{v_3^S}) \mapsto (x_{v_1}, x_{v_2}, x_{v_3}, x_{v_4}, x_{\phi^S(v_1^S)}, x_{\phi^S(v_2^S)}, x_{\phi^S(v_3^S)})\, .
$$
It is clear that $(i\circ j)^*$ and $\phi^*$ do not commute, and that $\phi^*$ does not leave ${\rm im}\, (i\circ j)^*$ invariant, unless $\phi^S={\rm Id}_{\{v_1^S, v_2^S, v_3^S\}}$. Nevertheless, the subspace
$${\rm Fix}\, (i\circ j)^* \cap {\rm Fix}\, \phi^* = \{(x_{v_1}, x_{v_2}, x_{v_3}, x_{v_4}, x_{v_1}, x_{v_2}, x_{v_3})\, |\, x_{v_i}=x_{\phi^S(v_i)}\ \mbox{for} \ i=1,2,3 \}$$
is a symmetry-induced invariant subspace for the dynamics that projects under the conjugacy $i^*: (x_{v_1}, x_{v_2}, x_{v_3}, x_{v_4}, x_{v_1^S}, x_{v_2^S}, x_{v_3^S})\mapsto (x_{v_1}, x_{v_2}, x_{v_3}, x_{v_4})$ to the invariant subset
$$\{ (x_{v_1}, x_{v_2}, x_{v_3}, x_{v_4})\, |\, x_{v_i}=x_{\phi^S(v_i)}\ \mbox{for} \ i=1,2,3 \}$$
for the dynamics of ${\rm \bf N}$.
\end{example}

\section{Nonhomogeneous networks}\label{nonhomogeneoussection}
The results for homogeneous networks of Sections \ref{homnetsection}, \ref{fundamentalsection} and \ref{hiddensection} can be generalised to nonhomogeneous networks without much effort, although the notation is heavier. We present the corresponding statements here without proof, because they are harder to formulate than to prove. We start with a definition:
\begin{definition}\label{defnonhomogeneous}
A {\it nonhomogeneous network} is a network with vertices of different colours, in which the arrows that target one vertex all have a different colour.
\end{definition}
Nonhomogeneous networks have a trivial symmetry groupoid that is not a group. Let us label the various colours of the cells of a nonhomogeneous network ${\rm\bf N}$ by $c=1,\ldots, C$ and let $V_c$ denote the collection of cells of colour $c$, so that the vertex set of ${\rm\bf N}$ is $V=\bigcup_{c=1}^C V_c$. We shall assume that every cell of colour $c$ receives exactly $m^{d,c}$ arrows from cells of colour $d$. We label the different colours of these arrows by $j=1, \ldots, m^{d,c}$. The interaction structure of the nonhomogeneous network can thus be described by input maps
$$\sigma_1^{d,c}, \ldots, \sigma^{d,c}_{m^{d,c}}: V_c\to V_d\ \mbox{for every pair of cell colours}\ c\ \mbox{and} \ d\, . $$
Here, $\sigma_{j}^{d,c}(v)\in V_d$ (for $v\in V_c$) is the source of the unique arrow of colour $j$ that targets $v$. We assume that $\sigma_1^{c,c}={\rm Id}_{V_c}$ and that $\sigma_j^{d,c}\neq\sigma_k^{d,c}$ when $j\neq k$. A map between the cells of nonhomogeneous networks defines a graph fibration if and only if it intertwines input maps.

In terms of input maps, a network map for a nonhomogeneous network ${\rm\bf N}$ is of the form 
$$(\gamma_f^{\rm\bf N})_v(x) = f^{(c)}\left(x_{\sigma_1^{1,c}(v)}, \ldots, x_{\sigma_{m^{C,c}}^{C,c}(v)}\right)\ \mbox{for every}\ v\in V_c\, .$$
Here, $f^{(c)}: E_{1}^{m^{1,c}}\times\ldots \times E_C^{m^{C,c}}\to E_c$ is the response function of cells of colour $c$ and $E_c$ denotes the common state space for cells of colour $c$ (i.e. $x_v\in E_c$ for $v\in V_c$).

Not surprisingly, one may verify that a partition $P$ of $V$ is balanced if and only if it is a refinement of the partition $\bigcup_{c=1}^C V_c$ of $V$ into cells of different colours, and if for all $\sigma_j^{d,c}$ and $P_k\subset V_c$ there is some $P_l\subset V_d$ such that $\sigma_j^{d,c}(P_k)\subset P_l$.

One can now remark that an input map of the form $\sigma_j^{d,c}: V_c\to V_d$ can only be composed with one of the form $\sigma_k^{e,d}:V_d \to V_e$. In particular, there is a smallest semigroupoid
$$\Sigma_{\rm\bf N} = \{\sigma_1^{d,c}, \ldots, \sigma^{d,c}_{n^{d,c}}\, |\, 1\leq c,d\leq C\}$$
that is closed under the composition of elements of the form $\sigma_k^{e,d}\circ\sigma_j^{d,c}$ and contains the original input maps. We shall write
$$\Sigma_{\rm\bf N} = \bigcup_{c=1}^C \Sigma_{\rm \bf N}^c\, , \ \mbox{in which}\ \Sigma_{\rm \bf N}^c:=\{\sigma_j^{d,c} \in \Sigma_{\rm\bf N}\, |\, d\ \mbox{and}\ j\ \mbox{arbitrary}\}\, .$$
We now define
\begin{definition}
Let ${\rm \bf N}$ be a nonhomogeneous network and let $1\leq c\leq C$ be a cell colour.
The $c$-th fundamental network
$\widetilde {\rm \bf N}^c$ of {\bf N} is the nonhomogeneous network with vertex set $\Sigma_{{\rm \bf N}}^c$, 
where vertex $\sigma_j^{d,c}$ is given colour $d$, and input maps $\widetilde \sigma_k^{e,d}: (\Sigma_{{\rm\bf N}}^c)_d\to (\Sigma_{{\rm\bf N}}^c)_e$ defined by
$$\widetilde \sigma_k^{e,d}(\sigma_j^{d,c}):= \sigma_{k}^{e,d}\circ\sigma_j^{d,c}\ \mbox{for}\ 1\leq k\leq m^{e,d} \, .$$ 
So $\widetilde {\rm \bf N}^c$ contains an arrow of colour $k$ from $\sigma_i^{e,c}$ to $\sigma_j^{d,c}$ if and only if $\sigma_i^{e,c}=\sigma_k^{e,d}\circ \sigma_j^{d,c}$. 
\end{definition} 
 \begin{theorem}\label{fundamentaltheoremoid}
 Every nonhomogeneous network ${\rm \bf N}$ is a quotient of its fundamental networks. 
 More precisely, for every vertex $v\in V_c$ of ${\rm \bf N}$, the map of vertices
$$\phi_v: \Sigma_{\rm\bf N}^c \to V \ \mbox{defined by}\ \phi_v(\sigma_j^{d,c}) := \sigma_j^{d,c}(v)$$
extends to a graph fibration from $\widetilde {\rm \bf N}^c$ to ${\rm \bf N}$.
 \end{theorem}
 
\begin{theorem}\label{symmetrytheoremoid}
For all $\sigma_i^{c,b}\in\Sigma_{\rm \bf N}$, the map
$$\phi_{\sigma_i^{c,b}}: \Sigma_{\rm\bf N}^c \to \Sigma_{\rm\bf N}^b\ \mbox{defined by}\ \phi_{\sigma_i^{c,b}}(\sigma_j^{d,c}):=\sigma_j^{d,c}\circ \sigma_i^{c,b}$$
extends to a graph fibration from $\widetilde {\rm \bf N}^c$ to $\widetilde {\rm \bf N}^b$. Hence the semigroupoid-equivariance
$$\phi_{\sigma_i^{c,b}}^* \circ  \gamma_f^{\widetilde {\rm\bf N}^b}= \gamma_f^{\widetilde {\rm\bf N}^c} \circ \phi_{\sigma_i^{c,b}}^*\ \mbox{for all}\ \sigma_i^{c,b}\in \Sigma_{\rm \bf N}\, .$$
\end{theorem}

\begin{theorem}\label{symmetrysynchronyoid}
For each $1\leq c\leq C$, let $P^{c}$ be a balanced partition of the cells  $\Sigma_{\rm\bf N}^c$ of $\widetilde {\rm\bf N}^c$. 
Moreover, let $\gamma^c: E_{\widetilde {\rm \bf N}^c} \to E_{\widetilde {\rm \bf N}^c}$ (for $1\leq c\leq C$) be a set of $\Sigma_{\rm \bf N}$-equivariant maps, that is
$$\phi_{\sigma_i^{c,b}}^* \circ  \gamma^b = \gamma^c \circ \phi_{\sigma_i^{c,b}}^*\ \mbox{for all}\ \sigma_i^{c,b}\in \Sigma_{\rm \bf N}\, .$$
Then $\gamma^c({\rm Syn}_{P^c})\subset {\rm Syn}_{P^c}$ for all $1\leq c\leq C$.
\end{theorem}
\begin{example}
Consider the ordinary differential equations
\begin{align}
&\dot x_{v_1} = f^{(1)}(x_{v_1}, x_{v_2}, x_{v_3}) \nonumber \\ \nonumber
&\dot x_{v_2} =f^{(1)}(x_{v_2}, x_{v_2}, x_{v_3})\nonumber \\ \nonumber
&\dot x_{v_3}= f^{(2)}(x_{v_2}, x_{v_3}) \nonumber
\end{align}
that have the structure of a nonhomogeneous network ${\rm\bf N}$ in which cells $v_1$ and $v_2$ have colour $1$ and cell $v_3$ has colour $2$. The input maps of ${\rm \bf N}$ are given by
$$
 \begin{array}{c|ccc} {\rm \bf N} & v_1 & v_2 & v_3 \\ \hline 
\sigma_1^{1,1}  & v_1 & v_2 & *   \\
\sigma_2^{1,1} &  v_2 & v_2 & *  \\
\sigma_1^{2,1}  & v_3 & v_3 & * \\
\sigma_1^{1,2} & * & * &v_2\\
\sigma_1^{2,2} & * & * & v_3
\end{array}  \ .
$$
One may check that these maps are closed under composition and hence constitute a semigroupoid $\Sigma_{\rm\bf N}$. The composition table of $\Sigma_{\rm\bf N}$ reads
$$
\begin{array}{c|ccccc} \Sigma_{\rm \bf N} & \sigma_1^{1,1} & \sigma_2^{1,1} & \sigma_1^{1,2} & \sigma_1^{2,1} & \sigma_1^{2,2} \\ \hline 
\sigma_1^{1,1} & \sigma_1^{1,1} & \sigma_2^{1,1}  & \sigma_1^{1,2} & * & * \\
\sigma_2^{1,1} & \sigma_2^{1,1} & \sigma_2^{1,1} & \sigma_1^{1,2} & * & * \\
\sigma_1^{1,2} & * & * & * & \sigma_2^{1,1}  & \sigma_1^{1,2} \\
\sigma_1^{2,1} & \sigma_1^{2,1} & \sigma_1^{2,1} & \sigma_1^{2,2} & *  & * \\
\sigma_1^{2,2} & * & * & * & \sigma_1^{2,1} &  \sigma_1^{2,2}
\end{array} \ .
$$
It follows that the two fundamental networks of ${\rm\bf N}$ have the equations of motion
\begin{center}
 \begin{tabular}{p{8cm}}\vspace{-.7cm}
 \begin{equation} \nonumber
 \mbox{Network}\ \widetilde {\rm\bf N}^1:  \left\{ 
\begin{array}{ll}
\dot x_{\sigma_1^{1,1}} &= f^{(1)}(x_{\sigma_1^{1,1}}, x_{\sigma_2^{1,1}}, x_{\sigma_1^{2,1}})   \\
\dot x_{\sigma_2^{1,1}} &=f^{(1)}(x_{\sigma_2^{1,1}}, x_{\sigma_2^{1,1}}, x_{\sigma_1^{2,1}})    \\
\dot x_{\sigma_1^{2,1}} &=f^{(2)}(x_{\sigma_2^{1,1}}, x_{\sigma_1^{2,1}})                
\end{array}\right.
\end{equation}
\\  \vspace{-1cm}  
   \begin{equation} \nonumber
  \mbox{Network}\ \widetilde {\rm\bf N}^2:  \left\{
 \begin{array}{ll}
 \dot x_{\sigma_1^{1,2}} &= f^{(1)}(x_{\sigma_1^{1,2}}, x_{\sigma_1^{1,2}}, x_{\sigma_1^{2,2}} )\\
 \dot x_{\sigma_1^{2,2}} &= f^{(2)}(x_{\sigma_1^{1,2}}, x_{\sigma_1^{2,2}})
\end{array}\right.
\end{equation}
\end{tabular}
\end{center}\vspace{-.3cm}
It is clear that network $\widetilde {\rm\bf N}^1$ is isomorphic to the original network ${\rm \bf N}$. The semigroupoid of symmetries of the fundamental networks consists of the maps
 \begin{equation} \nonumber
\begin{array}{ll}               
\phi_{\sigma_1^{1,1}}^*(x_{\sigma_1^{1,1}}, x_{\sigma_2^{1,1}}, x_{\sigma_1^{2,1}})  & = (x_{\sigma_1^{1,1}}, x_{\sigma_2^{1,1}}, x_{\sigma_1^{2,1}})\\
\phi_{\sigma_2^{1,1}}^*(x_{\sigma_1^{1,1}}, x_{\sigma_2^{1,1}}, x_{\sigma_1^{2,1}})  & = (x_{\sigma_2^{1,1}}, x_{\sigma_2^{1,1}}, x_{\sigma_1^{2,1}})\\
\phi_{\sigma_1^{1,2}}^*(x_{\sigma_1^{1,2}}, x_{\sigma_1^{2,2}})  & = (x_{\sigma_1^{1,2}}, x_{\sigma_1^{1,2}}, x_{\sigma_1^{2,2}})\\
\phi_{\sigma_1^{2,1}}^*(x_{\sigma_1^{1,1}}, x_{\sigma_2^{1,1}}, x_{\sigma_1^{2,1}})  & = (x_{\sigma_2^{1,1}}, x_{\sigma_1^{2,1}})\\
\phi_{\sigma_1^{2,2}}^*(x_{\sigma_1^{1,2}}, x_{\sigma_1^{2,2}})  & = (x_{\sigma_1^{1,2}}, x_{\sigma_1^{2,2}})
\end{array}
\end{equation}
It is not hard to check that these maps send solutions to solutions. The robust synchrony space $\{x_{\sigma_1^{1,1}}=x_{\sigma_2^{1,1}}\}$ inside the phase space of network $\widetilde {\rm \bf N}^1$ is equal to the image of $\phi^*_{\sigma_2^{1,1}}$ and to the image of $\phi^*_{\sigma_1^{1,2}}$. This explains the existence of this synchrony space from (hidden) semigroupoid-symmetry.
 \end{example}
   \bibliography{CoupledNetworks}
\bibliographystyle{amsplain}

  \end{document}